\documentclass[11pt, twoside]{amsart}
\usepackage[cp1251]{inputenc}
\usepackage{color}
\usepackage{amsthm,amssymb,latexsym}
\usepackage[centertags]{amsmath}
\usepackage[T2A]{fontenc}
\usepackage{graphs}
\usepackage{MnSymbol}
\usepackage{bbm}
\usepackage{centernot}
\usepackage{mathrsfs}
\usepackage{amsthm}
\usepackage{amsfonts}
\usepackage{textcomp}
\usepackage{newlfont}
\usepackage{enumitem}

\newif\ifPDF
\ifx\pdfoutput\undefined\PDFfalse
\else \ifnum \pdfoutput > 0 \PDFtrue
        \else \PDFfalse
        \fi
\fi

\ifPDF
  \usepackage[pdftex]{xcolor, graphicx}
  \usepackage[colorlinks=true,linkcolor=blue, citecolor=blue]{hyperref}%

\else
  \usepackage{color}
  \usepackage[dvips]{graphicx}
  \usepackage[dvips]{hyperref}
\fi

\usepackage{bbm}

\usepackage[top=1in, bottom=1.25in, left=1.25in, right=1.25in]
{geometry}
\theoremstyle{definition}
\newtheorem{definition}{Definition}[section]
\theoremstyle{remark}
\newtheorem{example}[definition]{Example}

\newtheorem{remark}[definition]{Remark}

\theoremstyle{plain}
\newtheorem{thm}[definition]{Theorem}
\newtheorem{theorem}[definition]{Theorem}
\newtheorem{prop}[definition]{Proposition}
\newtheorem{lemma}[definition]{Lemma}

\newtheorem{corol}[definition]{Corollary}

\newcommand{\tcb}{\textcolor{blue}}

\def\wt{\widetilde}
\def\ol{\overline}
\def\ov{\overline}
\def\tl{\widetilde}
\def\wh{\widehat}

\newcommand{\N}{\mathbb N}
\newcommand{\Z}{\mathbb Z}

\newcommand{\B}{\mathcal B}

\newcommand{\ent}{f^{(n)}_{vw}}

\newcommand{\be}{\begin{equation}}
\newcommand{\ee}{\end{equation}}
\newcommand{\ba}{\begin{aligned}}
\newcommand{\ea}{\end{aligned}}
\newcommand{\mc}{\mathcal}

\newcommand{\A}{{\cal A}}

\newcommand{\R}{{\mathbb R}}

\numberwithin{equation}{section}
\newcommand{\ignore}[1]{}

\begin{document}

\title{Bratteli diagrams in Borel dynamics}

\author{Sergey Bezuglyi}
%    Address of record for the research reported here
\address{Department of Mathematics, University of Iowa, Iowa City, USA}
\email{sergii-bezuglyi@uiowa.edu}
\author{Palle E.T. Jorgensen}
%    Address of record for the research reported here
\address{Department of Mathematics, University of Iowa, Iowa City, USA}
\email{palle-jorgensen@uiowa.edu}
\author{Olena Karpel}
%    Address of record for the research reported here
\address{Faculty of Applied Mathematics, AGH University of Krakow, Krakow, Poland \&
B. Verkin Institute for Low Temperature Physics and Engineering,
Kharkiv, Ukraine}

\email{okarpel@agh.edu.pl}
\author{Shrey Sanadhya}
%    Address of record for the research reported here
\address{Department of Mathematics, Ben-Gurion University of the Negev, Beersheba, Israel}
\email{sanadhya@post.bgu.ac.il}

\subjclass[2020]{37A05, 37B05, 37A40, 54H05, 05C60}

\keywords{Borel dynamical systems, Bratteli-Vershik model, tail-invariant measures, infinite ergodic theory.}

\date{}

\begin{abstract} Bratteli-Vershik models have been very successfully applied to the study of various
dynamical systems, in particular, in Cantor dynamics. In this paper, we 
study dynamics on the path spaces of \textit{generalized} Bratteli 
diagrams that form  models for non-compact Borel dynamical systems. 
Generalized Bratteli diagrams have countably infinite many vertices 
at each level, thus the corresponding incidence matrices are also 
countably infinite. We emphasize differences (and similarities) 
between generalized and classical Bratteli diagrams. 

Our main results: $(i)$ We utilize Perron-Frobenius theory for 
countably infinite matrices to establish criteria for 
the existence and uniqueness of tail-invariant path space measures 
(both probability and $\sigma$-finite). $(ii)$ We provide criteria 
for the topological transitivity of the tail equivalence relation. 
$(iii)$ We describe classes of stationary generalized Bratteli 
diagrams (hence 
Borel dynamical systems) that: $(a)$ do not support a probability 
tail-invariant measure, $(b)$ are not uniquely ergodic with respect to the tail 
equivalence relation. $(iv)$ We describe classes of generalized Bratteli diagrams which can or cannot admit a continuous Vershik map and construct a Vershik map which is a minimal homeomorphism of a (non locally compact) Polish space.
$(v)$ We 
provide an application of the 
theory of stochastic matrices to analyze diagrams with positive recurrent incidence matrices. 
\end{abstract}

\maketitle
\tableofcontents

\section{Introduction}\label{intro}
This paper is dedicated to the study of discrete dynamical 
systems realized on the path space of \textit{generalized 
Bratteli diagrams}. A generalized Bratteli diagram is a 
natural 
extension of the notion of classical (standard)
Bratteli diagrams where each level has a countably infinite 
set of vertices. The structure of such diagrams 
is determined by a sequence of countably infinite incidence
matrices.

%%%Palle's text
In general terms, a Bratteli diagram is a certain 
combinatorial structure which encompasses the following: 
The diagram takes the form of a graph which is represented 
as a countable union of levels, and a specification of 
edges between levels. The level count in turn is indexed 
or labeled by non-negative integers and edges link 
only levels with index count differing by one. For the last 
decades, such diagrams (graphs) have come to serve as 
a powerful tool for the analysis of path space, and 
representing infinite paths via the particular graphs under 
consideration. The countable index of the levels may represent discrete time in associated models of dynamics. 
Our aim here is to extend earlier results on dynamics via diagrams, and associated path space constructions, to 
the broader context of Borel dynamics, and then to study the corresponding measures on these generalized path spaces. 
Our analysis will entail an extended Perron-Frobenius 
theory, a tail equivalence relation, and the construction 
of corresponding tail-invariant measures, and finally a 
study of generalized Vershik maps.

This wide framework of path space analysis is motivated by a 
variety of applications. In this paper, we identify new 
properties 
of stationary and non-stationary generalized Bratteli diagrams. 
Our main results contribute to the following five important 
questions: $(i)$ identify the structure and the properties of 
 Vershik maps and $(ii)$  the tail equivalence relation 
 associated with generalized Bratteli diagrams; $(iii)$ existence 
 and uniqueness of tail-invariant measures, finite and 
 $\sigma$-finite.
In this context, we present an analysis of $(iv)$ Bratteli 
diagrams with positive recurrent incidence matrices; and $(v)$ 
path space measures induced by stochastic matrices. These 
and other results can be found in  Theorems 
\ref{thm_VM not cont}, \ref{Thm:ContVM} \ref{thm_top_trans},
\ref{Thm:Unique finite}, \ref{Thm:Unique sigma finite},
\ref{thm:H grow} and \ref{Thm:seq_wh_mu_main}.

The existing literature on Bratteli diagrams, corresponding dynamical 
systems, invariant path-space measures, and other areas used in the paper
is very extensive. We give the references below 
 discussing the
most important ingredients of our work. The reader 
can find the main ideas in \cite{Bratteli1972}, 
\cite{HermanPutnamSkau1992}, \cite{GiordanoPutnamSkau1995},
\cite{Kechris1995}, \cite{Nadkarni1991}, \cite{Kitchens1998},
\cite{Vershik_1982}, and other fundamental works cited below.

Our results may be interesting for mathematicians studying 
infinite matrices as well as for experts in Markov chains and 
random walks on a countable set. If all incidence matrices of a 
generalized Bratteli diagram are 0-1 matrices, then the path 
space of such a diagram is a well-known object in the theory 
of random walks on a countable set. It is worth noting that our
focus is on the study of the tail equivalence relation which 
gives the principal different kind of dynamics on the path space. 
Sections \ref{Section:stat_GBD} and 
\ref{sect:stochastic} contain several key theorems and 
examples about infinite matrices, their eigenvectors, and 
corresponding invariant measures. 
\vskip 0.3cm

\textit{Bratteli diagrams, incidence matrices, path spaces.} 
We recall that 
Bratteli diagrams are infinite-graded graphs that were  
named after Ola Bratteli who introduced them in his pioneering 
paper \cite{Bratteli1972} on the classification of approximately 
finite $C^*$-algebras. In short, a Bratteli diagram is 
a countable graph $G= (V, E)$ where vertices $V = \bigcup_n V_n$
and edges $E = \bigcup_n E_n$ are divided into disjoint 
finite sets (levels) $V_n$ and $E_n$. 
The edges from the set $E_n$ exist only for some vertices from 
consecutive levels $V_n$ and $V_{n+1}$. This set of edges $E_n$
defines a $|V_{n+1}| \times |V_n|$ matrix $F_n$ called the 
incidence matrix. Every $F_n$ has non-negative integer entries.
In this paper, we consider the generalized 
Bratteli diagrams satisfying the property of finite row sums
for every incidence matrix $F_n$. This means that every row has
finitely many non-zero entries. For a Bratteli diagram $B$, the 
path space $X_B$ is formed by infinite sequences $(e_i)$ 
of edges such that $e_{i+1}$ begins at the vertex where 
$e_i$ ends. 
\vskip 0.3cm

\textit{Bratteli diagrams and dynamical systems.}
For the last decades, Bratteli diagrams 
have been intensively studied and used in various areas of
dynamics. The trend to use discrete 
structures, such as graphs and sequences of partitions, proved to 
be a very powerful tool in the theory of 
dynamical systems, see e.g. \cite{GambaudoMartens2006} and the 
papers cited below. In the 1970s, Krieger and Vershik applied 
sequences of refining partitions to the study of ergodic 
automorphisms of a measure space 
\cite{Vershik1973}, 
\cite{Vershik_1981}, \cite{Vershik_1982}, \cite{Krieger1976},
\cite{ConnesKrieger1977}. In particular, Vershik proved that 
any ergodic automorphism of a measure space can be realized as
a transformation acting on the path space indexed by a sequence 
of refining partitions, a prototype of a Bratteli diagram. 

At the beginning of the 1990s, these ideas found their new 
applications in Cantor dynamics. Putnam \cite{Putnam1989} 
showed that, for every minimal homeomorphism $\varphi$ 
of a Cantor set $X$, 
there exists a  sequence of refining partitions into clopen sets 
that approximates the orbits of $\varphi$ and the topology on
$X$. Following this article, Herman-Putnam-Skau \cite{HermanPutnamSkau1992} proved that every 
minimal homeomorphism of a Cantor set can be realized as a
homeomorphism $\varphi_B$ (called a Vershik map) 
of a path space $X_B$ of a Bratteli diagram $B$. In other words, 
Bratteli diagrams represent models of minimal Cantor dynamical 
systems. 
This remarkable result led to a breakthrough in Cantor dynamics
based on the works by Giordano-Putnam-Skau   
\cite{GiordanoPutnamSkau1995}), Glasner-Weiss 
\cite{GlasnerWeiss_1995}, and others.  
It was proved that minimal homeomorphisms could be 
completely classified with respect to orbit equivalence. 
The articles by Forrest \cite{Forrest1997} and 
Durand-Host-Skau \cite{Durand_Host_skau_1999} answered the question about the role of 
stationary Bratteli diagrams: they represent substitution
dynamical systems.
Further applications of Bratteli diagrams in Cantor dynamics
extended the class of homeomorphisms that can be realized as 
Vershik maps. Medynets \cite{Medynets_2006} proved this fact for 
aperiodic homeomorphisms of a Cantor set (see also 
\cite{BezuglyiDooleyMedynets2005}). Recently, 
the papers by Downarowicz-Karpel \cite{DownarowiczKarpel_2019}
and Shimomura \cite{Shimomura2020}, \cite{Shimomura2020_1} 
showed that any homeomorphism
of a Cantor set could be represented as a Vershik map on a 
Bratteli diagram. There are several important classes of Bratteli 
diagrams that deserve special attention. 
Stationary non-simple Bratteli diagrams give models for aperiodic 
substitution systems \cite{Bezuglyi_Kwiatkowski_Medynets_2009}).
Finite rank Bratteli diagrams (i.e., $|V_n|$ is bounded), which, 
in particular, represent interval exchange transformations, 
were studied in \cite{BezuglyiKwiatkowskiMedynetsSolomyak2013}.
Finite and infinite invariant measures were the focus of the 
papers \cite{BezuglyiKarpelKwiatkowski2019}, 
\cite{AdamskaBezuglyiKarpelKwiatkowski2017}. 
Eigenvalues of Cantor minimal systems were considered  
in a series of papers, see e.g. 
\cite{DurandFrankMaass_2019} for references.
We do not discuss all the interesting applications here. 
The reader can find more results about invariant measures, 
dynamics, and applications in \cite{AminiElliottGolestani2021},
\cite{DurandPerrin2022}, 
\cite{GiordanoGoncalvesStarling2017},
\cite{GiordanoMatuiPutnamSkau2010}, \cite{DownarowiczMaass2008}, 
\cite{GjerdeJohansen2000}, 
\cite{Putnam2018},  \cite{Trevinio2018}. 
See also  the surveys
\cite{Durand2010}, \cite{BezuglyiKarpel2016}, \cite{DownarowiczKarpel2018},
\cite{BezuglyiKarpel_2020} and the literature mentioned there. 
The reader who is interested in the classification of stationary 
Bratteli diagrams, various links to operator algebras, and 
$K$-theory can find more information in 
\cite{Bratteli1972}, \cite{Effros1981}, 
\cite{EffrosHandelmanShen1980}, 
\cite{BratteliJorgensenKimRoush2000},
\cite{BratteliJorgensenKimRoush2001},
\cite{BratteliJorgensenKimRoush2002}.
\vskip 0.3cm

\textit{Generalized Bratteli diagrams vs standard Bratteli 
diagrams.} 
Why do we need generalized Bratteli diagrams with countable 
levels? One obvious reason to study such diagrams is explained by
the following result. Bezuglyi-Dooley-Kwiatkowski  
\cite{BezuglyiDooleyKwiatkowski_2006} proved that every aperiodic Borel automorphism of an uncountable standard Borel space admits a realization as a Vershik map on the 
path space of a \textit{generalized Bratteli diagram}. 
A recent result in this direction was obtained in 
\cite{Bezuglyi_Jorgensen_Sanadhya_2022} where the authors proved 
that there is a wide class of substitution dynamical systems on infinite alphabets that 
can be realized  
as Vershik maps on stationary generalized Bratteli 
diagrams. We also refer our readers to related recent works \cite{Manibo_Rust_Walton_2022}, \cite{Frettloh_Garber_Manibo_2022}.
Among other possible applications of generalized 
Bratteli diagrams, we can mention Markov chains, random walks, iterated function systems
\cite{https://doi.org/10.48550/arxiv.2210.14059}, 
harmonic analysis on the path space
of generalized Bratteli diagrams \cite{Bezuglyi_Jorgensen_2021}, 
etc. 

It is clear that the variety of classes of generalized 
Bratteli diagrams is much wider than that of standard 
Bratteli diagrams. 
This fact suggests the possibility of now establishing key 
results 
in the wider context of all generalized Bratteli diagrams.
Another obvious observation is that 
the case of countably infinite matrices requires new techniques 
and methods in comparison with 
finite matrices. In this paper, we consider mostly two classes
of diagrams: stationary diagrams (when $F_n =F$) and 
bounded size diagrams (when all $F_n$ are banded matrices with 
bounded row sums), see Section 
\ref{sect Basic} for definitions. These diagrams have 
the predictable behavior of infinite paths and, therefore, 
we can better 
understand their dynamical properties. It is an interesting 
problem to find out how much the structure of a generalized 
Bratteli diagram determines dynamics and invariant measures 
on the path space of a diagram.

The main results of this paper are concentrated on 
principal problems in dynamics on the path spaces of Bratteli diagrams.
We consider the existence and uniqueness of measures invariant 
with respect to the tail equivalence relation and Vershik maps.
We discuss the dynamical properties of the Vershik map, and how 
the structure of a Bratteli diagram affects the dynamics on the 
path space. The notions of isomorphic and order isomorphic 
generalized Bratteli diagrams are considered in this paper. 
\vskip 0.3cm

\textit{Extended Perron-Frobenius theory, the tail equivalence 
relation, and Vershik map}.
Our main emphasis in the current work is to point out the 
differences and similarities between the dynamics on path space 
of the classical Bratteli diagrams and the generalized ones. 
The tools used in this work are also significantly different. 
For example, considering stationary Bratteli diagrams, we work 
with Perron-Frobenius eigenpairs. For classical diagrams, the
Perron-Frobenius theory covers all possible cases of stationary 
Bratteli diagrams. Moreover, using eigenvectors and eigenvalues
of non-negative matrices, one can explicitly describe all finite 
and $\sigma$-finite invariant measures for simple and non-simple 
stationary Bratteli diagrams 
\cite{BezuglyiKwiatkowskiMedynetsSolomyak2010}.  
For infinite non-negative matrices, the Perron-Frobenius theory
does not cover all cases of generalized Bratteli diagrams: 
there are stationary  diagrams with incidence matrices
that do not have a finite Perron eigenvalue. 
Another important circumstance is that the cases of recurrent 
and transient incidence matrices lead to essentially different 
results regarding the uniqueness of ergodic invariant measures.
This means that one needs to use different techniques 
for finding invariant measures. We use the book by Kitchens
\cite{Kitchens1998} for references about the definitions 
and main results of the Perron-Frobenius theory. For the 
reader's convenience, we included in Appendix 
\ref{APP:Perron-Frobenius_Theory} the facts that are used in this
paper.

One more essential distinction between standard and generalized Bratteli diagrams consists of the existence of invariant measures
on the path spaces. For a standard diagram $B$, the path space $X_B$ is a compact Cantor set, and every 
Vershik map $\varphi_B$ is a homeomorphism of $X_B$. 
Therefore, the classical Bogoliubov-Krylov theorem guarantees the 
existence of a probability $\varphi_B$-invariant measure on 
$X_B$. By contrast, for a generalized
Bratteli diagram, $X_B$ is a Polish zero-dimensional
space, and $(X_B, \varphi_B)$
is a Borel dynamical system. There are then generalized Bratteli
diagrams that do not support finite invariant measures. The 
following question is natural: for what classes of 
generalized Bratteli diagrams are there finite invariant 
measures? 
We note that minimal Cantor dynamics deals with probability 
invariant measures only. The settings for generalized Bratteli 
diagrams lead to the study of both finite and infinite invariant 
measures. We construct and describe classes of generalized Bratteli 
diagrams that do not support a probability invariant measure. 

On every Bratteli diagram $B$, we can consider dynamical systems
of two kinds: the tail equivalence relation $\mathcal R$, and 
a Borel dynamical system $(X_B, \varphi_B)$ defined by a Vershik
map $\varphi_B$. Then $\mathcal R$ is a countable Borel 
equivalence relation which is completely defined by the diagram.
To define a Vershik map $\varphi_B$, we need to consider a 
partial ordering on the set of all edges $E$. The question 
about the continuity of a Vershik map was studied 
in Cantor dynamics \cite{Medynets_2006}, 
\cite{BezuglyiKwiatkowskiYassawi2014}, 
\cite{BezuglyiYassawi2017}, \cite{JanssenQuasYassawi2017}. The result of Bezuglyi-Dooley-Kwiatkowski  
\cite{BezuglyiDooleyKwiatkowski_2006} shows that every aperiodic Borel automorphism of an uncountable standard Borel space admits a realization as a Vershik map on the 
path space of an ordered generalized Bratteli diagram such that the Vershik map is a homeomorphism. 
The left-to-right ordering on a simple (standard) Bratteli 
diagram always gives rise to a continuous Vershik map. But this 
is not the case for generalized Bratteli diagrams,  
even for diagrams with reasonable simple structure, see 
Section \ref{sect Bndd size}. In other words, there are irreducible
generalized Bratteli diagrams such that the left-to-right order
does not generate a continuous Vershik map. Moreover, in the class of generalized Bratteli diagrams with a unique infinite minimal path and a unique infinite maximal path one can find diagrams for which both Vershik map and its inverse are discontinuous, or Vershik map is continuous but its inverse is not (see Section \ref{Sect:contin_V_map}). 

Infinite matrices, especially banded matrices, is the subject of 
great interest because of their applications in various 
areas of mathematics and mathematical physics. In this 
context we mention \cite{AvniBreuerSimon2020}, 
\cite{ChristiansenSimonZinchenko2012}, 
\cite{ChristiansenSimonZinchenko2013}.
\vskip 0.3cm

 \textit{The outline of the paper and main results}. 
In Section \ref{sect Basic} we provide basic definitions and 
discuss the properties of generalized Bratteli diagrams. 
We consider such notions as tail equivalence relation, 
Vershik map, isomorphism of Bratteli diagrams. 
In Section \ref{sect Bndd size} we focus mostly on the 
notion of 
\textit{bounded size} generalized Bratteli diagrams (see 
Definition \ref{Def:BD_bdd_size}) and show that they form 
a natural class of diagrams that can be viewed as 
intermediate 
between classical (standard) and generalized Bratteli diagrams. 
We show that even diagrams with this natural structure can 
provide interesting examples that contrast the classical case. We find classes of bounded size generalized Bratteli diagrams such 
that, for the left-to-right ordering, 
the Vershik map is discontinuous (unlike the case of standard Bratteli diagrams), moreover, every infinite maximal path is a point of discontinuity. We also provide some classes of bounded size diagrams for which it is possible to prolong the Vershik map to a homeomorphism. 
%\tcb{The result in \cite{BezuglyiDooleyKwiatkowski_2006} states that that any aperiodic Borel automorphism of a standard Borel space can be realized as a Vershik map acting on a path space of an ordered generalized Bratteli diagram which has no infinite minimal and no infinite maximal paths, in particular, the Vershik map is a homeomorphism.}
In Section \ref{Sect:contin_V_map}, we consider 
arbitrary generalized Bratteli diagrams and find conditions 
under which they have an order such that the Vershik map can 
be prolonged to a homeomorphism. 
In particular, we are interested in the orders such that the diagram does not possess infinite minimal and infinite maximal paths. We give sufficient conditions for a generalized Bratteli diagram to have such an order. 
We give an example of an ordered generalized Bratteli diagram with a non locally compact path space such that the corresponding Vershik map is a minimal homeomorphism.
In contrast to the case of standard Bratteli diagrams, there are examples of generalized Bratteli diagrams with 
a unique minimal and a unique maximal path such that 
the corresponding Vershik map cannot be prolonged to a 
homeomorphism. 
We show that in the class of ordered generalized Bratteli diagrams with a unique infinite minimal path and a unique infinite maximal path one can find examples of diagrams such that (i) both the Vershik map $\varphi_B$ and its inverse  $\varphi_B^{-1}$ are not continuous;
(ii) the Vershik map $\varphi_B$ is continuous but 
the inverse $\varphi_B^{-1}$ is discontinuous;
(iii) both the Vershik map $\varphi_B$ and its inverse  $\varphi_B^{-1}$ are continuous.
Section \ref{sec: topological property} discusses some 
topological properties of generalized Bratteli diagrams.
 It is proved that, for an irreducible stationary 
generalized Bratteli diagram, the tail equivalence relation 
is topologically transitive. The case of non-stationary 
diagrams 
is also considered. It is proved that irreducible bounded size generalized 
Bratteli diagrams do not generate minimal tail equivalence 
relations. In Section \ref{Sec:Tail_inv_mes} we discuss the tail 
invariant measures 
on the path space of generalized Bratteli diagrams. We provide 
explicit examples of non-stationary generalized Bratteli diagrams 
that do not support full probability measures invariant with 
respect to the Vershik map, and examples of stationary 
generalized Bratteli diagrams which do not support finite 
tail-invariant measures, i.e., the tail equivalence relations for 
such diagrams are compressible. In Section 
\ref{Section:stat_GBD}, we study 
finite and $\sigma$-finite invariant measures on stationary 
generalized Bratteli diagrams. 
We give an explicit description of these measures and prove their 
uniqueness. Several examples that illustrate these theorems 
are also given in Section \ref{Section:stat_GBD}. The detailed 
calculations and proofs for these examples are provided in Appendix 
\ref{APP:Example}. 
Section \ref{sect:stochastic} deals with applications of 
stochastic matrices, related to stationary and non-stationary 
generalized Bratteli diagrams, to the problem of the existence of invariant measures. We give examples of null-recurrent incidence matrices such that the invariant measure for the corresponding generalized Bratteli diagrams is not unique. We also show that, for positive recurrent matrices, one
can control the growth of the heights of the 
Kakutani-Rokhlin towers in terms of the corresponding eigenvectors. 
In Section~\ref{Sect:OP}, we present further possible directions of research.
In Appendix \ref{APP:Perron-Frobenius_Theory}, we provide a 
brief  description of the Perron-Frobenius theory for 
infinite matrices.

%%%%% Section 2
\section{Generalized Bratteli diagrams: basic definitions and 
facts}\label{sect Basic}
%%%%%

This section contains the main definitions of the objects we 
consider
in the paper. We discuss generalized Bratteli diagrams and 
their 
subclasses, the corresponding infinite matrices, 
the tail equivalence relation on the path space of a
diagram, finite and $\sigma$-finite measures invariant with 
respect to tail equivalence relation, and orderings on the 
set of edges and corresponding Vershik maps.

\subsection{Definition of generalized Bratteli diagrams}
The notion of a generalized Bratteli diagram is a natural 
extension of the notion 
of a Bratteli diagram to the case when 
all levels in a generalized Bratteli diagram are countable. 
It was proved in \cite{BezuglyiDooleyKwiatkowski_2006}
that any aperiodic Borel automorphism can be realized as 
a Vershik map on the path space of a generalized Bratteli 
diagram, see Theorem \ref{Thm:GBD_models_Borel_dyn} below.

We will use the standard notation $\mathbb N, \Z, \R,$  
$\mathbb{N}_0 = \N \cup \{0\}$ for the sets of numbers, and 
$|\cdot |$ denotes the cardinality of a set.

\begin{definition}\label{Def:generalized_BD} A 
\textit{generalized Bratteli diagram} is a graded graph 
$B = (V, E)$ such that the 
vertex set $V$ and the edge set $E$ can be represented as 
partitions $V = \bigsqcup_{i=0}^\infty  V_i$ and $E = 
\bigsqcup_{i=0}^\infty  E_i$ satisfying the following 
properties: 
\vspace{2mm}

\noindent $(i)$ The number of vertices at each level 
$V_i$, $i 
\in \N_0$ is countably infinite (in most cases, we will 
identify each 
$V_i$ with $\Z$ or $\N$). The set $V_i$ is called the $i$th 
level of the diagram $B$. For all $i \in \N_0$, the set $E_i$
of all edges between $V_i$ and $V_{i+1}$ is countable.
\vspace{2mm}

\noindent $(ii)$ For every edge $e\in E$, we define the range and 
source maps $r$ and $s$ such that $r(E_i) = V_{i+1}$ and 
$s(E_i) = V_{i}$ for $i \in \N_0$. In particular, we have 
$s^{-1}(v)\neq \emptyset $ for all $v\in V$, and  
$r^{-1}(v)\neq\emptyset$ for all $v \in V\setminus V_0$. 

\vspace{2mm}

\noindent $(iii)$ For every vertex $v \in V \setminus V_0$, we 
have $|r^{-1}(v)| < \infty$.  
\end{definition} 

\begin{remark} (1) 
If the level $V_0$ consists of a single vertex and each set $V_n$ 
is finite, then we get the usual definition of a Bratteli diagram 
which was first defined in \cite{Bratteli1972}. In particular, 
these diagrams were
used to model Cantor dynamical systems and classify them up to 
orbit equivalence (see \cite{HermanPutnamSkau1992}, 
\cite{GiordanoPutnamSkau1995}). 

(2) Since all levels in a generalized Bratteli diagram are 
countable sets, we can use the integers $\Z$ (or natural 
numbers $\N$) to enumerate the vertices. 
%It is more convenient 
%to use $\Z$ because in this case, we do not have a boundary that
%may change the structure of a diagram.
When we index the vertices 
at each level by $\Z$, the generalized diagram is called 
\textit{two-sided infinite}, and when the vertices are indexed by 
$\N$, we call it \textit{one-sided infinite}.
\end{remark}

To define the path space of a generalized Bratteli diagram,
we consider a 
finite or infinite sequence of edges $(e_i: e_i\in E_i)$ such 
that $s(e_i)=r(e_{i-1})$ which is called a finite or infinite 
path, respectively. Given a generalized Bratteli diagram $B = 
(V, E)$, we denote the set of infinite
paths starting at $V_0$ by $X_B$ and call it the \textit{path 
space}. For a finite path $\ol e = (e_0, ... , e_n)$, we denote 
$s(\ol e) = s(e_0), r(\ol e) = r(e_n)$. The set 
$$
    [\ol e] := \{x = (x_i) \in X_B : x_0 = e_0, ..., x_n = e_n\}, 
$$ 
is called the \textit{cylinder set} associated with $\ol e$. 

The \textit{topology} on the path space $X_B$ is generated by
cylinder sets.
This topology coincides with the topology defined by the 
following metric on $X_B$: for $x = (x_i), \, y = (y_i)$, set 
$$
\mathrm{dist}(x, y) = \frac{1}{2^N},\ \ \ N = \min\{i \in \N_0 : 
x_i \neq y_i\}.
$$
The path space $X_B$ is a zero-dimensional Polish space and 
therefore a standard Borel space.

For a vertex $v \in V_m$ and a vertex $w \in V_{n}$, we will 
denote 
by $E(v, w)$ the set of all finite paths between $v$ and $w$. 
Set $f^{(i)}_{v,w} = |E(v, w)|$ for every $w \in V_i$ and $v \in 
V_{i+1}$. In such a way,  we associate with the generalized 
Bratteli diagram $B = (V,E)$ 
a sequence of non-negative countable infinite matrices $(F_i)$, 
$i \in \N_0$, 
(called the \textit{incidence  matrices}) given by
\begin{equation}\label{Notation:f^i}
    F_i = (f^{(i)}_{v,w} : v \in V_{i+1}, w\in V_i),\ \   
    f^{(i)}_{v,w}  \in \N_0.
\end{equation} 

In this paper, a matrix $F = (f_{ij})$ is called 
\textit{infinite} 
(or countably infinite) if its rows and columns
are indexed by the same countably infinite set. Assuming that 
all matrices $F^n, n \in \N$ are defined (i.e., they have 
finite entries), we denote the entries of $F^n$ by 
$f_{ij}^{(n)}$. 
Observe that this notation is similar to \eqref{Notation:f^i}. 
It will be clear from the context whether $f_{ij}^{(n)}$ denotes 
the $(i,j)$-th entry of the matrix $F_n$ (in a sequence of 
matrices $(F_n)_{n \in \N_0}$) or it denotes the $(i,j)$-th entry 
of the $n$-th power of the matrix $F$.  

\begin{remark}
The structure of a generalized Bratteli diagram $B=(V, E)$ is 
completely determined by the sequence of incidence matrices 
$(F_n), 
\, n \in \N_0$. In this case, we write $B = B(F_n)$. For each $n 
\in \N_0$, the matrix $F_n$ has at most finitely many non-zero 
entries in each row, and none of its rows or columns are entirely 
zero. A column of $F_n$ may have a finite or infinite number of 
non-zero entries. Note that, $X_B$ is locally 
compact if and only if every column of $F_n$ has finitely many 
non-zero entries for all $n$.
\end{remark}

\begin{definition}\label{Def:irreducible_GBD} (1) 
Let $B = B(F_n)$ be
a generalized Bratteli diagram. If  $F_n = F$ for every $n \in 
\N_0$, then the diagram $B$ is called \textit{stationary.} We 
will write $B = B(F)$ in this case.

(2) A generalized Bratteli diagram $B =(V,E)$, where 
all levels $V_i$ are identified with a set $V_0$ (e.g. $V_0 = 
\mathbb{N}$ or $\mathbb{Z}$),  is called  \textit{irreducible} if 
for any vertices $i, j \in V_0$ and any level $V_n$ there exist 
$m 
> n$ and a finite path connecting $i \in V_n$ and $j \in V_m$. In 
other words, the $(j, i)$-entry of the matrix $F_{m-1} \cdots 
F_n$ 
is non-zero. Otherwise, the diagram is called \textit{reducible}.

\end{definition}

In particular, a  generalized stationary Bratteli diagram is
irreducible if and only if the corresponding incidence matrix is 
irreducible (see Appendix ~\ref{APP:Perron-Frobenius_Theory} for 
more definitions and results about infinite positive matrices). 

\begin{definition}\label{Def:Tail_equiv_relation}
Two paths $x= (x_i)$ and $y=(y_i)$ in $X_B$ are called 
\textit{tail equivalent} if there exists an $n \in \mathbb{N}_0$ 
such that $x_i = y_i$ for all $i \geq n$. This notion defines a 
\textit{countable Borel equivalence relation} $\mathcal R$ on the
path space $X_B$ which is called the \textit{tail equivalence 
relation}.
\end{definition}

\begin{remark} The set of generalized Bratteli diagrams contains 
various ``exotic'' examples. In this paper, we will consider the
diagrams whose properties are natural from the point of view 
of dynamical systems. 
We will assume that the path space $X_B$ of a generalized 
Bratteli 
diagram $B$ has no isolated points, i.e., for every infinite  
path 
$(x_0, x_1, x_2, ... ) \in X_B$ and every $n \in \N_0$, there 
exists a level $m > n$ such that $|s^{-1}(r(x_m))| > 1$. Hence, 
the 
set $X_B$ is uncountable. We will consider only such Bratteli 
diagrams for which the tail equivalence relation $\mathcal{R}$ is 
aperiodic. We will not consider cases where the Bratteli diagram 
is 
a disjoint union of two or more Bratteli diagrams, that is, the 
Bratteli diagram should be connected when considered as an 
undirected graph.
\end{remark}

\begin{definition}\label{Def:telescoping} Given a generalized 
Bratteli diagram $B = (V,E)$ and a  monotone increasing sequence 
$(n_k : k \in \N_0), n_0 = 0$, we define a new generalized Bratteli 
diagram $B' = (V', E')$ as follows: the vertex sets are 
determined by $V'_k = V_{n_k} $, and the edge sets  $E'_k = 
E_{n_{k}} \circ ...\circ E_{n_{k+1}-1}$ are formed by finite paths 
between the levels $V'_k$ and 
$V'_{k+1}$. The diagram $B' = (V', E')$  is called a 
\textit{telescoping} of the original diagram $B = (V,E)$. 
\end{definition}

Note that for a  generalized stationary Bratteli diagram, after 
telescoping, we can always assume that the incidence matrix is 
aperiodic (see Appendix \ref{APP:Perron-Frobenius_Theory}).

\subsection{Isomorphism of generalized Bratteli diagrams}\label{Subsec:isomBD}
Similarly to the case of standard Bratteli diagrams (see e.g. 
\cite{Durand2010}), we define the notion of isomorphic 
generalized Bratteli diagrams:

\begin{definition} \label{def_isom BD}
Two generalized Bratteli diagrams $B = (V, E)$ and $B' = 
(V', E')$ 
are called \textit{isomorphic} if there exist two sequence of
bijections $(g_n : V_n \rightarrow V_n')_{n \in \N_0}$ and 
$(h_n : E_n \rightarrow E_n')_{n \in \N_0}$ such that 
for every $n \in \N_0$, we have $g_n(V_n) = V_n'$ and $h_n(E_n) = 
E_n'$, and $s' \circ h_n = g_n \circ s$, $r' \circ h_n = 
g_n \circ r$. 

\end{definition} 

To illustrate Definition \ref{def_isom BD}, we consider 
the following example. 

\begin{example}[Isomorphic generalized Bratteli 
diagrams, Figure~\ref{Fig:Isom_BD}]\label{Ex:isom_bd}
Let $B(F)$ be a stationary generalized Bratteli diagram such that 
every level of $B$ is identified with $\Z$, and the incidence 
matrix $F = (f_{ij})_{i,j \in \Z}$ has entries
$$
f_{ij} = 
\left\{
\begin{aligned}
& 2, \mbox{ for } i = j,\\
& 1, \mbox{ for } |i - j| = 1,\\
& 0, \mbox{ otherwise. }
\end{aligned}
\right.
$$ 
Similarly, let $B'(F')$ be a stationary generalized Bratteli 
diagram such that every level of $B'$ is identified with $\N_0$, 
and its incidence matrix $F' = (f'_{ij})_{i,j \in \N_{0}}$ has 
entries $f'_{00} = f'_{11} = 2$, $f'_{01} = 1, f'_{02} = 1, f'_{10} = 1, f'_{13} = 1, f'_{20} = 1, f'_{31} = 1$,
and for $i,j \notin \{0,1\}$: 
$$
f'_{ij} = 
\left\{
\begin{aligned}
& 2, \mbox{ for } i = j,\\
& 1, \mbox{ for } |i - j| = 2,\\
& 0, \mbox{ otherwise. }
\end{aligned}
\right.
$$ 
The diagrams $B$ and $B'$ (elaborated in 
Figure~\ref{Fig:Isom_BD}) are isomorphic. Indeed, the bijections 
$(g_n : V_n \rightarrow V_n')_{n \in \N_0}$ and $(h_n : E_n 
\rightarrow E_n')_{n \in \N_0}$ that give the isomorphism are 
defined as follows: since the diagrams are stationary, we set,
for every $n \in \N_0$, $g_n = g : \mathbb{Z} \rightarrow 
\mathbb{N}_0$ 
$$
g (n) = 
\left\{
\begin{aligned}
&2n, \mbox{ if } n \geq 0,\\
&- 2n - 1, \mbox{ if } n < 0.
\end{aligned}
\right.
$$ 
The bijection $h_n : E_n \rightarrow E_n'$ is defined as follows: 
The two vertical edges in the diagram $B(F)$ with range $i \in 
V_n$ are mapped to the two vertical edges in the diagram $B'(F')$ 
with range $g(i) \in V_n$. For $i > 0$, the 
non-vertical edge with range $i$ coming from left (respectively 
from the right) is mapped to the non-vertical edge 
with range 
$g(i) \in V'_n$ coming from the left (respectively from the 
right). For $i < 0$, the 
non-vertical edge with range $i$ coming from left (respectively 
from the right) is mapped to the non-vertical edge 
with range 
$g(i) \in V'_n$ coming from the right (respectively from the 
left). For the non-vertical edges with range $0 \in V_n$, 
the mapping  can be seen from the 
Figure~\ref{Fig:Isom_BD}. It is easy to see that with the above 
bijections the two diagrams are isomorphic. 

\begin{figure}
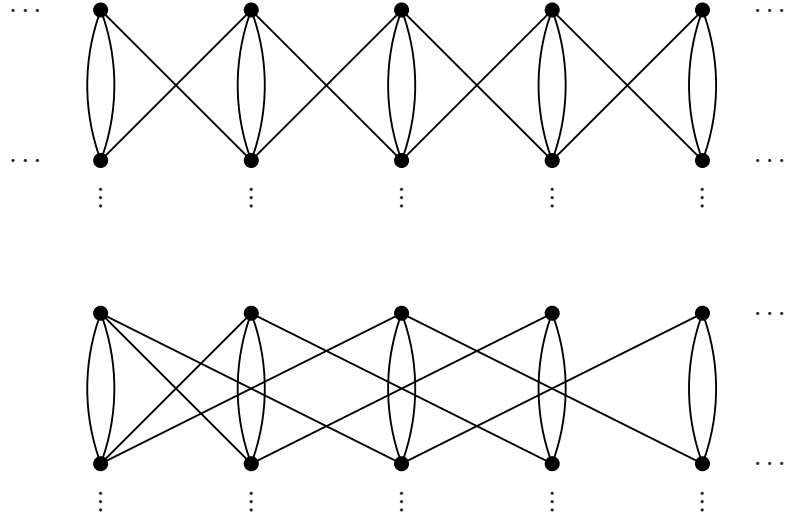

\unitlength=1cm
\begin{graph}(11,4)
% \graphnodesize{0.2}
% \roundnode{V0}(3,6)
%  %\nodetext{V0}(-1,0){$V_0$}
 \roundnode{V11}(2,3)
 %\nodetext{V11}(-0.7,0){$w_1^{(0)}$}
  %\nodetext{V12}(0.7,0){$w_2^{(0)}$}
 \roundnode{V12}(4,3)
 \roundnode{V13}(6,3)
 \roundnode{V14}(8,3)
  \roundnode{V15}(10,3)
  % The second level vertices
 \roundnode{V21}(2,1)
 \roundnode{V22}(4,1)
 \roundnode{V23}(6,1)
  \roundnode{V24}(8,1)
    \roundnode{V25}(10,1)
  %\nodetext{V21}(-0.7,0){$w_1^{(1)}$}
 % \nodetext{V22}(0.7,0){$w_2^{(1)}$}
  %
 %
 % EDGES
 \graphlinewidth{0.025}
% % First level
%  \edge{V0}{V11}
%  \edge{V0}{V12}

 % Second level
 \bow{V21}{V11}{0.09}
  \bow{V21}{V11}{-0.09}
    \edge{V22}{V11}
    \edge{V21}{V12}
     
 \bow{V22}{V12}{0.09}
 \bow{V22}{V12}{-0.09}

 \bow{V23}{V13}{0.09}
  \bow{V23}{V13}{-0.09}

 \edge{V23}{V12}
    \edge{V22}{V13}
    
     \bow{V24}{V14}{0.09}
  \bow{V24}{V14}{-0.09}
  
   \edge{V24}{V13}
    \edge{V23}{V14}
    
         \bow{V25}{V15}{0.09}
  \bow{V25}{V15}{-0.09}
  
   \edge{V25}{V14}
    \edge{V24}{V15}

 \freetext(10.9,3){$\ldots$}
  \freetext(10.9,1){$\ldots$}
   \freetext(1,3){$\ldots$}
  \freetext(1,1){$\ldots$}
  \freetext(2,0.5){$\vdots$}
  \freetext(4,0.5){$\vdots$}
    \freetext(6,0.5){$\vdots$}
      \freetext(8,0.5){$\vdots$}  
        \freetext(10,0.5){$\vdots$} 
    
\end{graph}

\begin{graph}(11,4)
% \graphnodesize{0.2}
% \roundnode{V0}(3,6)
%  %\nodetext{V0}(-1,0){$V_0$}
 \roundnode{V11}(2,3)
 %\nodetext{V11}(-0.7,0){$w_1^{(0)}$}
  %\nodetext{V12}(0.7,0){$w_2^{(0)}$}
 \roundnode{V12}(4,3)
 \roundnode{V13}(6,3)
 \roundnode{V14}(8,3)
  \roundnode{V15}(10,3)
  % The second level vertices
 \roundnode{V21}(2,1)
 \roundnode{V22}(4,1)
 \roundnode{V23}(6,1)
  \roundnode{V24}(8,1)
    \roundnode{V25}(10,1)
  %\nodetext{V21}(-0.7,0){$w_1^{(1)}$}
 % \nodetext{V22}(0.7,0){$w_2^{(1)}$}
  %
 %
 % EDGES
 \graphlinewidth{0.025}
% % First level
%  \edge{V0}{V11}
%  \edge{V0}{V12}

 % Second level
 \bow{V21}{V11}{0.09}
  \bow{V21}{V11}{-0.09}
    \edge{V22}{V11}
    \edge{V21}{V12}
     \edge{V21}{V13}
     
 \bow{V22}{V12}{0.09}
 \bow{V22}{V12}{-0.09}

 \bow{V23}{V13}{0.09}
  \bow{V23}{V13}{-0.09}

 \edge{V23}{V11}
    \edge{V22}{V14}
    
     \bow{V24}{V14}{0.09}
  \bow{V24}{V14}{-0.09}
  
   \edge{V24}{V12}
    \edge{V23}{V15}
    
         \bow{V25}{V15}{0.09}
  \bow{V25}{V15}{-0.09}
  
   \edge{V25}{V13}
    %\edge{V24}{V15}

 \freetext(10.9,3){$\ldots$}
  \freetext(10.9,1){$\ldots$}
  \freetext(2,0.5){$\vdots$}
  \freetext(4,0.5){$\vdots$}
    \freetext(6,0.5){$\vdots$}
      \freetext(8,0.5){$\vdots$}  
        \freetext(10,0.5){$\vdots$} 
    
\end{graph}
\caption{Isomorphic generalized Bratteli diagrams $B$ and 
$B'$.}\label{Fig:Isom_BD}
\end{figure}
\end{example}
\medskip

To define a dynamical system on the path space of a 
generalized 
Bratteli diagram, we need to take a linear order $>$ on each
(finite) set $r^{-1}(v),\ v
\in V\setminus V_0$. This order defines a partial order on 
the sets of edges $E_i,\ i=0,1,...$: edges
$e,e'$ are comparable if and only if $r(e)=r(e')$. For a 
generalized Bratteli diagram $B=(V,E)$ equipped with partial order 
$>$, we define a partial lexicographical order on
the set $E_{k}\circ\cdots\circ E_{l}$ of all finite paths 
from 
$V_k$ to $V_{l+1}$ as follows:
$(e_{k},...,e_{l}) > (f_{k},...,f_{l})$ if and only if $e_i > f_i$ 
for some $i$ with $k\le i\le l$, and $e_j=f_j$ for $i<j\le 
l$. Then any 
two paths from the (finite) set $E(V_0,v)$ of all paths 
connecting a vertex from $V_0$ and $v$ are comparable with 
respect to the lexicographic order. 

\begin{definition}\label{Def:Ordered_GBD} 
A generalized Bratteli diagram $B=(V,E)$
together with a partial order $>$ on $E$ is called 
\textit{an ordered generalized Bratteli diagram} $B=(V,E,>)$. 
\end{definition}

We call an infinite path $e= (e_0,e_1,..., e_i,...)$ 
\textit{maximal (respectively minimal)} if every
$e_i$ has a maximal (respectively minimal) number among all 
elements from $r^{-1}(r(e_i))$. The same definition is used 
for finite maximal/minimal paths. 
Remark that there are unique minimal and maximal 
paths in the set $E(V_0,v)$ of all finite paths arriving at 
$v$ for each $v\in V_i,\ i > 0$.

We note that, in contrast to standard Bratteli diagrams, 
there are orders on generalized Bratteli diagrams such that the sets
$X_{max}$ and $X_{min}$ of maximal and minimal paths are empty,
see \cite{BezuglyiDooleyKwiatkowski_2006} and Example~\ref{Ex:noXminnoXmax}. 
It is not hard to see that $X_{max}$ and $X_{min}$ are 
closed subsets of $X_B$. 

\begin{definition}\label{Def:VershikMap}  For an ordered 
generalized Bratteli diagram $B=(V,E,>)$, we define a Borel 
transformation 
\begin{equation}\label{eq: Vershik map}
\varphi_B : X_B \setminus X_{max} \rightarrow X_B \setminus X_{min}
\end{equation}
as follows.  Given $x = (x_0, x_1,...)\in X_B\setminus X_{max}$, 
let $m$ be the smallest number such that $x_m$ is not maximal. Let 
$g_m$ be the successor of $x_m$ in the finite set $r^{-1}(r(x_m))$.
Then we set $\varphi_B(x)= (g_0, g_1,...,g_{m-1},g_m,x_{m+1},...)$
where $(g_0, g_1,..., g_{m-1})$ is the minimal path in $E(V_0, 
s(g_{m}))$. The map $\varphi_B$ is
a Borel bijection. Moreover, $\varphi_B$ is a homeomorphism from 
$X_B\setminus X_{max}$ onto $X_B\setminus X_{min}$. If  
$\varphi_B$ admits a bijective extension 
to the 
entire path space $X_B$, then we call the Borel transformation 
$\varphi_B : X_B  \rightarrow X_B$ a \textit{Vershik map}, and 
the 
Borel dynamical system $(X_B,\varphi_B)$ is called a generalized 
\textit{Bratteli-Vershik} system.

\end{definition}

\begin{remark} We collect here several facts about Vershik maps 
on generalized Bratteli diagrams. 

(1) Let $B = (V, E, >)$ be an ordered generalized Bratteli diagram.
Relation \eqref{eq: Vershik map} defines $\varphi_B$ uniquely as 
a map from $X_B \setminus X_{max}$ onto  $X_B \setminus X_{min}$. 
We note that if $\varphi_B$ can be prolonged to a Vershik map on 
$X_B$, then, in general, such an extension is not unique. Indeed, if $|X_{min}| = |X_{max}| > 1$ then we can 
choose arbitrary a Borel map from $X_{max}$ onto $X_{min}$ as a Vershik map acting on $X_{max}$.

(2) There exist orders $>$ on $X_B$ such that
both sets $X_{max}$ and $X_{min}$ are empty.
In this case, $\varphi_B$ is uniquely 
determined according to \eqref{eq: Vershik map} of Definition 
\ref{Def:VershikMap}. Also, there orders $>$ on $B$ such that 
$|X_{max}| \neq |X_{min}|$; in particular, one of these sets may
be empty. In the latter case, $B (V,E, >)$ does not support a 
Vershik map.

(3) We note that for simple standard Bratteli diagrams, the 
left-to-right
ordering always generates a Vershik homeomorphism. Indeed, in this 
case, all minimal edges start from the first vertex on each level, 
and all maximal edges start from the last vertex. This guarantees 
the uniqueness of an infinite minimal and infinite maximal path. 
For irreducible generalized Bratteli diagrams the left-to-right 
ordering does not necessarily 
produce a Borel Vershik automorphism (see 
Example~\ref{Ex:LROrdering_no_V.map} below) or a continuous 
Vershik map (see Theorem \ref{thm_VM not cont} below).

\end{remark}

\begin{definition}\label{def_isom BD_1} Two ordered generalized 
Bratteli diagrams $B = (V, E, >)$, $B' = (V', E',>')$ are called 
\textit{order isomorphic} if they are isomorphic (see Definition 
\ref{def_isom BD}) and, for all $n \in \N$, $v \in V_n$ and 
$e_1,e_2 
\in r^{-1}(v)$, we have $e_2>e_1$ if and only if $h_n(e_1)>' 
h_n(e_2)$ for $h_n(e_1), h_n(e_2) \in r^{-1}(g_n(v))$, where $g_n: 
V_n \rightarrow V_n'$ and $h_n: E_n \rightarrow E_n'$ are the 
bijections defined as in Definition \ref{def_isom BD}. 
\end{definition} 

We show that an order isomorphism of two generalized  
Bratteli diagrams implies an isomorphism of the respective 
generalized Bratteli-Vershik dynamical systems. 

\begin{thm} Let  $B = (V, E, >)$ and $B' = (V', E',>')$ be order 
isomorphic generalized Bratteli diagrams. 
Assume that $|X_B(max)| = |X_B(min)|$ and $|X_{B'}(max)| = 
|X_{B'}(min)|$. Then the orders $>$ and $>'$ generate Vershik maps 
$\varphi_{B}: X_B \rightarrow X_B$ and $\varphi_{B'} : X_{B'} 
\rightarrow X_{B'}$ such that the generalized 
Bratteli-Vershik systems $(X_B, \varphi_{B})$ and $(X_{B'}, 
\varphi_{B'})$ are Borel isomorphic.
\end{thm}

\begin{proof} We will construct two Vershik maps $\varphi_{B}: 
X_B \rightarrow X_B$ and $\varphi_{B'}  : X_{B'} \rightarrow 
X_{B'}$ and define a Borel map $f: X_{B} \rightarrow X_{B'}$
such that, for $x \in X_B$, 
\begin{equation}\label{eq: conj}
f(\varphi_{B} (x)) = \varphi_{B'} (f(x)).
\end{equation} 
We use the concatenation of the maps $(h_n : E_n 
\rightarrow E'_n)_{n \in \N_0}$ given in Definition 
\ref{def_isom BD} to define  $f$ for $x = (x_0, x_1,...)\in X_B$, 
$$ f(x) = (h_n(x_n))_{n \in \N_0} : = (x'_n)_{n \in \N_0} : = 
x' \in X_{B'}.
$$ 
Since every $h_n$ is a bijection, we see that $f: X_{B} 
\rightarrow X_{B'}$ is a well defined bijective map. It follows 
from  Definition \ref{def_isom BD} that $f(X_{B} (max)) =  
X_{B'} (max)$ and $f(X_{B} (min)) = X_{B'} (min)$. 

By Definition \ref{Def:VershikMap}, we have 
\begin{equation}\label{eq: restricted1}
    \varphi_B : X_{B}\setminus X_{B} (max)\rightarrow X_{B}
    \setminus X_{B} (min), \ \  \varphi_{B'} : X_{B'}\setminus
    X_{B'}(max) \rightarrow X_{B'}\setminus X_{B'} (min)
\end{equation}
We show that, for $x \in X_{B}\setminus X_{B} (max)$, the 
map $f : X_{B}\setminus X_{B} (max) \rightarrow X_{B'}\setminus
X_{B'} (max) $ intertwines the Vershik maps $\varphi_B$ and 
$\varphi_{B'}$.
To see this, take $x = (x_0, x_1,...)\in X_B\setminus X_B(max)$ and
find  the smallest integer $k$ such that $x_k$ is not maximal. Let 
$y_k$ be the successor of $x_k$ in the finite set $r^{-1}(r(x_k))$.
Then, by definition of the Vershik map, $\varphi_B(x)= (y_0, 
y_1,...,y_{k-1},y_k,x_{k+1},...)$
where $(y_0, y_1,..., y_{k-1})$ is the minimal path in $E(V_0, 
r(y_{k-1}))$. By definition of $f$, we have
$$
f(\varphi_B(x)) = (h_0(y_0), \cdots, h_{k-1}(y_{k-1}), h_k(y_k), 
h_{k+1}(x_{k+1}), \cdots )
$$
where $(h_0(y_0), \cdots, h_{k-1}(y_{k-1}))$ is the minimal path
with range $s'(h_k(y_k))$. We note that $h_k(y_k)$ is the successor 
of $h_k(x_k)$ in the set $(r')^{-1}(r'(h(y_k)))$.

We now compute $\varphi_{B'}(f(x))$:
$$
\ba
\varphi_{B'}(f(x)) = &\ \varphi_{B'}(h_0(x_0), \cdots, 
h_{k-1}(x_{k-1}), h_k(x_k), h_{k+1}(x_{k+1}), \cdots )\\
= & \ (h_0(y_0), \cdots, h_{k-1}(y_{k-1}), h_k(y_k), 
h_{k+1}(x_{k+1}),  \cdots) 
\ea
$$
because the minimal path with range in $s'(h_k(y_k))$ is unique.
Hence, $\varphi_{B'}f = f\varphi_B$ on $X_{B}\setminus 
X_{B} (max)$. 

It remains to show that the relation $\varphi_{B'}f = f\varphi_B$
can be extended to all $X_B$.
By the condition of the theorem, $\varphi_B$ can be extended
to $X_B$, we keep the same notation $\varphi_B$ for the extension.
Since $f$ implements a bijection between $X_B(max)$ and 
$X_{B'}(max)$ and between $X_B(min)$ and $X_{B'}(min)$, we can 
extend the definition of $\varphi_{B'}$ as follows: for $x \in
X_{B'} (max)$, we set
\begin{equation}\label{eq:iso} 
    \varphi_{B'} (x) = f \circ \varphi_B \circ f^{-1} (x) . 
\end{equation}
Then $\varphi_{B'} : X_{B'}(max) \to X_{B'}(min)$, and, taking 
into account the result proved above, we conclude that
relation \eqref{eq:iso} holds for all $x$. 
\end{proof} 

Now we mention an observation that will be of use later on.

\begin{lemma}\label{lem:Vershik cont}
Let $B = (V, E, >)$ be an ordered generalized Bratteli diagram. 
Then the sets $X_{max}$ and $X_{min}$ of maximal and minimal
infinite paths 
are closed, in particular, they can be empty. The Vershik 
map $\varphi_B : X_B 
\setminus X_{max} \rightarrow X_B \setminus X_{min}$ is a homeomorphism. 
\end{lemma}

Let $(X, \B)$ be a standard Borel space. 
Recall that any two uncountable standard Borel spaces
are Borel isomorphic. For a standard Borel space $(X, \B)$, a 
one-to-one Borel map $T$ of $X $ onto itself is called a 
\textit{Borel automorphism} of $(X, \B)$. The following result 
shows that any aperiodic Borel automorphism of a standard Borel 
space can be realized as a Vershik map on a generalized Bratteli 
diagram.

\begin{theorem}[\cite{BezuglyiDooleyKwiatkowski_2006}]
\label{Thm:GBD_models_Borel_dyn} Let $T$ be an aperiodic Borel
 automorphism acting
on a standard Borel space $(X, \B)$. Then there exists an ordered
generalized Bratteli diagram $B=(V,E, >)$ and a Vershik map  
$\varphi_B : X_B \to X_B$ such that $(X, T)$ is Borel isomorphic to 
$(X_B,\varphi_B)$. Moreover, $\varphi_B$ is a homeomorphism of 
the path space $X_B$.
\end{theorem}

%%%%%
\section{Generalized Bratteli diagrams of bounded size} \label{sect Bndd size} 

In this section, we 
discuss generalized Bratteli diagrams of \textit{bounded size} (see 
Definition \ref{Def:BD_bdd_size}). These diagrams are characterized 
by the fact that all incidence matrices are \textit{banded}. Such diagrams have a locally compact path space and in our analysis, we consider them as an intermediate step between 
the classical Bratteli diagrams (i.e. having finitely many vertices 
at each level) and generalized Bratteli diagrams.
Generalized Bratteli diagrams of bounded size present models for substitution dynamical systems on countably infinite alphabets (see~\cite{Bezuglyi_Jorgensen_Sanadhya_2022}).  We present more examples of 
generalized Bratteli diagrams of bounded size together with the corresponding substitution dynamical systems and invariant measures 
in Subsection \ref{Subsec:Examples}.
\ignore{In the 
rest of this paper, generalized Bratteli diagrams of bounded size 
will play an important role. }
In Subsection 
\ref{Subsec:bdd_size} we prove  statements about the 
structure of such diagrams which will be of use later on. 
In Subsection \ref{sub:Left_right}, we consider the Vershik map on 
the path space of a bounded size diagram defined by the 
\textit{left-to-right order}.  We show
that, unlike the classical simple Bratteli diagrams, the left-to-
right order on an irreducible 
generalized Bratteli diagram of bounded size does not necessarily give rise to 
a continuous Vershik map. We give conditions under which the left-to-right order can give rise to a continuous Vershik map and conditions under which every infinite maximal path is necessarily a point of discontinuity.

%%%%
\subsection{The structure of the generalized Bratteli diagrams of bounded 
size}\label{Subsec:bdd_size} 

This subsection is dedicated to studying the structure of 
generalized Bratteli diagrams $B=(V, E)$ of \textit{bounded size}. 
Unless stated otherwise, we will identify the set of vertices at 
each level with integers, i.e., for each $i \in \N_0$, we have 
$V_i = \Z$. 
Observe that such a diagram has a locally compact path space $X_B$, 
and the restriction of $X_B$ to any cylinder set can be represented 
as the path space of a standard Bratteli diagram.

\begin{definition}\label{Def:BD_bdd_size} A generalized Bratteli 
diagram $B(F_n)$ is called of \textit{bounded size} if there exists 
a sequence of pairs of natural numbers $(t_n, L_n)_{n \in \N_0}$ 
such that, for all $n \in \mathbb{N}_0$ and all $v \in V_{n+1}$,
\begin{equation}\label{eq: Bndd size}
s(r^{-1}(v)) \in \{v - t_n, \ldots, v + t_n\} \quad \mbox{and} 
\quad \sum_{w \in V_{n}} f^{(n)}_{vw} = \sum_{w \in V_{n}} |E(w,v)| 
\leq L_n.
\end{equation} 
If the sequence $(t_n, L_n)_{n \in \N_0}$ is constant, i.e. $t_n = 
t$ and $L_n = L$ for all $n \in \N_0$, then we say that the 
diagram $B(F_n)$ is of \textit{uniformly bounded size}.
\end{definition} 

Observe that, the condition $ \sum_{w \in V_{n}} f^{(n)}_{v,w} \leq 
L_n$ implies that the set $s(r^{-1}(v))$ is finite. Moreover, 
the cardinality $|s(r^{-1}(v))|$ is bounded as a function of $v$
for every level $V_n$. The role 
of the first condition in \eqref{eq: Bndd size} is to control the
sources of edges arriving at $v \in V_{n+1}$.

\begin{remark} (1)
We will use the following convention for bounded size Bratteli 
diagrams. For each $n \in \N_0$, the pair of natural numbers 
$(t_n, L_n)$ are chosen to be minimal possible. Also, for every 
$n \in \mathbb{N}_0$, it is assumed that $E(v - t_n, v)$ and 
$E(v + t_n, v)$ are nonempty for all $v \in V_{n+1}$.
These assumptions are made to simplify our notation and 
calculations. 
Otherwise, we would 
have to use two sequences, $(t_n^+)$, $(t_n^-)$, where
$$
t_n^+ = \max\{t : E(v + t, v) \neq \emptyset\}
$$
and 
$$
t_n^- = \max\{t : E(v - t, v) \neq \emptyset\}.
$$
The usage of two different sequence $(t^{\pm}_n)$ instead of 
$(t_n)$ would affect computations but not the corresponding 
results.

(2) Observe that the property of bounded size implies that the 
incidence matrices of the diagram are \textit{banded} infinite 
matrices. Let $B = B(F_n)$ be a generalized Bratteli diagram of 
bounded size corresponding to a sequence $(t_n, L_n)_{n \in \N_0}$. 
Then all non-zero entries of $F_n$ belong to a band of width 
$2t_n +1$ along the main diagonal. Moreover, the sum of entries 
in every row of $F_n$ is bounded by $L_n$. In other words, 
for every $n \in \N_0$ and $v \in V_{n+1}$ we have $f^{(n)}_{v,w} 
= 0 \ \mbox{if} \ |v - w| > t_n$
and $\sum_{w \in V_n} f^{(n)}_{v,w} \leq L_n$. Also, by the above
assumption, $E(v \pm t_n, v) \neq \emptyset$ 
(or $f^{(n)}_{v,v \pm t_n} >0$) for all $v\in V_{n+1}, v \pm t_n 
\in V_n$ and $n \in \mathbb{N}_0$.
    
\end{remark} 

In \cite{Bezuglyi_Jorgensen_Sanadhya_2022}, the authors studied  
substitution dynamical systems on a countably infinite alphabet 
$\A$ as Borel dynamical systems. It was proved that 
a substitution dynamical system 
on a countably infinite alphabet which is 
\textit{left determined} (a generalization of the recognizability 
property to the countable case, see \cite{Ferenczi_2006})
and has \textit{bounded size} (see the definition below) admits 
a realization as a Vershik map on a stationary generalized 
Bratteli diagram.
 
\begin{definition}Identifying $\A$ with $\Z$, we say that a 
substitution $\sigma : n \to \sigma(n) $, $n \in \Z$, is of 
\textit{bounded size} if it is of bounded length and there exists 
a positive integer $t$ (independent of $n$) such that for every 
$n \in \Z$, if $m \in \sigma (n)$, then $m \in \{n-t,...,n,..., 
n+t\}$. 
 \end{definition} 

Clearly, this definition is analogous to the definition of bounded 
size generalized Bratteli diagrams.

\begin{theorem}\cite{Bezuglyi_Jorgensen_Sanadhya_2022}
Let $\sigma$ be a bounded size left determined substitution on a countably infinite alphabet and $(X_{\sigma}, T)$ be the corresponding subshift. Then there exists a
stationary ordered generalized Bratteli diagram $B = (V, E, \geq)$ of bounded size and a Vershik map $\varphi \colon X_B \rightarrow X_B$ such that $(X_{\sigma}, T)$ is Borel isomorphic to $(X_B, \varphi)$.
\end{theorem}

In the rest of the subsection, we discuss the structure of 
generalized Bratteli diagrams of bounded size.

\begin{lemma}\label{lemma_bdd_size_upper_cone}
Let $B=(V, E)$ be a generalized Bratteli diagram of bounded size. 
Let $n \in \mathbb{N}_0$, $v \in V_{n+1}$ and $E(V_0, v)$ be the set of all finite paths 
$\ov e = (e_0, \ldots, e_{n})$ such that $r(\ov e) = v$. Then $$
s(E(V_0, v)) \subset \left\{v - \sum_{i = 0}^n t_i, \ldots, v + 
\sum_{i = 0}^n t_i\right\}
$$
and
$$
|E(V_0, v)| \leq L_0 \cdots L_n.
$$
\end{lemma}

\begin{proof}
We prove the lemma by induction. Case $n = 0$ follows from the 
definition. Suppose the statement of the lemma is true for $n = k$. 
Then for any $v \in V_{k+2}$,
$$
s(r^{-1}(v)) \in \{v - t_{k+1}, \ldots, v + t_{k+1}\}
$$
and
\begin{eqnarray*}
s(E(V_0, v)) &\subset& \bigcup_{w \in \{v - t_{k+1}, \ldots, v + 
t_{k+1}\}} s(E(V_0, w))\\
 &\subset& \left\{v - t_{k+1} - \sum_{i = 0}^k t_i, \ldots, v + 
 t_{k+1} + \sum_{i = 0}^k t_i\right\}\\
 &=& \left\{v - \sum_{i = 0}^{k+1} t_i, \ldots, v + \sum_{i = 
 0}^{k+1} t_i\right\}.
\end{eqnarray*} We also have
$$
|E(V_0,v)| = \sum_{w \in V_{k+1}} f_{vw}^{(k+1)}|E(V_0,w)| 
\leq (L_0 \cdots L_k)\sum_{w \in V_{k+1}} f_{vw}^{(k+1)} \leq L_0 
\cdots L_k \cdot L_{k+1}.
$$
\end{proof}

\begin{corol}\label{corol_complete_upper_cone}
Let $B=(V, E)$ be a generalized Bratteli diagram of bounded size. 
Let $n \in \mathbb{N}_0$, $v \in V_{n+1}$ and $E(V_0, v)$ be the set of all finite paths 
$\ov e = (e_0, \ldots, e_{n})$ such that $r(\ov e) = v$. Then for 
every $m \leq n$ we have 
$$
s(E(V_m, v)) \subset \left\{v - \sum_{i = m}^{n} t_i, \ldots, v + 
\sum_{i = m}^n t_i\right\}.
$$
\end{corol}

\begin{proof} The proof follows from an induction argument similar 
to the proof of Lemma \ref{lemma_bdd_size_upper_cone}.
\end{proof}

Assume that $B=(V,E)$ is a diagram of bounded size corresponding to 
a sequence $(t_n, L_n)_{n \in \N_0}$. Fix $v \in V_{n+1}$ and let 
$\ov e = (e_0, \ldots, e_{n})$ be a finite path with $r(\ov e) = v 
\in V_{n+1}$. By 
Lemma~\ref{lemma_bdd_size_upper_cone}, 
$$
s(e_0) \in \left\{v - \sum_{i = 0}^n t_i, \ldots, v + 
\sum_{i = 0}^n t_i\right\}.
$$

\begin{lemma}\label{lemma_bdd_size_lower_cone} Let $B=(V,E)$ be a 
diagram of bounded size corresponding to a sequence 
$(t_n, L_n)_{n \in \N_0}$. Fix $v \in V_{n+1}$, and let $\ov e = 
(e_0, \ldots, e_{n})$ be a finite path with $r(\ov e) = v \in 
V_{n+1}$. Then for all infinite paths $x = 
(x_n)_{n\in \N_0} \in [\ov e]$ and all $m \geq 0$: 
\begin{equation}\label{formula_bdd_size_lower_cone}
r(x_{n + m}) \in \left\{v - \sum_{i = 1}^m t_{n + i}, \ldots, v + 
\sum_{i = 1}^m t_{n + i}\right\} \subset V_{n+m +1}.
\end{equation}
\end{lemma}

\begin{proof}
We prove this lemma by induction. Fix $v \in V_{n+1}$, then 
(\ref{formula_bdd_size_lower_cone}) is trivially true for 
$m = 0$. By induction step, we assume that 
(\ref{formula_bdd_size_lower_cone}) holds for $m = k$ i.e. 
$$
r(x_{n + k}) \in \left\{v - \sum_{i = 1}^k t_{n + i}, \ldots, v + 
\sum_{i = 1}^k t_{n + i}\right\} \subset V_{n+k +1}.
$$
By Lemma \ref{lemma_bdd_size_upper_cone} if $u \in V_{n + 
k + 2}$, such that $E(v,u) \neq \emptyset$ then 
$$
v \in  \left\{u - \sum_{i = 0}^k t_{n + 1 + i}, \ldots, u + 
\sum_{i = 0}^k t_{n + 1 + i}\right\}.
$$ 
This implies,
$$
u \in  \left\{v - \sum_{i = 1}^{k+1} t_{n + i}, \ldots, v + 
\sum_{i = 1}^{k+1} t_{n + i}\right\}.
$$ In other words
$$
r(x_{n+k+1}) \in \left\{v - \sum_{i = 1}^{k+1} t_{n + i}, \ldots, 
v + \sum_{i = 1}^{k+1} t_{n + i}\right\} \subset V_{n+k+2}
$$ as needed.
\end{proof}

\subsection{Continuity of a Vershik 
map for generalized Bratteli diagrams of bounded size}\label{sub:Left_right} 
In the case of classical Bratteli diagrams, i.e., the diagrams
with finitely 
many vertices at each level, it is a well-known fact that, for 
simple Bratteli diagrams, the left-to-right order always gives rise 
to a Vershik homeomorphism 
\cite{HermanPutnamSkau1992}, \cite{Durand2010}. 
In this subsection, 
we show that this is not true for irreducible generalized Bratteli 
diagrams of bounded size. 

%We also provide an example of a Bratteli diagram (not of bounded size) for which the sets of infinite minimal and maximal paths are empty, hence the corresponding Vershik map is a homeomorphism.

Let $B=(V, E)$ be an irreducible generalized Bratteli diagram of 
bounded size with the corresponding sequence 
$(t_n, L_n)_{n \in \N_0}$. We will denote by $\omega$ a fixed 
partial order on $E$. To emphasize that a 
diagram $B$ is ordered, we will write $(B,\omega)$. 
Let $X_{max} = X_{max}(\omega)$ and $X_{min} = X_{min}(\omega)$ 
denote the sets of infinite 
maximal and minimal paths, respectively, with respect to the 
corresponding order.

\begin{lemma}\label{gbd size} Let $B = (B,\omega)$ be an 
ordered generalized Bratteli diagram of 
bounded size where $\omega$ is the left-to-right partial order on 
$E$, then
$$
|X_{max}|= |X_{min}| = \aleph_{0}.
$$
\end{lemma}

\begin{proof} Let $(t_n, L_n)_{n \in \N_0}$ denote the sequence 
corresponding to the bounded size diagram $(B,\omega)$. Recall that 
for $n \in \N$ and $v \in V_{n}$, $E(V_0, v)$ denotes the set of 
all finite paths $\ov e = (e_0, \ldots, e_{n})$ such that 
$r(\ov e) = v$. Since 
$\omega$ is left-to-right ordering, the leftmost finite path and 
the rightmost finite path is the unique minimal and the unique
maximal path,
respectively, in the set $E(V_0, v)$. We denote them by $\ov 
e_{min}$ and $\ov e_{max}$. Since $(B,\omega)$ is of bounded size 
there are  vertices $u= v + t_{n+1}$ and $u'= v - t_{n+1}$  in
$V_{n+1}$ such that $e(v,u)$ is the  minimal edge in the set 
$r^{-1}(u)$ and $e(v,u')$ is the maximal 
edge in the set $r^{-1}(u')$. 

This observation implies that, for every $v \in V_n$, the finite
minimal path
$\ov e_{min} \in E(V_0, v)$ and  the finite maximal path 
$\ov e_{max} \in E(V_0, v))$ have unique minimal and unique maximal
extensions to the level $n+1$. This proves that the sets $X_{max}$ 
and $X_{min}$ are countable. 
\end{proof}

Observe that for a two-sided generalized Bratteli diagram of 
bounded size $B=(V, E)$ (i.e. each $V_i$ is identified with $\Z$) 
with left-to-right ordering there exist countably infinite 
``slanted'' 
minimal paths which go ``from left-to-right'' countably infinite 
``slanted'' maximal paths which go ``from right to left''. Since by 
convention, each vertex has the rightmost and leftmost incoming 
edge corresponding to the parameters $t_n$, all infinite minimal 
paths go ``parallelly'' to each other, the same holds for maximal 
paths.

Note that for an ordered generalized Bratteli diagram $(B, \omega)$,
the Vershik map $\varphi_B(\omega)$ (corresponding to $\omega$) is 
a well defined map from $X_B \setminus X_{max}(\omega)$ to $X_B 
\setminus X_{min}(\omega)$. Denote by $\Phi_B(\omega)$ the set of 
all possible extensions of $\varphi_B(\omega)$ to Vershik maps 
defined on the entire $X_B$. In general,  $\Phi_B(\omega)$ can be 
empty. As mentioned in Lemma \ref{lem:Vershik cont}, the map 
$\varphi_B(\omega) : X_B \setminus X_{max}(\omega) \to X_B 
\setminus X_{min}(\omega)$ is continuous. In what follows, we
discuss the question of whether $\varphi_B(\omega)$ can be 
extended to a \textit{continuous} Vershik map on the entire path 
space $X_B$.

We first give a simple example where there is no Vershik map 
on a generalized Bratteli diagram, in other words, the map 
$\varphi_B(\omega) : X_B \setminus X_{max}(\omega) \to X_B 
\setminus X_{min}(\omega)$ can not be extended to the entire path 
space in Borel fashion. 

\begin{example}\label{Ex:LROrdering_no_V.map} Let $B = 
(B,\omega)$ be a one-sided generalized Bratteli diagram of bounded 
sized (i.e. each $V_i$ is identified with $\N$) where $\omega$ is 
the left-to-right ordering. We considered such a diagram in
Example~\ref{Ex:isom_bd} (although without the ordering, see the second diagram in 
Figure \ref{Fig:Isom_BD}). Then it is easy to see that $(B,\omega)$ 
admits infinitely many minimal paths and no maximal paths in the 
diagram. 
Hence, the Vershik map $\varphi_B :X_B \setminus X_{max}
\to X_B \setminus X_{min}$ cannot be extended to a Borel 
bijection to $X_B$.
\end{example} 

Now we show that there exists a class of generalized Bratteli 
diagrams of bounded size such that the left-to-right order does 
not give rise to a continuous Vershik map. 

\begin{theorem}\label{thm_VM not cont}
Let $B  = (B,\omega)$ be a generalized Bratteli diagram of 
bounded size, with the corresponding sequence 
$(t_n, L_n)_{n \in \N_0}$ and left-to-right ordering $\omega$. 
Assume that for all $N \in \mathbb{N}$ there exists $k \in 
\mathbb{N}$ such that 
\begin{equation}\label{condit_discont_V_map}
t_{N+k} < \sum_{i = N}^{N + k - 1} t_i.
\end{equation}
Let $\varphi_{B} (\omega) \in \Phi(\omega)$ be any Vershik map on
$X_B$ corresponding to the order $\omega$. Then 
$\varphi_{B} (\omega)$ is not continuous. Moreover, 
every infinite maximal path is a point of discontinuity. 
\end{theorem}

\begin{proof} 
Assume that the vertices of $B$ are enumerated by $\mathbb{Z}$.
By Lemma \ref{gbd size}, there always exists a well defined 
Vershik map $\varphi_{B} = \varphi_{B}(\omega)$ on $X_B$. 
To show that $\varphi_B$ is not continuous, we will prove that
any neighborhood of a maximal path contains two infinite paths, 
$x'$ and $x''$, such that the distance between their images under 
the Vershik map is equal to $1$.

Pick any $x = (x_n)_{n \in \N_0}\in X_{max}$, we will show that in 
an arbitrarily small 
neighborhood of $x$ there are paths $x' = (x'_n)_{n \in \N_0}, x'' 
= (x''_n)_{n \in \N_0} \in 
X_B \setminus X_{max}$ such that $d(\varphi_B(x'), \varphi_B(x'') 
=1$. For $\varepsilon > 0$, take any $n$  such that 
$\frac{1}{2^{n-1}} < \varepsilon$. Let $x'_l = x_l$ for $l < n$ 
and $x_n'$ be the minimal edge with the source in $r(x_{n-1})$. 
Define $x'_k$ for $k > n$ in an arbitrary way such that $x' = 
(x_l')$ form a path in $X_B$. Then $d(x,x') < \varepsilon$ and $x'$ 
is not a maximal path. Let $y' = (y'_n) = \varphi_B(x')$. 
Then $y'_n$ 
is the successor of $x_n'$ and since the ordering is left-to-right, 
we have $s(y_n') \geq r(x_{n-1})= s(x'_n)$. 

By~(\ref{condit_discont_V_map}), there exists $k \in \mathbb{N}$
such that 
$$
t_{n+k} < \sum_{i = n}^{n + k - 1} t_i.
$$ 
Set $x''_l = x_l$ for $0 \leq l \leq n+k-1$. 
Let $x_{n+k}''$ be the minimal edge with the source in 
$r(x_{n+k-1})$ and let $x''_l$ for $l > n+k-1$ be defined in an 
arbitrary way such that $x'' = (x_l'')$ form a path in $X_B$. Then 
$d(x,x'') < \varepsilon$ and $x''$ is not a maximal path. Let 
$y''_{n+k}$ be the successor of $x''_{n+k}$ and denote $y'' = 
(y''_l) = \varphi_B(x'')$.
Since $B$ is a diagram of bounded size, we have $s(y''_{n+k}) \leq 
r(x_{n+k-1}) + 2t_{n+k}$. Recall that we assume that the sets
$E(v-t_n,v)$ and 
$E(v+t_n,v)$ are nonempty for all $v \in V_{n+1}$ and all $n$. 
This means that  
$$
r(x_{n+k-1}) = r(x_{n-1}) - \sum_{i = n}^{n+k-1}t_i.
$$ 
Since $(y_0'',\ldots, y_{n+k-1}'')$ is the finite minimal path 
and the sets 
$E(v-t_n,v)$, $E(v + t_n,v)$ are nonempty for every $v \in 
V_{n+1}$, we have 
$$s(y''_{n}) = s(y''_{n+k}) -  \sum_{i = n}^{n+k-1}t_i.
$$ 
Thus,
$$
\ba
    s(y''_{n}) =\  & s(y''_{n+k}) -  \sum_{i = n}^{n+k-1}t_i\\
    \leq \ & r(x_{n+k-1}) + 2t_{n+k} - \sum_{i = n}^{n+k-1}t_i\\ 
    =\ & r(x_{n-1}) + 2t_{n+k} - 2\sum_{i = n}^{n+k-1}t_i\\
    < \ & r(x_{n-1}).
\ea     
$$ Hence, $s(y_n') > s(y_n'')$ and since all minimal paths in the 
diagram go ``parallelly'' to each other, we have 
$d(\varphi_B(x'), 
\varphi_B(x'')) = 1$.
\end{proof}

\begin{remark}
We note that the conditions of the theorem above hold for 
generalized Bratteli diagrams of uniformly bounded size. Indeed, 
since $t_n = t$ for all $n\in \mathbb{N}_0$, for every $N\in 
\mathbb{N}$ it is enough to take $k = 2$ 
and the condition~(\ref{condit_discont_V_map}) will be satisfied. 
Thus, even irreducible generalized Bratteli diagrams of uniformly 
bounded size with 
left-to-right ordering do not support a continuous Vershik map.
\end{remark} 

Now we show that there exists a class of generalized Bratteli 
diagrams of bounded size such that the left-to-right order does 
give rise to a continuous Vershik map. This class is 
dual to the class of diagrams considered in 
Theorem \ref{thm_VM not cont}.

\begin{theorem}\label{Thm:ContVM} 
Let $B = (B,\omega)$ be a 
generalized Bratteli diagram of bounded size, with the 
corresponding sequence 
$(t_n, L_n)_{n \in \N_0}$ and left-to-right ordering $\omega$. 
Let $L_n = 2$ for all $n \in 
\mathbb{N}$ and let there exist $N \in \mathbb{N}_0$ such that
\begin{equation}\label{condit_cont_V_map}
t_{N+k} = \sum_{i = N}^{N + k - 1} t_i
\end{equation}
for all $k \in \mathbb{N}$. Then there exists a continuous 
Vershik map corresponding to $\omega$.
\end{theorem}

\begin{proof} Clearly, the Vershik map is well-defined on 
$X_B \setminus X_{max}$. 
Define the Vershik map $\varphi_B$ on the set of maximal paths as 
follows: for every vertex $w \in V_N$, let $\varphi_B$ map the 
unique maximal path $x_{max}^{(w)}$ passing through $w$ to the 
unique minimal path $x_{min}^{(w)}$ passing through $w$. Then 
$\varphi_B$ is a bijection. Since every vertex $v \in V \setminus 
V_0$ has exactly two incoming edges, it is easy to verify the 
continuity of $\varphi_B$. Indeed, let $w \in V_N$ and $x$ be any 
non-maximal path that coincides with $x_{max}^{(w)}$ up to 
level $M \geq N$. Then
$$
d(x_{max}^{(w)}, x) = \frac{1}{2^M} \leq \frac{1}{2^N}.
$$ 
Then it follows from the structure of the diagram that the distance 
$$
d(\varphi_B(x), x^{(w)}_{min}) = \frac{1}{2^M}.
$$ Thus, $x$
will be mapped to the neighborhood of the corresponding $x_{min}$ 
(see Figure \ref{Fig:ContinVMap} in Example \ref{ex_Cont VM} below).
\end{proof}

\begin{figure}[ht]
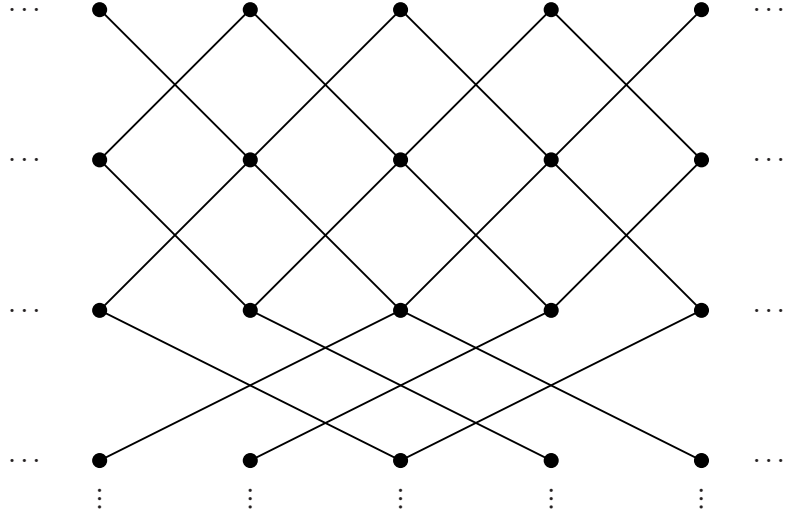

\unitlength=1cm
\begin{graph}(11,8)
% \graphnodesize{0.2}
% \roundnode{V0}(3,6)
%  %\nodetext{V0}(-1,0){$V_0$}
 \roundnode{V11}(2,7)
 %\nodetext{V11}(-0.7,0){$w_1^{(0)}$}
  %\nodetext{V12}(0.7,0){$w_2^{(0)}$}
 \roundnode{V12}(4,7)
 \roundnode{V13}(6,7)
 \roundnode{V14}(8,7)
  \roundnode{V15}(10,7)

  %
    % The second level vertices
   \roundnode{V21}(2,5)
 %\nodetext{V11}(-0.7,0){$w_1^{(0)}$}
  %\nodetext{V12}(0.7,0){$w_2^{(0)}$}
 \roundnode{V22}(4,5)
 \roundnode{V23}(6,5)
 \roundnode{V24}(8,5)
  \roundnode{V25}(10,5)
  % The third level vertices
 \roundnode{V31}(2,3)
 \roundnode{V32}(4,3)
 \roundnode{V33}(6,3)
  \roundnode{V34}(8,3)
    \roundnode{V35}(10,3)
  
  % The fourth level vertices 
  \roundnode{V41}(2,1)
 \roundnode{V42}(4,1)
 \roundnode{V43}(6,1)
  \roundnode{V44}(8,1)
    \roundnode{V45}(10,1)
  
 %
 % EDGES
 \graphlinewidth{0.025}
% % First level
%  \edge{V0}{V11}
%  \edge{V0}{V12}

     \edge{V22}{V11}
    \edge{V21}{V12}
     
 \edge{V23}{V12}
    \edge{V22}{V13}
    
   \edge{V24}{V13}
    \edge{V23}{V14}
    
   \edge{V25}{V14}
    \edge{V24}{V15}
    
 % Second level
\edge{V32}{V21}
    \edge{V31}{V22}
     
 \edge{V33}{V22}
    \edge{V32}{V23}
    
   \edge{V34}{V23}
    \edge{V33}{V24}
    
   \edge{V35}{V24}
    \edge{V34}{V25}
    
     % Third level
     
 %\edge{V42}{V31}
    \edge{V41}{V33}
     
 \edge{V43}{V31}
    \edge{V42}{V34}
    
   \edge{V44}{V32}
    \edge{V43}{V35}
    
   \edge{V45}{V33}
    %\edge{V44}{V35}
    
  \freetext(10.9,7){$\ldots$}
  \freetext(10.9,5){$\ldots$}
 \freetext(10.9,3){$\ldots$}
  \freetext(10.9,1){$\ldots$}
   \freetext(1,7){$\ldots$}
  \freetext(1,5){$\ldots$}  
   \freetext(1,3){$\ldots$}
  \freetext(1,1){$\ldots$}
  \freetext(2,0.5){$\vdots$}
  \freetext(4,0.5){$\vdots$}
    \freetext(6,0.5){$\vdots$}
      \freetext(8,0.5){$\vdots$}  
        \freetext(10,0.5){$\vdots$} 
\end{graph}
\caption{A continuous Vershik map (left-to-right 
ordering)}\label{Fig:ContinVMap}
\end{figure} Now we give an example which illustrates Theorem \ref{Thm:ContVM}. 

\begin{example}\label{ex_Cont VM} Let $B  = (B,\omega)$ be a 
two-sided generalized Bratteli diagram of bounded size, with the 
corresponding sequence 
$(t_n, L_n)_{n \in \N_0}$ and left-to-right ordering $\omega$. 
Suppose that 
$$
L_n = 2, \,\,\, t_n = 2^{n-1} \,\,\,\mathrm{for} \,\,\, n \in \N 
\,\,\,\mathrm{and}\,\,\, t_0 = 1.
$$ 
The diagram $B = (V,E)$ is shown in Figure \ref{Fig:ContinVMap}. 
In other words, every vertex $v \in V \setminus V_0$ 
has exactly two incoming edges. Endow $B$ with the left-to-right 
ordering. To define the Vershik map, we use the following rule:   
for every vertex $w \in V_0$, the maximal infinite path 
which starts at $w$ is mapped to the minimal infinite path 
which starts at the same vertex $w$. 
Then it is easy to check that the corresponding Vershik map 
is a homeomorphism. Note that, in this case, 
for every $n \geq 1$
$$
t_{n} = \sum_{i = 0}^{n-1} t_i
$$
(compare with the result of Theorem \ref{Thm:ContVM}).
We remark also that the Bratteli diagram, constructed in this 
example, is reducible. Indeed, starting from the 
level $V_2$, all infinite paths that begin at even vertices 
will go through even vertices only.

\end{example}

\section{Continuity of Vershik map for generalized Bratteli diagrams}\label{Sect:contin_V_map}

The aim of this section is to find conditions under which a 
generalized Bratteli diagram 
admits an order such 
that the corresponding Vershik map can be prolonged to a 
homeomorphism.
In particular, we are interested in orders $\omega$ such that there are no infinite minimal and no infinite maximal paths. The Vershik map $\varphi_B(\omega)$ corresponding to such an order is uniquely defined, and it is always a homeomorphism. Such orders were also used in 
 Theorem \ref{Thm:GBD_models_Borel_dyn} where generalized Bratteli diagrams corresponding to Borel dynamical systems were constructed using sequences of vanishing markers \cite{BezuglyiDooleyKwiatkowski_2006}. 

\subsection{Ordered generalized Bratteli diagrams with no infinite minimal and maximal paths.}
Let $B$ be a generalized Bratteli diagram such that, for every level $n \in \mathbb{N}$, the set of its vertices $V_n$ is identified with $\mathbb{N}$. Note that this condition is not a restriction, since the set of vertices of any generalized Bratteli diagram can be enumerated this way (see Subsection~\ref{Subsec:isomBD} and Example~\ref{Ex:isom_bd}). 
We call an edge $e$ of such a diagram \textit{slanted from right to left} if $r(e) < s(e)$, \textit{slanted from left to right} if $r(e) > s(e)$, and \textit{vertical} if $r(e) = s(e)$. We also say that a finite or infinite path is slanted from right to left if it consists of edges slanted from right to left.

\begin{theorem}\label{Thm:suff_cond_no_x_min_max}
    Let $B$ be a generalized Bratteli diagram such that the set of vertices on each level is identified with $\mathbb{N}$. Suppose that 
    there exists a level $N$ such that starting from that level every vertex has at least two incoming edges which are slanted from right to left. Then $B$ admits an order such that the sets of infinite minimal and infinite maximal paths are empty.
\end{theorem}

\begin{proof}
Put an arbitrary order on $B$ up to level $N$. Starting from level $N$, we define an order such that, for every vertex, all minimal and maximal edges are slanted from right to left. Indeed, by the assumption, every vertex $v$ beginning level $N$ and below has at least two incoming edges slanted from right to left, say $e_v$ and $e'_v$. We define an order $\omega$ such that a minimal edge and a maximal edge in 
$r^{-1}(v)$ are $e_v$ and $e'_v$, respectively. Since all
vertices are enumerated by natural numbers, we denote by
$v_1$ the leftmost vertex in every level starting from $N$.
Then the vertex $v_1$ of the diagram has no outgoing minimal edge and no outgoing maximal edge since all its outgoing edges are either vertical or slanted from left to right.
    
Let $x = (x_n)$ be any minimal or maximal path that passes through some vertex $w \in V_N$. Since for every $n \geq N$, all minimal and maximal edges are slanted from right to left, we obtain $r(x_n) < w + N - n$ and $r(x_{n+1}) < r(x_n)$ for every $n \geq N$. Since all levels of vertices are enumerated by natural numbers, every such path $x$ must be finite. %Indeed, eventually, the path $x$ will pass through $v_1$, the first vertex of the diagram, and there will be no way to prolong it further since there is no minimal edge and no maximal edge that start from the first vertex.

\end{proof}

\begin{example}\label{Ex:noXminnoXmax}
The diagram in Figure~\ref{Fig:noXminnoXmax} satisfies the conditions of  Theorem~\ref{Thm:suff_cond_no_x_min_max} for $N = 0$. Indeed, define $B = B(F)$ to be a one-sided infinite generalized Bratteli diagram
where $A = F^{T} = (a_{ij})_{i,j \in \N}$ is defined by
\begin{equation}
    A = \begin{pmatrix}
    1 & 1 & 1 & 1 & 1 & 1 & \cdots\\
    1 & 0 & 0 & 0 & 0 & 0 & \cdots\\
    1 & 1 & 0 & 0 & 0 & 0 & \cdots\\
    0 & 1 & 1 & 0 & 0 & 0 & \cdots\\
    0 & 0 & 1 & 1 & 0 & 0 & \cdots\\
    0 & 0 & 0 & 1 & 1 & 0 & \cdots\\
   \vdots & \vdots & \vdots & \vdots & \vdots & \vdots & \ddots
    \end{pmatrix}.
\end{equation}
In other words, for all $j \in \mathbb{N}$,
\begin{equation} \label{Matrix A_HomeoVmap}
     a_{ij} = \left\{
\begin{aligned}
& 1 \mbox{ if } i = 1 \mbox{ or } i = j + 1 \mbox{ or } i = j + 2,\\
& 0 \mbox{ otherwise }.
\end{aligned}
\right.
\end{equation}
Note that every vertex in $V \setminus V_0$ has exactly three incoming edges. Define a stationary order on $B$ as follows: let all the edges outgoing from the first vertex have label $1$, and all other edges can be labeled by $0$ and $2$ in an arbitrary way. So it defines an ordered generalized Bratteli diagram. Then, for any vertex $v \in V_0$, any minimal or maximal path starting from this vertex will be slanted from right to left and finite (see Figure~\ref{Fig:noXminnoXmax} below). Notice that the Vershik map corresponding to this diagram is a minimal homeomorphism. Indeed, every infinite path in $X_B$ passes through the first vertex infinitely many times and the orbit of such path visits every cylinder set of $X_B$.

    \begin{figure}[hbt!]
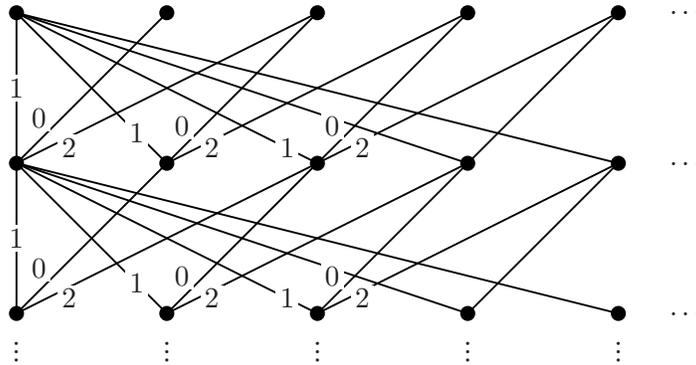

\unitlength=1cm
\begin{graph}(9,6)
% \graphnodesize{0.2}
% \roundnode{V0}(3,6)
%  %\nodetext{V0}(-1,0){$V_0$}
 \roundnode{V11}(0,5)
 %\nodetext{V11}(-0.7,0){$w_1^{(0)}$}
  %\nodetext{V12}(0.7,0){$w_2^{(0)}$}
 \roundnode{V12}(2,5)
 \roundnode{V13}(4,5)
 \roundnode{V14}(6,5)
  \roundnode{V15}(8,5)

  %
    % The second level vertices
   \roundnode{V21}(0,3)
 %\nodetext{V11}(-0.7,0){$w_1^{(0)}$}
  %\nodetext{V12}(0.7,0){$w_2^{(0)}$}
 \roundnode{V22}(2,3)
 \roundnode{V23}(4,3)
 \roundnode{V24}(6,3)
  \roundnode{V25}(8,3)
  % The third level vertices
 \roundnode{V31}(0,1)
 \roundnode{V32}(2,1)
 \roundnode{V33}(4,1)
  \roundnode{V34}(6,1)
    \roundnode{V35}(8,1)
  
 %
 % EDGES
 \graphlinewidth{0.025}
% % First level
%  \edge{V0}{V11}
%  \edge{V0}{V12}

    \edge{V21}{V11}
    \edgetext{V21}{V11}{$1$}
    \edge{V22}{V11}
    \edge{V21}{V12}
    \edge{V21}{V13}

 \edge{V23}{V11}
 \edge{V22}{V13}
  \edge{V22}{V14}
    
   \edge{V24}{V11}
    \edge{V23}{V14}
       \edge{V23}{V15}
    
   \edge{V25}{V11}
    \edge{V24}{V15}
    
 % Second level

    \edge{V31}{V21}
    \edgetext{V31}{V21}{$1$}
    \edge{V32}{V21}
    \edge{V31}{V22}
    \edge{V31}{V23}

 \edge{V33}{V21}
 \edge{V32}{V23}
  \edge{V32}{V24}
    
   \edge{V34}{V21}
    \edge{V33}{V24}
       \edge{V33}{V25}
    
   \edge{V35}{V21}
    \edge{V34}{V25}

  \freetext(8.9,5){$\ldots$}
 \freetext(8.9,3){$\ldots$}
  \freetext(8.9,1){$\ldots$}
  
  \freetext(0,0.5){$\vdots$}
  \freetext(2,0.5){$\vdots$}
    \freetext(4,0.5){$\vdots$}
      \freetext(6,0.5){$\vdots$}  
        \freetext(8,0.5){$\vdots$} 

          %\freetext(-0.4,4){$1$} 
          %\freetext(0.3,4.3){$1$} 
            \freetext(0.3,3.6){$0$} 
            \freetext(0.7,3.2){$2$}
             \freetext(2.2,3.5){$0$} 
            \freetext(2.6,3.2){$2$}
             \freetext(4.2,3.5){$0$} 
            \freetext(4.6,3.2){$2$}
            \freetext(0.3,1.6){$0$} 
            \freetext(0.7,1.2){$2$}
             \freetext(2.2,1.5){$0$} 
            \freetext(2.6,1.2){$2$}
             \freetext(4.2,1.5){$0$} 
            \freetext(4.6,1.2){$2$}
          %\freetext(1,5){$1$} 

            \freetext(1.6,3.4){$1$}
            \freetext(1.6,1.4){$1$}
            \freetext(3.6,3.2){$1$}
            \freetext(3.6,1.2){$1$}
\end{graph}
\caption{The Vershik map is a minimal homeomorphism, the sets of infinite minimal and maximal paths are empty}\label{Fig:noXminnoXmax}
\end{figure}
\end{example}

\begin{example}\label{Ex:NoXminXmax2}
The diagram in Figure~\ref{Fig:tel_noXminnoXmax} satisfies the conditions of the Theorem~\ref{Thm:suff_cond_no_x_min_max} for $N = 0$ after telescoping with respect to even levels. The incidence matrix of the diagram has the form
$$
F = \begin{pmatrix}
    2 & 1 & 0 & 0 & 0 & \cdots\\
    1 & 2 & 1 & 0 & 0 & \cdots\\
    0 & 1 & 2 & 1 & 0 & \cdots\\
    0 & 0 & 1 & 2 & 1 & \cdots\\
    0 & 0 & 0 & 1 & 2 & \cdots\\
    \vdots & \vdots & \vdots & \vdots & \vdots & \ddots
    \end{pmatrix}.
$$

The orders of $E_0$ and $E_1$ are presented in Figure~\ref{Fig:tel_noXminnoXmax}. For every even $i$, the edges $E_i$ are enumerated in the same way as $E_0$, and for every odd $i$, in the same way as $E_1$. Thus, after telescoping with respect to even-numbered levels, we obtain a stationary ordered Bratteli diagram such that all minimal and maximal edges are slanted from right to left.

    \begin{figure}[ht]
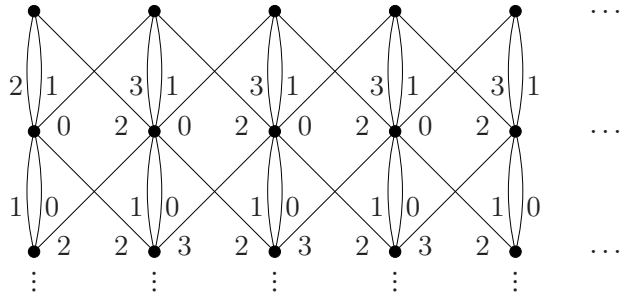

\unitlength = 0.4cm
\begin{center}
\begin{graph}(25,9)
\graphnodesize{0.4}

% Vertices of the first level
\roundnode{V11}(3,9)

\roundnode{V12}(7,9)
%\nodetext{V12}(-0.6,0.4){$v$}
\roundnode{V13}(11,9)
%\nodetext{V13}(0.6,0.4){$w$}
\roundnode{V14}(15,9)
\roundnode{V15}(19,9)

%Vertices of the second level
\roundnode{V21}(3,5)
\freetext(2.4,6.5){2}
\freetext(3.6,6.5){1}
\freetext(4,5.2){0}

\roundnode{V23}(7,5)%odometer
\freetext(5.9,5.2){2}
\freetext(6.4,6.5){3}
\freetext(7.6,6.5){1}
\freetext(8,5.2){0}

\roundnode{V24}(11,5)

\roundnode{V25}(15,5)

\freetext(13.9,5.2){2}
\freetext(14.4,6.5){3}
\freetext(15.6,6.5){1}
\freetext(16,5.2){0}

\roundnode{V26}(19,5)
\freetext(17.9,5.2){2}
\freetext(18.4,6.5){3}
\freetext(19.6,6.5){1}
%\freetext(19,6){$w$}

% Edges of the second level
\bow{V21}{V11}{-0.06}%odometer
\bow{V21}{V11}{0.06}%odometer

\bow{V23}{V12}{-0.06}%odometer
\bow{V23}{V12}{0.06}%odometer

\edge{V23}{V11}
\edge{V21}{V12}

\edge{V24}{V12}
%\edgetext{V24}{V12}{0}
\freetext(9.9,5.2){2}
\freetext(10.4,6.5){3}
\freetext(11.6,6.5){1}
\freetext(12,5.2){0}

\bow{V24}{V13}{-0.06}%odometer
\bow{V24}{V13}{0.06}%odometer
\bow{V25}{V14}{-0.06}%odometer
\bow{V25}{V14}{0.06}%odometer
\bow{V26}{V15}{-0.06}%odometer
\bow{V26}{V15}{0.06}%odometer
%\edgetext{V24}{V13}{1}
%\edgetext{V25}{V12}{0}
\edge{V25}{V13}
\edge{V26}{V14}
\edge{V23}{V13}
\edge{V24}{V14}
\edge{V25}{V15}

%\edgetext{V25}{V13}{1}
%\edgetext{V23}{V12}{0}

%Vertices of the third level
\roundnode{V31}(3,1)
\freetext(2.4,2.5){1}
\freetext(3.6,2.5){0}
\freetext(4,1.2){2}

\roundnode{V35}(7,1)
\freetext(5.9,1.2){2}
\freetext(6.4,2.5){1}
\freetext(7.6,2.5){0}
\freetext(8,1.2){3}

\roundnode{V36}(11,1)
\freetext(9.9,1.2){2}
\freetext(10.4,2.5){1}
\freetext(11.6,2.5){0}
\freetext(12,1.2){3}

\roundnode{V37}(15,1)
\freetext(13.9,1.2){2}
\freetext(14.4,2.5){1}
\freetext(15.6,2.5){0}
\freetext(16,1.2){3}

\roundnode{V38}(19,1)
\freetext(17.9,1.2){2}
\freetext(18.4,2.5){1}
\freetext(19.6,2.5){0}

\freetext(3,0){$\vdots$}
\freetext(7,0){$\vdots$}
\freetext(11,0){$\vdots$}
\freetext(15,0){$\vdots$}
\freetext(19,0){$\vdots$}
\freetext(22,9){$\ldots$}
\freetext(22,5){$\ldots$}
\freetext(22,1){$\ldots$}
% Edges of the third level

\bow{V31}{V21}{-0.06}%odometer
\bow{V31}{V21}{0.06}%odometer
\edge{V35}{V21}
\edge{V31}{V23}

\bow{V35}{V23}{-0.06}
\bow{V35}{V23}{0.06}
%\freetext(14.3,3.5){0}
%\freetext(15.7,3.5){1}

\edge{V35}{V24}
%\edgetext{V36}{V24}{1}
\edge{V36}{V23}
\edge{V36}{V25}
%\edgetext{V36}{V23}{0}

\edge{V37}{V24}
\edge{V37}{V26}
%\edgetext{V37}{V24}{1}

\edge{V38}{V25}
%\edgetext{V38}{V25}{1}

\bow{V24}{V36}{-0.06}%odometer
\bow{V24}{V36}{0.06}%odometer
\bow{V25}{V37}{-0.06}%odometer
\bow{V25}{V37}{0.06}%odometer
\bow{V26}{V38}{-0.06}%odometer
\bow{V26}{V38}{0.06}%odometer

\end{graph}
\caption{The Vershik map is a homeomorphism: after telescoping, all minimal and maximal edges are slanted from right to left.}\label{Fig:tel_noXminnoXmax}
\end{center}
\end{figure}

\end{example}

\begin{remark} Note that if every vertex of a generalized Bratteli diagram has an outgoing minimal and an outgoing maximal edge then the sets of infinite minimal paths and infinite maximal paths are non-empty. The diagram in Example \ref{Ex:NoXminXmax2} has the property that every vertex has at least one extreme (minimal or maximal) outgoing edge, but this does not guarantee the existence of infinite maximal or infinite minimal paths. Moreover, the aforementioned property  does not hold after telescoping with respect to even levels. After such a telescoping, the first vertex does not have any outgoing minimal or maximal edges.
\end{remark}

\begin{example}
The diagram shown in Figure~\ref{Fig:notsatThem_emptyXminXmax} does not satisfy conditions of Theorem~\ref{Thm:suff_cond_no_x_min_max} for any $N$, but still, there is an order for which there are no infinite minimal and infinite maximal paths. For every vertex $w \in V\setminus V_0$, if there are two incoming edges that are slanted from right to left, we enumerate them as minimal and maximal. Otherwise, we enumerate two vertical incoming edges as minimal and maximal. The remaining edges can be enumerated in arbitrary ways, for instance, from left to right. Then every minimal and maximal edge is either vertical or slanted from right to left and any minimal or maximal path eventually goes only through edges slanted from right to left.

This example shows that the assumption made in Theorem 
\ref{Thm:suff_cond_no_x_min_max} is not necessary.

%\tcr{For every vertex eventually outgoing minimal and maximal edges are going to the left. It's not enough (from first vertex diagonal path)}

     \begin{figure}[ht]
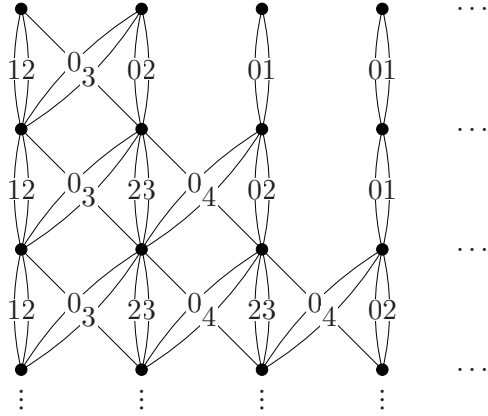

\unitlength = 0.4cm
\begin{center}
\begin{graph}(21,14)
\graphnodesize{0.4}

% Vertices of zero level
\roundnode{V01}(3,13)
\roundnode{V02}(7,13)
\roundnode{V03}(11,13)
\roundnode{V04}(15,13)
%\roundnode{V05}(19,13)

% Vertices of the first level
\roundnode{V11}(3,9)

\roundnode{V12}(7,9)
%\nodetext{V12}(-0.6,0.4){$v$}
\roundnode{V13}(11,9)
%\nodetext{V13}(0.6,0.4){$w$}
\roundnode{V14}(15,9)
%\roundnode{V15}(19,9)

%Edges zero level
\bow{V11}{V01}{-0.06}%odometer
\bowtext{V11}{V01}{-0.06}{$2$}
\bow{V11}{V01}{0.06}%odometer
\bowtext{V11}{V01}{0.06}{$1$}
\bow{V12}{V02}{-0.06}%odometer
\bowtext{V12}{V02}{-0.06}{$2$}
\bow{V12}{V02}{0.06}%odometer
\bowtext{V12}{V02}{0.06}{$0$}
\bow{V13}{V03}{-0.06}%odometer
\bowtext{V13}{V03}{-0.06}{$1$}
\bow{V13}{V03}{0.06}%odometer
\bowtext{V13}{V03}{0.06}{$0$}
\bow{V14}{V04}{-0.06}%odometer
\bowtext{V14}{V04}{-0.06}{$1$}
\bow{V14}{V04}{0.06}%odometer
\bowtext{V14}{V04}{0.06}{$0$}
%\bow{V15}{V05}{-0.06}%odometer
%\bow{V15}{V05}{0.06}%odometer

\edge{V12}{V01}
\bow{V11}{V02}{0.06}
\bowtext{V11}{V02}{0.06}{$0$}
\bow{V11}{V02}{-0.06}
\bowtext{V11}{V02}{-0.06}{$3$}

%Vertices of the second level
\roundnode{V21}(3,5)
%\freetext(2.4,6.5){2}
%\freetext(3.6,6.5){1}
%\freetext(4,5.2){0}

\roundnode{V22}(7,5)%odometer
%\freetext(5.9,5.2){2}
%\freetext(6.4,6.5){3}
%\freetext(7.6,6.5){1}
%\freetext(8,5.2){0}

\roundnode{V23}(11,5)

\roundnode{V24}(15,5)

%\freetext(13.9,5.2){2}
%\freetext(14.4,6.5){3}
%\freetext(15.6,6.5){1}
%\freetext(16,5.2){0}

%\roundnode{V25}(19,5)
%\freetext(17.9,5.2){2}
%\freetext(18.4,6.5){3}
%\freetext(19.6,6.5){1}
%\freetext(19,6){$w$}

% Edges of the second level
\bow{V21}{V11}{-0.06}%odometer
\bowtext{V21}{V11}{-0.06}{$2$}
\bow{V21}{V11}{0.06}%odometer
\bowtext{V21}{V11}{0.06}{$1$}

\bow{V22}{V12}{-0.06}%odometer
\bowtext{V22}{V12}{-0.06}{$3$}
\bow{V22}{V12}{0.06}%odometer
\bowtext{V22}{V12}{0.06}{$2$}

\edge{V22}{V11}
\bow{V21}{V12}{0.06}
\bowtext{V21}{V12}{0.06}{$0$}
\bow{V21}{V12}{-0.06}
\bowtext{V21}{V12}{-0.06}{$3$}

\edge{V23}{V12}

\bow{V23}{V13}{-0.06}%odometer
\bowtext{V23}{V13}{-0.06}{$2$}
\bow{V23}{V13}{0.06}%odometer
\bowtext{V23}{V13}{0.06}{$0$}
\bow{V24}{V14}{-0.06}%odometer
\bowtext{V24}{V14}{-0.06}{$1$}
\bow{V24}{V14}{0.06}%odometer
\bowtext{V24}{V14}{0.06}{$0$}
%\bow{V25}{V15}{-0.06}%odometer
%\bow{V25}{V15}{0.06}%odometer
%\edgetext{V24}{V13}{1}
%\edgetext{V25}{V12}{0}
%\edge{V24}{V13}
%\edge{V25}{V14}
\bow{V22}{V13}{0.06}
\bowtext{V22}{V13}{0.06}{$0$}
\bow{V22}{V13}{-0.06}
\bowtext{V22}{V13}{-0.06}{$4$}
%\edge{V23}{V14}
%\edge{V24}{V15}

%Vertices of the third level
\roundnode{V31}(3,1)
%\freetext(2.4,2.5){1}
%\freetext(3.6,2.5){0}
%\freetext(4,1.2){2}

\roundnode{V32}(7,1)
%\freetext(5.9,1.2){2}
%\freetext(6.4,2.5){1}
%\freetext(7.6,2.5){0}
%\freetext(8,1.2){3}

\roundnode{V33}(11,1)
%\freetext(9.9,1.2){2}
%\freetext(10.4,2.5){1}
%\freetext(11.6,2.5){0}
%\freetext(12,1.2){3}

\roundnode{V34}(15,1)
%\freetext(13.9,1.2){2}
%\freetext(14.4,2.5){1}
%\freetext(15.6,2.5){0}
%\freetext(16,1.2){3}

%\roundnode{V35}(19,1)
%\freetext(17.9,1.2){2}
%\freetext(18.4,2.5){1}
%\freetext(19.6,2.5){0}

\freetext(3,0){$\vdots$}
\freetext(7,0){$\vdots$}
\freetext(11,0){$\vdots$}
\freetext(15,0){$\vdots$}
%\freetext(19,0){$\vdots$}
\freetext(18,13){$\ldots$}
\freetext(18,9){$\ldots$}
\freetext(18,5){$\ldots$}
\freetext(18,1){$\ldots$}
% Edges of the third level

\bow{V31}{V21}{-0.06}%odometer
\bowtext{V31}{V21}{-0.06}{$2$}
\bow{V31}{V21}{0.06}%odometer
\bowtext{V31}{V21}{0.06}{$1$}
\edge{V32}{V21}
\bow{V31}{V22}{0.06}
\bowtext{V31}{V22}{0.06}{$0$}
\bow{V31}{V22}{-0.06}
\bowtext{V31}{V22}{-0.06}{$3$}

\bow{V32}{V22}{-0.06}
\bowtext{V32}{V22}{-0.06}{$3$}
\bow{V32}{V22}{0.06}
\bowtext{V32}{V22}{0.06}{$2$}

\bow{V32}{V23}{0.06}
\bowtext{V32}{V23}{0.06}{$0$}
\bow{V32}{V23}{-0.06}
\bowtext{V32}{V23}{-0.06}{$4$}
\edge{V33}{V22}
\bow{V33}{V24}{0.06}
\bowtext{V33}{V24}{0.06}{$0$}
\bow{V33}{V24}{-0.06}
\bowtext{V33}{V24}{-0.06}{$4$}

\edge{V34}{V23}
%\edge{V34}{V25}

%\edge{V35}{V24}

\bow{V23}{V33}{-0.06}%odometer
\bowtext{V23}{V33}{-0.06}{$2$}
\bow{V23}{V33}{0.06}%odometer
\bowtext{V23}{V33}{0.06}{$3$}
\bow{V24}{V34}{-0.06}%odometer
\bowtext{V24}{V34}{-0.06}{$0$}
\bow{V24}{V34}{0.06}%odometer
\bowtext{V24}{V34}{0.06}{$2$}
%\bow{V25}{V35}{-0.06}%odometer
%\bow{V25}{V35}{0.06}%odometer

\end{graph}
\caption{The sets of infinite minimal and maximal paths are empty (the labels are shown only for bow-shaped edges).}\label{Fig:notsatThem_emptyXminXmax}
\end{center}
\end{figure}
 
\end{example}

% In the following remark we describe the class of generalized Bratteli diagrams, such that for any order there exist infinite minimal and infinite maximal infinite paths.

\begin{remark} Not every generalized Bratteli diagram can be endowed with an order such that there are no infinite minimal and no infinite maximal paths.
It is easy to provide examples of reducible generalized Bratteli diagrams such that for any order there exist infinite minimal and infinite maximal paths. 
We recall the definition of a \textit{vertex subdiagram} which was used before for standard Bratteli diagrams but can be also defined in the same way for generalized ones. Let $B = (V,E)$ be a generalized Bratteli diagram. Let $\overline W = \{W_n\}_{n>0}$ be a sequence of proper, non-empty subsets $W_n \subset V_n$. Set $W'_n = V_n \setminus W_n$. The (vertex) subdiagram $\overline B =  (\ol W, \ol E)$ is a (standard or generalized) Bratteli diagram defined by the vertices $\ol W = \bigcup_{i\geq 0} W_n$ and the edges $\ol E$ that have their source and range in $\ol W$. In other words, the incidence matrix $\ol F_n$ of $\ol B$ is defined by those edges from $B$ that have their source and range in vertices from $W_{n}$ and $W_{n+1}$, respectively (see e.g. \cite{BezuglyiKarpel2016}). Suppose that a path space of a generalized Bratteli diagram $B$ has a compact subset that is invariant under the tail equivalence relation $\mathcal{R}$ and is represented by a standard Bratteli subdiagram $\ov B$ of $B$. Then for any order on ${B}$, the path space $X_{\ov B}$ has at least one minimal path and one maximal path. Since $X_{\ov B}$ is invariant under $\mathcal{R}$, there are no incoming edges to vertices of $\ov B$ from the vertices that do not lie in $\ov B$. Thus, the maximal and minimal paths which lie in $X_{\ov B}$ stay maximal and minimal also in $X_B$. For instance, in Figure~\ref{Fig:AlwaysXmiXmax}, the subdiagram $\ov B$ consists of all paths which pass through the first vertex on each level. More generally, if a generalized Bratteli diagram $B$ has $n$  subdiagrams with compact $\mathcal{R}$-invariant path spaces for some $n \in \mathbb{N}$, then any order on $B$ admits at least $n$ infinite minimal and $n$ infinite maximal paths.
\end{remark}

    \begin{figure}[hbt!]
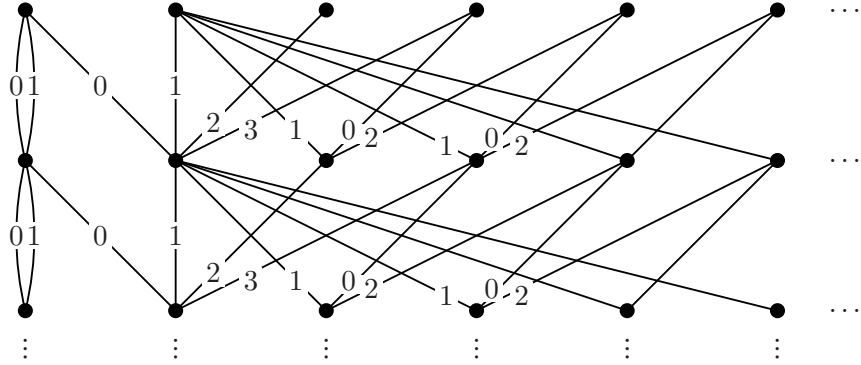

\unitlength=1cm
\begin{graph}(11,6)
% \graphnodesize{0.2}
% \roundnode{V0}(3,6)
%  %\nodetext{V0}(-1,0){$V_0$}
\roundnode{V10}(0,5)
 \roundnode{V11}(2,5)
 %\nodetext{V11}(-0.7,0){$w_1^{(0)}$}
  %\nodetext{V12}(0.7,0){$w_2^{(0)}$}
 \roundnode{V12}(4,5)
 \roundnode{V13}(6,5)
 \roundnode{V14}(8,5)
  \roundnode{V15}(10,5)

  %
    % The second level vertices
    \roundnode{V20}(0,3)
   \roundnode{V21}(2,3)
 %\nodetext{V11}(-0.7,0){$w_1^{(0)}$}
  %\nodetext{V12}(0.7,0){$w_2^{(0)}$}
 \roundnode{V22}(4,3)
 \roundnode{V23}(6,3)
 \roundnode{V24}(8,3)
  \roundnode{V25}(10,3)

  % The third level vertices
  \roundnode{V30}(0,1)
 \roundnode{V31}(2,1)
 \roundnode{V32}(4,1)
 \roundnode{V33}(6,1)
  \roundnode{V34}(8,1)
    \roundnode{V35}(10,1)
  
 %
 % EDGES
 \graphlinewidth{0.025}
% % First level
%  \edge{V0}{V11}
%  \edge{V0}{V12}

\bow{V20}{V10}{0.06}
\bow{V20}{V10}{-0.06}
\bowtext{V20}{V10}{0.06}{$0$}
\bowtext{V20}{V10}{-0.06}{$1$}
\edge{V21}{V10}
\edgetext{V21}{V10}{$0$}

    \edge{V21}{V11}
    \edgetext{V21}{V11}{$1$}
    \edge{V22}{V11}
        \freetext(3.6,3.4){$1$}
    \edge{V21}{V12}
    \freetext(2.5,3.5){$2$}
    \edge{V21}{V13}
    \freetext(3,3.4){$3$}

 \edge{V23}{V11}
        \freetext(5.6,3.2){$1$}
 \edge{V22}{V13}
     \freetext(4.3,3.4){$0$}
  \edge{V22}{V14}
       \freetext(4.6,3.3){$2$}
    
   \edge{V24}{V11}
    \edge{V23}{V14}
         \freetext(6.2,3.3){$0$}
       \edge{V23}{V15}
             \freetext(6.6,3.2){$2$}
    
   \edge{V25}{V11}
    \edge{V24}{V15}
    
 % Second level
 
\bow{V30}{V20}{0.06}
\bow{V30}{V20}{-0.06}
\bowtext{V30}{V20}{0.06}{$0$}
\bowtext{V30}{V20}{-0.06}{$1$}
\edge{V31}{V20}
\edgetext{V31}{V20}{$0$}

    \edge{V31}{V21}
    \edgetext{V31}{V21}{$1$}
    \edge{V32}{V21}
       \freetext(3.6,1.4){$1$}
    \edge{V31}{V22}
        \freetext(2.5,1.5){$2$}
    \edge{V31}{V23}
        \freetext(3,1.4){$3$}

 \edge{V33}{V21}
        \freetext(5.6,1.2){$1$}
 \edge{V32}{V23}
    \freetext(4.3,1.4){$0$}
  \edge{V32}{V24}
        \freetext(4.6,1.3){$2$}
    
   \edge{V34}{V21}
    \edge{V33}{V24}
    \freetext(6.2,1.3){$0$}
       \edge{V33}{V25}
        \freetext(6.6,1.2){$2$}
    
   \edge{V35}{V21}
    \edge{V34}{V25}

  \freetext(10.9,5){$\ldots$}
 \freetext(10.9,3){$\ldots$}
  \freetext(10.9,1){$\ldots$}
  
  \freetext(0,0.5){$\vdots$}
  \freetext(2,0.5){$\vdots$}
    \freetext(4,0.5){$\vdots$}
      \freetext(6,0.5){$\vdots$}  
        \freetext(8,0.5){$\vdots$} 
        \freetext(10,0.5){$\vdots$} 

          %\freetext(-0.4,4){$1$} 
          %\freetext(0.3,4.3){$1$} 
           % \freetext(0.3,3.6){$0$} 
            %\freetext(0.7,3.2){$2$}
             %\freetext(2.2,3.5){$0$} 
           % \freetext(2.6,3.2){$3$}
           %  \freetext(4.2,3.5){$0$} 
           % \freetext(4.6,3.2){$2$}
            %\freetext(0.3,1.6){$0$} 
            %\freetext(0.7,1.2){$2$}
            % \freetext(2.2,1.5){$0$} 
            %\freetext(2.6,1.2){$3$}
             %\freetext(4.2,1.5){$0$} 
           % \freetext(4.6,1.2){$2$}
          %\freetext(1,5){$1$} 
\end{graph}
\caption{For the stationary order presented above, the Vershik map is a continuous bijection, but its inverse is discontinuous. For any order, there is at least one infinite minimal and one infinite maximal paths.}\label{Fig:AlwaysXmiXmax}
\end{figure}

\subsection{Prolongation of the Vershik map.}

The goal of this subsection is to emphasize a sharp difference in the properties of Vershik maps for generalized
and standard Bratteli diagrams. 

\begin{theorem}
There are stationary ordered generalized Bratteli diagrams with a unique infinite minimal and a unique infinite maximal paths such that

\medskip
(i) both the Vershik map $\varphi_B$ and its inverse  $\varphi_B^{-1}$ are not continuous;

\medskip
(ii) the Vershik map $\varphi_B$ is continuous but 
the inverse $\varphi_B^{-1}$ is discontinuous;

\medskip
(iii) both the Vershik map $\varphi_B$ and its inverse  $\varphi_B^{-1}$ are continuous.
    
\end{theorem}

The proof of this theorem is given in Examples 
\ref{ex:discontinuous map}, 
\ref{ex:discontinuous inverse} and Remark \ref{rem:cont VM}. 

We recall that for standard Bratteli diagrams, it is obvious that if a diagram has a unique minimal infinite path $x_{min}$ and a unique maximal infinite path $x_{max}$, then the Vershik map which sends $x_{max}$ to $x_{min}$ is a homeomorphism. This fact follows from compactness of the path space. In the example below we show that the result does not hold for generalized Bratteli diagrams.

\begin{example}\label{ex:discontinuous map}
The stationary diagram in Figure~\ref{Fig:UniqueXmiXmaxNotHomeo} has a unique infinite minimal path $x_{min}$ and a unique infinite maximal path $x_{max}$ such that they pass through the first vertex on each level of the diagram. The second vertex of the diagram has no minimal or maximal outgoing edges and all maximal and minimal edges which start not at the first vertex are slanted from right to left, which guarantees that there are no other infinite minimal and maximal paths (see also Example~\ref{Ex:noXminnoXmax}).
This means that the corresponding Vershik map $\varphi_B$
sends $x_{max}$ to $x_{min}$ with necessity. We claim that 
$\varphi_B$ is not a homeomorphism of $X_B$ in this case.
Indeed, let $x$ be a non-maximal path that coincides with $x_{max}$ for exactly $n$ first edges, and then passes through a non-maximal edge. Then the image of $x$ under the Vershik map $\varphi_B$ lies in the cylinder set corresponding to the minimal path of length $n$ which is slanted from right to left and ends in the second vertex of the diagram. Thus, the images of the non-maximal paths which lie in the neighborhood of the unique maximal path are not mapped into the neighborhood of the unique minimal path. 
Hence the Vershik map $\varphi_B$ is not continuous on $X_B$. Similarly, $\varphi_B^{-1}$ is discontinuous.

    \begin{figure}[hbt!]
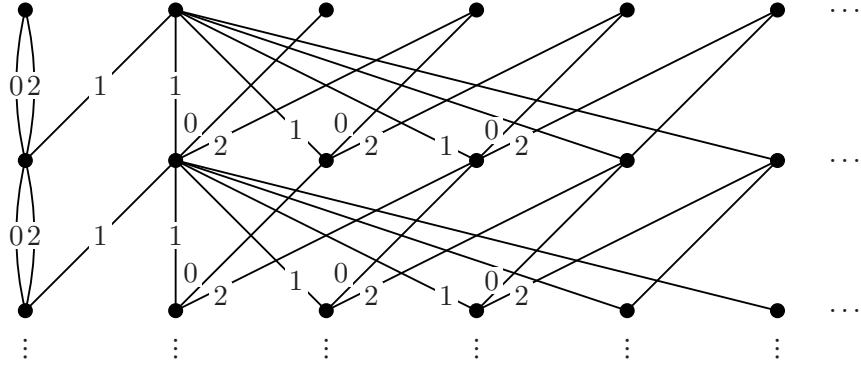

\unitlength=1cm
\begin{graph}(11,6)
% \graphnodesize{0.2}
% \roundnode{V0}(3,6)
%  %\nodetext{V0}(-1,0){$V_0$}
\roundnode{V10}(0,5)
 \roundnode{V11}(2,5)
 %\nodetext{V11}(-0.7,0){$w_1^{(0)}$}
  %\nodetext{V12}(0.7,0){$w_2^{(0)}$}
 \roundnode{V12}(4,5)
 \roundnode{V13}(6,5)
 \roundnode{V14}(8,5)
  \roundnode{V15}(10,5)

  %
    % The second level vertices
    \roundnode{V20}(0,3)
   \roundnode{V21}(2,3)
 %\nodetext{V11}(-0.7,0){$w_1^{(0)}$}
  %\nodetext{V12}(0.7,0){$w_2^{(0)}$}
 \roundnode{V22}(4,3)
 \roundnode{V23}(6,3)
 \roundnode{V24}(8,3)
  \roundnode{V25}(10,3)

  % The third level vertices
  \roundnode{V30}(0,1)
 \roundnode{V31}(2,1)
 \roundnode{V32}(4,1)
 \roundnode{V33}(6,1)
  \roundnode{V34}(8,1)
    \roundnode{V35}(10,1)
  
 %
 % EDGES
 \graphlinewidth{0.025}
% % First level
%  \edge{V0}{V11}
%  \edge{V0}{V12}

\bow{V20}{V10}{0.06}
\bow{V20}{V10}{-0.06}
\bowtext{V20}{V10}{0.06}{$0$}
\bowtext{V20}{V10}{-0.06}{$2$}
\edge{V20}{V11}
\edgetext{V20}{V11}{$1$}

    \edge{V21}{V11}
    \edgetext{V21}{V11}{$1$}
    \edge{V22}{V11}
    \edge{V21}{V12}
    \edge{V21}{V13}

 \edge{V23}{V11}
 \edge{V22}{V13}
  \edge{V22}{V14}
    
   \edge{V24}{V11}
    \edge{V23}{V14}
       \edge{V23}{V15}
    
   \edge{V25}{V11}
    \edge{V24}{V15}
    
 % Second level
 
\bow{V30}{V20}{0.06}
\bow{V30}{V20}{-0.06}
\bowtext{V30}{V20}{0.06}{$0$}
\bowtext{V30}{V20}{-0.06}{$2$}
\edge{V30}{V21}
\edgetext{V30}{V21}{$1$}

    \edge{V31}{V21}
    \edgetext{V31}{V21}{$1$}
    \edge{V32}{V21}
    \edge{V31}{V22}
    \edge{V31}{V23}

 \edge{V33}{V21}
 \edge{V32}{V23}
  \edge{V32}{V24}
    
   \edge{V34}{V21}
    \edge{V33}{V24}
       \edge{V33}{V25}
    
   \edge{V35}{V21}
    \edge{V34}{V25}

  \freetext(10.9,5){$\ldots$}
 \freetext(10.9,3){$\ldots$}
  \freetext(10.9,1){$\ldots$}
  
  \freetext(0,0.5){$\vdots$}
  \freetext(2,0.5){$\vdots$}
    \freetext(4,0.5){$\vdots$}
      \freetext(6,0.5){$\vdots$}  
        \freetext(8,0.5){$\vdots$} 
        \freetext(10,0.5){$\vdots$} 

             \freetext(2.2,3.5){$0$} 
            \freetext(2.6,3.2){$2$}
             \freetext(4.2,3.5){$0$} 
            \freetext(4.6,3.2){$2$}
                  \freetext(6.2,3.4){$0$} 
            \freetext(6.6,3.2){$2$}
            \freetext(6.2,1.4){$0$} 
            \freetext(6.6,1.2){$2$}
             \freetext(2.2,1.5){$0$} 
            \freetext(2.6,1.2){$2$}
             \freetext(4.2,1.5){$0$} 
            \freetext(4.6,1.2){$2$}

            \freetext(3.6,3.4){$1$}
            \freetext(3.6,1.4){$1$}
            \freetext(5.6,3.2){$1$}
            \freetext(5.6,1.2){$1$}

\end{graph}
\caption{There is a unique infinite minimal and a unique infinite maximal path, Vershik map $\varphi_B$ is a Borel bijection, both $\varphi_B$ and $\varphi_B^{-1}$ are discontinuous.}\label{Fig:UniqueXmiXmaxNotHomeo}
\end{figure}
\end{example}

The following remark describes a class of ordered generalized Bratteli diagrams with a unique infinite minimal path and a unique infinite maximal path such that the corresponding Vershik map is a homeomorphism.

\begin{remark}\label{rem:cont VM}
Similar to the case of standard Bratteli diagrams, 
it is easy to 
see that if for every level $n$ of a generalized Bratteli diagram $B$ there is a unique 
vertex $v_{min}^{(n)}$ such that all minimal edges of $E_n$ start at $v_{min}^{(n)}$, and a unique 
vertex $v_{max}^{(n)}$ such that all maximal edges of $E_n$ start at $v_{max}^{(n)}$,
then there is a unique infinite minimal path $x_{min}$, a unique infinite maximal path $x_{max}$, and the Vershik map which maps $x_{max}$ to $x_{min}$ is a homeomorphism of $X_B$.
\end{remark}

The example below presents a stationary ordered generalized Bratteli diagram $B$ with a unique infinite minimal and a unique infinite maximal path such that the Vershik map $\varphi_B$ is a continuous bijection, but $\varphi^{-1}_B$ is discontinuous.

\begin{example}\label{ex:discontinuous inverse}
The stationary ordered diagram on Figure~\ref{Fig:AlwaysXmiXmax} has a unique infinite minimal path $x_{min}$ and a unique infinite maximal path $x_{max}$ passing through the first vertex on each level of the diagram. All minimal and maximal edges which do not end in the first vertex of the diagram are slanted from right to left, there are no outgoing minimal or maximal edges from the second vertex of the diagram. Thus, the infinite minimal and infinite maximal paths are unique. We define the Vershik map $\varphi_B$ on these paths by setting $\varphi_B(x_{max}) = x_{min}$. Then it is easy to see that $\varphi_B$ is a continuous bijection of $X_B$. We show that  $\varphi_B^{-1}$ is not continuous. Indeed, any non-maximal path from a neighborhood of $x_{max}$ is mapped to a neighborhood of $x_{min}$. Now, consider a path $x$ which first pass through minimal edges and the first vertex of the diagram for $n$ levels, but then goes once along the minimal edge to the second vertex and then along the edge $e$ enumerated by $1$ to the third vertex. We see that $x$ lies in a neighborhood of $x_{min}$, but this path $x$ is not mapped to a neighborhood of $x_{max}$ by $\varphi^{-1}_B$. Indeed, the edge $e$ has a predecessor which is the edge 
$e'$ labeled by $0$ slanted from right to left. The source of $e'$ is the fourth vertex of the diagram, and it is joined with $V_0$ by a finite maximal path slanted from right to left. Thus, the preimage of $x$ does not belong to a neighborhood of $x_{max}$.
\end{example}

Recall that some results concerning the prolongation of a Vershik map for generalized Bratteli diagrams of bounded size can be also
found in Section~\ref{sect Bndd size}.

%%%%%Section 5

\section{Topological transitivity of the tail equivalence relation}\label{sec: topological property}
In this section, we prove that the tail equivalence relation on the 
path space of a stationary generalized Bratteli diagram 
with an irreducible and aperiodic incidence matrix is 
topologically 
transitive. We also show that the irreducibility of the diagram does not imply the tail equivalence relation is minimal (see Theorem \ref{thm:non_minimal}). 

\begin{thm}\label{thm_top_trans}
Let $B = (V,  E)$ be a generalized stationary Bratteli diagram 
with an 
irreducible aperiodic incidence matrix $F = (f_{ij})_{i,j \in 
\mathbb{Z}}$. Then the tail equivalence relation $\mathcal{R}$ is 
topologically transitive.
\end{thm}

\begin{proof}
The idea of the proof is to show that the diagram has 
``vertical'' paths with dense orbits in the path space. We identify 
vertices at each level with $\Z$. Fix a vertex $i \in \mathbb{Z}$. 
Since $F$ is irreducible, there exists $k \in \N$ such that 
$f^{(k)}_{ii} > 0$. In other words, there is an infinite path $x$ 
in $X_B$ which passes through the vertex $i$ on 
levels $s \cdot k$ for all $s \in \mathbb{N}$.
We will show that $x$ is a transitive point, i.e.,  
$$
\ov{[x]}_{\mathcal{R}} = X_B.
$$ 
Indeed, it is enough to show that the tail equivalence class of 
$x$ intersects every cylinder set. Let $[\ov e] = (e_0, \ldots, 
e_{N-1})$ be an arbitrary cylinder set. Denote $j = r(e_{N-1}) \in 
V_{N}$, we show that for $s \in \N$ large enough there is a path 
between $i \in V_{sk}$ and $j \in V_{N}$.
Since the matrix $F$ is irreducible, there exists $t \in \mathbb{N}$
such that $f^{(t)}_{ij} > 0$. By Lemma~\ref{Lemma_Perron_value}, 
there exists $l = l(j) \in \N$ such that 
$$
f^{(m)}_{jj} > 0 \mbox{ for all } m \geq l.
$$

Note that there exist  $s \in \N$ and an integer $M \geq l$ 
such that there is a 
path of length $t$ between $i \in V_{sk}$ and $j \in V_{sk-t}$ and 
a path of length $M$ between $j \in V_{sk-t}$ and $j \in V_{N}$ 
(see Figure~\ref{Fig:Transitive_R}).
To see this, choose $s$ such that 
$sk - t - N > l$. Take $M = sk - t - N$. 
Hence, we obtain
$$
f^{(M + t)}_{ij} \geq f^{(t)}_{ij}\, f^{(M)}_{jj} > 0.
$$

Thus, there exists a path of length $M+t$ between vertex 
$j = r(e_N) 
\in V_{N}$ and vertex $i$ on level $V_{sk}$. Therefore, 
$[x]_{\mathcal{R}}$ and $[\ov e]$ have a non-empty intersection.

\begin{figure}
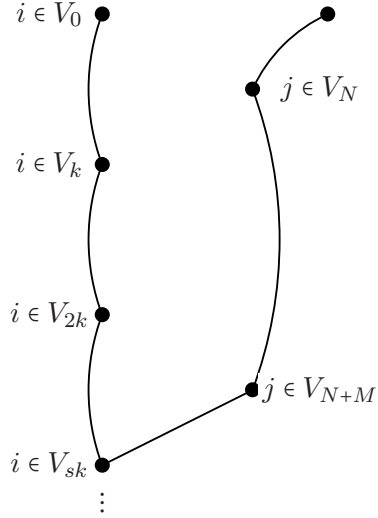

\unitlength=1cm
\begin{graph}(7,8)
% \graphnodesize{0.2}
% \roundnode{V0}(3,6)
%  %\nodetext{V0}(-1,0){$V_0$}
 \roundnode{V11}(2,7)
 \nodetext{V11}(-0.7,0){$i \in V_0$}
  %\nodetext{V12}(0.7,0){$w_2^{(0)}$}
 %\roundnode{V12}(4,7)
 \roundnode{V13}(5,7)
  % The second level vertices
 \roundnode{V21}(2,5)
 \roundnode{V22}(4,6)
  \nodetext{V22}(0.9,0){$j \in V_N$}
%\roundnode{V23}(6,5)
  \nodetext{V21}(-0.7,0){$i \in V_k$}
 % \nodetext{V22}(0.7,0){$w_2^{(1)}$}
 % The third level vertices
 \roundnode{V31}(2,3)
 \roundnode{V32}(4,2)
 %\roundnode{V33}(6,3)
 \nodetext{V31}(-0.7,0){$i \in V_{2k}$}
 \nodetext{V32}(0.9,0){$j \in V_{N + M}$}
  %\nodetext{V32}(0.7,0){$w_2^{(2)}$}
   % The fourth level vertices
    \roundnode{V41}(2,1)
     \nodetext{V41}(-0.7,0){$i \in V_{sk}$}
   
 %
 % EDGES
 \graphlinewidth{0.025}
% % First level
%  \edge{V0}{V11}
%  \edge{V0}{V12}

 % Second level
 \bow{V21}{V11}{0.09}

\bow{V22}{V13}{0.09}

 %third level
 \bow{V31}{V21}{0.09}
 \bow{V32}{V22}{-0.09}

  %fourth level
  
 \bow{V41}{V31}{0.09}
 \edge{V41}{V32}
 
  %\freetext(6.9,5){$\ldots$}
 %\freetext(6.9,3){$\ldots$}
 % \freetext(6.9,1){$\ldots$}
  \freetext(2,0.5){$\vdots$}
  %\freetext(4,0.5){$\vdots$}
   % \freetext(6,0.5){$\vdots$}
\end{graph}
\caption{Illustration to the proof of Theorem~\ref{thm_top_trans}. 
A diagram with ``vertical'' paths resulting in dense orbits in the 
path space.}\label{Fig:Transitive_R}
\end{figure}

\end{proof}

\begin{remark} Theorem \ref{thm_top_trans} focuses on 
\textit{stationary} diagrams. In this remark, we consider 
a structural property of the diagram which implies the transitivity 
of the tail equivalence relation for non-stationary diagrams as 
well. 
Let $B(V, E)$ be an irreducible generalized Bratteli diagram with 
a vertex $i \in \Z$ (recall that we identify vertices at each level 
with integers) such that for every level $n$ there 
is an edge between $i \in V_n$ and $i \in V_{n+1}$. In other words, 
there exists an infinite vertical path $x \in X_B$ which passes 
through vertex $i$ on each level of the diagram. Then, it can be
easily proved (in the same manner as in the proof of 
Theorem~\ref{thm_top_trans})
that the orbit $[x]_{\mathcal R}$ is dense in $X_B$, and hence 
the tail equivalence relation $\mathcal{R}$ is topologically 
transitive. More generally, it suffices to assume that there exists
a sequence $(n_k)$ such that there is a path between the vertex
$i \in V_{n_k}$ and $i \in V_{n_{k+1}}$ for all $k$. Then after 
telescoping, we can use the above result.
\end{remark} 

It is natural to ask if there exist conditions on the structure of 
a generalized Bratteli diagram that would imply the minimality of 
the tail equivalence relation. To this effect, we show (see 
Theorem~\ref{thm:non_minimal}) that bounded size diagrams (see 
Definition~\ref{Def:BD_bdd_size}) with irreducible incidence 
matrices contain proper closed subsets that are invariant under the 
tail equivalence relation and nowhere dense in the path space. In 
particular, this shows that the tail equivalence relation is 
\textit{not minimal}. Section \ref{Sect:contin_V_map} provides an example of a minimal Vershik map on an ordered generalized Bratteli diagram which has no infinite minimal and no infinite maximal paths (see Example \ref{Ex:noXminnoXmax}). Since the orbits of the Vershik map and the tail equivalence relation are the same for such a diagram, we obtain an example of a stationary generalized Bratteli diagram with an irreducible aperiodic incidence matrix such that the corresponding tail equivalence relation is minimal.

\medskip

Let $B = (V,E) $ be a generalized Bratteli diagram of bounded size, 
with the corresponding sequence 
$(t_n, L_n)_{n \in \N_0}$. For $w \in V_0$ we define 
$$
Z_w^+ = \left\{x = (x_n) \in X_B : s(x_0) \geq w \mbox{ and } 
r(x_n) \geq w + \sum_{i = 0}^{n} t_i \mbox{ for } n \in 
\mathbb{N}_0\right\}.
$$
Similarly, for $w \in V_0$ define
$$
Z_w^- = \left\{x = (x_n) \in X_B : s(x_0) \leq w \mbox{ and } 
r(x_n) \leq w - \sum_{i = 0}^{n} t_i \mbox{ for } n \in 
\mathbb{N}_0\right\}.
$$

\begin{lemma}\label{Lemma_slanting_sets_invar} Let $B = (V,E)$ 
be a generalized Bratteli diagram of bounded size, with the 
corresponding sequence $(t_n, L_n)_{n \in \N_0}$. Then,
for every $w \in V_0$, the sets $Z_w^+$, $Z^-_w$ are invariant with 
respect to the tail equivalence relation $\mathcal{R}$.
\end{lemma}

\begin{proof}
Fix $x = (x_n) \in Z_w^+$ and consider an infinite path $y = 
(y_n)\in X_B$ which is tail equivalent to $x$. Thus there exists 
$n \in \mathbb{N}$ such that $r(x_n) = 
r(y_n) = v$ for some $v \in V_{n+1}$. Since $x \in Z_w^+$, we have 
$$
v \geq w + \sum_{i = 0}^{n} t_i.
$$
By Corollary~\ref{corol_complete_upper_cone}, 
$$
s(y_m) \in \left[v - \sum_{i = m}^{n} t_i, v + \sum_{i = m}^n 
t_i\right] \subset \left[w + \sum_{i = 1}^{m-1} t_i, \infty \right)
$$
for all $m \leq n$. 
In other words, 
$$
r(y_{m}) = s(y_{m+1}) \in \left[w + \sum_{i = 1}^{m} t_i, \infty 
\right).
$$
Since $x \in Z_w^+$ and $x$ and $y$ are tail equivalent, we also 
have
$$
r(y_k) \in \left[w + \sum_{i = 1}^{k} t_i, \infty \right)
$$
for all $k \geq n$. 
Thus, $y \in  Z_w^+$, this shows that $Z^+_w$ is invariant with 
respect to the tail equivalence relation 
$\mathcal{R}$. A similar argument shows that $Z^-_w$ is also 
invariant with respect to the tail equivalence relation 
$\mathcal{R}$. 
\end{proof} 

We will call the sets $Z^+_w$, $Z^-_w$ \textit{slanting sets} for 
$w \in V_0$.

\begin{thm}\label{thm:non_minimal} Let $B = (V,E) $ be a generalized Bratteli diagram of bounded size, with the 
corresponding sequence $(t_n, L_n)_{n \in \N_0}$. 
Then, for every $w \in V_0$, the sets $Z^+_w$, $Z^-_w$ are closed 
nowhere dense sets with respect to the topology generated by 
cylinder sets. In particular, this shows (together with Lemma 
\ref{Lemma_slanting_sets_invar}) that the tail equivalence relation 
$\mathcal{R}$ is \textit{not minimal}.
\end{thm}

\begin{proof}
First, we prove that the set $Z^+_w$ is closed. Let $Z_0$ be a 
union of cylinder sets, on level $V_0$ which correspond to the 
vertices in the interval $[w, \infty) \subset V_0$. Let $Z_n$ be a union of all cylinder sets 
corresponding to finite paths of length $n$ which lie in $Z^+_w$. 
Then we have
$$
Z_0 \supset Z_1 \supset \ldots \supset Z_n  \supset 
\ldots,
$$
each $Z_n$ is closed and 
$$
Z^+_w = \bigcap_{n = 0}^{\infty} Z_n.
$$
Hence $Z^+_w$ is closed.

We show that $Z^+_w$ does not contain any cylinder set. 
Let $\ov e = (e_0, \ldots, e_n)$ be a finite path which lies in 
$Z^+_w$ and $v = r(e_n)  \in V_{n+1}$.
There exists $m$ such that $v - \sum_{i = 1}^{m} t_{n + i} < w + \sum_{i = 1}^{m} t_{n + i}$. Since $E(u - t_n, u) \neq \emptyset$ for all $n$ and for all $u \in V_{n+1}$, there is a finite path $(e_{n+1}, \ldots, e_{n+m})$ between $v \in V_{n+1}$ and $v - \sum_{i = 1}^{m} t_{n + i} \in V_{n+m}$. Hence the cylinder set generated by the path 
$(e_0, \ldots, 
e_{n+m})$ is a subset of $[\ov e]$ which does not belong to 
$Z^+_w$. Thus, $Z^+_w$ has empty interior. 
Since $Z^+_w$ is closed, it follows that $Z^+_w$ is nowhere dense. 
A similar argument shows that $Z^-_w$ is also closed and nowhere 
dense.
\end{proof}

To end this section, we discuss the cardinality  of the sets of 
the form $Z^+_w$ and $Z^-_w$ for a bounded size generalized 
Bratteli diagram and provide conditions that guarantee that 
$Z^+_w$ and $Z^-_w$ are countable sets. 

Let $B=(V,E)$ be a bounded size generalized Bratteli diagram. For 
every $w \in V_0$, let $Y_w^+$ denote the set 
of all infinite paths which start at $w$ and then pass through 
the rightmost possible vertex on each level, i.e. for every 
$m \in \mathbb{N}$, the paths from $Y_w^+$ go through the vertex 
$w + \sum_{i = 0}^{m-1} t_i$ on level $m$. Obviously, we have 
$Y_w^+ \subset Z_w^+$ for all $w \in V_0$. Analogously, let 
$Y_w^- \subset Z_w^-$ be the set of all infinite paths which 
start at $w$ and then pass through the leftmost possible vertex 
on each level. Note that the sets $Y_w^+$, $Y_w^-$ can be either 
finite or uncountable (then they are odometers). We can say that 
$Y_w^+$, $ Y_w^-$ are the ``boundary'' paths for $Z_w^+$,
$Z_w^-$.

\begin{prop} Let $B = (V,E) $ be a generalized Bratteli diagram of 
bounded size, with the corresponding sequence 
$(t_n, L_n)_{n \in \N_0}$. If the sets $Y_w^+$ are 
finite for all $w \in V_0$ then the sets $Z^+_w$, are 
countable for all $w \in V_0$. Otherwise, there exists an 
uncountable set $Z^+_w$. 
The same is true for $Y_w^-$ and $Z^-_w$.
In particular, if $|E(v + t_n, v)| = |E(v - t_n, v)| = 1$ for all 
$v \in V_{n+1}$ and $n\in \N$, then sets $Z^+_w$, $Z^-_w$ are 
countable for all $w \in V_0$.
\end{prop}

\begin{proof}
Let $Y_w^+$ be finite for all $w \in V_0$. Fix any $w \in V_0$.
Suppose $y = (y_n)_{n = 0}^{\infty} \in Z^+_w$ and $s(y_0) = u 
\geq w$. Since $B$ is of bounded size, for every $m \in 
\mathbb{N}_0$:
$$
r(y_m) = u + \tl t_0 + \ldots + \tl t_{m},
$$
where $\tl t_i \in [-t_i, t_i]$ for $i = 1, \ldots, m$.
Since $y \in Z^+_w$, we have
$$
r(y_m) = u + \tl t_0 + \ldots + \tl t_{m} \in 
\left[w + \sum_{i = 0}^{m} t_i, \infty \right).
$$
Thus, for all $m \in \mathbb{N}_0$,
$$
u + \sum_{i = 0}^{m} \tl t_i \geq w + \sum_{i = 0}^{m} t_i
$$
and
$$
u - w \geq \sum_{i = 0}^{m} (t_i - \tl t_i).
$$
Recall that $t_i \geq \tl t_i$, hence $t_i - \tl t_i \geq 0$. 
Thus, we have
$$
u - w \geq \sum_{i = 0}^{\infty} (t_i - \tl t_i),
$$
which is possible only if there are finitely many non-zero 
elements among $(t_i - \tl t_i)$. Hence, the path $y$ should  go 
in the same 
direction as the ``boundary'' paths $Y_w^{+}$, except for 
finitely many deviations. Since all the sets $Y_w^{+}$ are 
finite, the number of such paths $y$ is countable. If for some $w 
\in V_0$ the set $Y_w^+$ is uncountable, then the set 
$Z_w^+ \supset Y_w^+$ is also uncountable. The same proof works 
for $Y_w^-$ and $Z_w^{-}$.
\end{proof}

%%%%Section 5

\section{Tail-invariant measures for generalized Bratteli 
diagrams}\label{Sec:Tail_inv_mes} 
In this section, we discuss \textit{tail-invariant measures}
on the path space of a generalized Bratteli diagram. We emphasize 
that in this paper, the term \textit{measures} is used for 
non-atomic positive Borel measures. Moreover, we are mostly 
interested in \textit{full measures}, i.e., every cylinder set 
must be of positive measure. We consider both finite 
(probability) and $\sigma$-finite measures. In the case of 
$\sigma$-finite measures, we are interested in only those measures 
which take finite values on cylinder sets. 

We describe every tail-invariant measure in terms of a sequence of 
positive vectors associated with vertices of each level, see 
Theorem  \ref{BKMS_measures=invlimits}. We give an explicit 
construction of ordered generalized Bratteli diagrams for which there 
exists no full probability measure invariant under the Vershik map, and the restriction
of the tail equivalence relation onto the equivalence class of some non-empty clopen set is
compressible
 (see Theorem 
\ref{Thm:no_full_inv_prob_mu}). We also provide a class of 
generalized Bratteli diagrams such that there exists no tail 
invariant (finite or $\sigma$-finite) measure with  finite values 
on cylinder sets (see Proposition 
\ref{Prop:no_meas_fin_cyl_sets}).

\begin{definition}\label{def: tail inv meas} Let $B =(V, E)$ be 
a generalized Bratteli diagram and $\mathcal R$ 
the tail equivalence relation on the path space $X_B$ (see 
Definition \ref{Def:Tail_equiv_relation}). A measure $\mu$ on 
$X_B$ is called \textit{tail-invariant} if, for any cylinder sets
$[\ol e]$ and $[\ol e']$ such that $r(\ol e) = r(\ol e')$, we have
$\mu([\ol e]) = \mu([\ol e'])$.
\end{definition}

\begin{remark}
The theory of countable Borel equivalence relations is a key object
in Borel dynamics; it has been considered from different 
points of view in numerous books and articles, see, e.g., 
\cite{Gao2009}, \cite{GaoJackson2015}, 
\cite{Dougherty_Jackson_Kechris1994},
\cite{JacksonKechrisLouveau2002},
\cite{KechrisMiller2004} and the literature within. For any
generalized Bratteli diagram $B$, the tail equivalence relation 
$\mathcal R$ is 
a countable Borel hyperfinite equivalence relation. This means
that there exists a Borel automorphism $T : X_B \to X_B$ whose 
orbits coincide with the orbits of $\mathcal R$. Can we take 
a Vershik map $\varphi_B$ for $T$? First, we note that 
the set of tail-invariant measures does not depend on an order on
$B$. Second, the orbits of a Vershik map and the
tail-invariant relation differ at the sets of maximal and minimal 
paths. In general, every measure $\mu$ that is invariant with respect to a Vershik map is also tail-invariant. Hence, if the sets of maximal and minimal paths have zero
measure, then we can identify tail-invariant measures with 
measures invariant with respect to a Vershik map. This happens
for generalized Bratteli diagrams of bounded size, see Lemma 
\ref{gbd size}. As a rule, we 
will consider measures on the path space of (ordered) Bratteli 
diagrams with zero-measure sets of maximal and minimal paths. 
This property will allow us to use the notions of tail-invariant 
measures and that of $\varphi_B$-invariant measures
interchangeably, see Section \ref{Sec:Tail_inv_mes}.
\end{remark}

In what follows we will use the following obvious fact:
Suppose that a tail-invariant Borel measure $\mu$ on $X_B$ takes 
finite values on all cylinder sets. Then $\mu$ is uniquely 
determined by its values on cylinder sets in $X_B$, i.e., it can
be extended uniquely to all Borel sets. 

The definition below uses 
the notion of \textit{Kakutani-Rokhlin 
towers} which is well-studied in the context of Cantor dynamics. We 
refer the reader to \cite{HermanPutnamSkau1992},
\cite{Putnam2018}, \cite{GiordanoPutnamSkau1995}, 
 \cite{BezuglyiKarpel2016} where this notion is discussed. 

\begin{definition} \label{Def:Kakutani-Rokhlin} Let $B =(V, E)$ 
be a generalized Bratteli diagram, for $w \in V_n, 
n \in \N$, denote 
$$
X_w^{(n)} = \{x = (x_i)\in X_B : r(x_{n-1}) = w\}.
$$
The collection of all such sets forms a partition of $X_B$ into  
\textit{Kakutani-Rokhlin towers}
corresponding to the vertices from $V_{n}$.
Each finite path $\ov e = (e_0, \ldots, e_{n-1})$ with 
$r(e_{n-1}) 
= w$, determines a ``floor'' of this tower
$$
X_w^{(n)}(\ov e) = \{x = (x_i)\in X_B : x_i = e_i,\; i = 
0,\ldots, n-1 \}
$$
(we denoted this set by $[\ol e]$ above; the notation 
$X_w^{(n)}(\ov
e)$ indicates the position of $[\ol e]$ in the tower $X_w^{(n)}$).
Clearly, 
$$
X_w^{(n)} = \bigcup_{\ol e \in E(V_0, w)} X_w^{(n)}(\ov e).
$$ Thus, the set $X_w^{(n)}$ is a union of a finite number of 
cylinder sets and can be considered as a \textit{tower associated 
with the vertex $w \in V_n$}.
\end{definition}

\begin{definition}\label{Def:Height}  For $v \in V_n$ and $v_0 \in
V_0$, we set $h^{(n)}_{v_0, v} = |E(v_0, v)| $ and define 
$$
H^{(n)}_v = \sum_{v_0 \in V_0} h^{(n)}_{v_0, v}, \ \ n \in \N.
$$ 
Set $H^{(0)}_v = 1$ for all $v\in V_0$.
This gives us the vector $H^{(n)} = (H^{(n)}_{v} : 
v \in V_n)$ associated with every level $n\in \N_0$. Since
$H^{(n)}_v = |E(V_0, v)|$, we call
$H^{(n)}_v$ the \textit{height of the tower} $X_v^{(n)}$ 
corresponding to the vertex $v\in V_n$.
 \end{definition} 

\begin{remark} We have defined the vector $H^{(0)} = (H^{(0)}_{v} : v \in V_0)$ such that $H^{(0)}_{v} = 1$ for all $v$ (see Definition~\ref{Def:Height}).  In fact, one can choose any finite values for $H^{(0)}_{v}$.  The role of $H^{(0)}$ can be interpreted as the vector of heights of the Kakutani-Rokhlin towers between the vertices of $V_0$ and an imaginary level $V_{-1}$ consisting of exactly one vertex. 
 \end{remark}
 
 Observe that 
 $$ H^{(n+1)}_v = \sum_{w \in V_n} \ent H^{(n)}_w,  \ \ \ v \in 
 V_{n+1},
 $$ 
 which immediately implies that 
\begin{equation}\label{lem vector H} 
    F_n H^{(n)} = H^{(n+1)} \ \ \ \textrm{ and } \ \ \ F_n \cdots 
    F_0 H^{(0)} = H^{(n+1)},  \,\,\,n \in \N_0.
\end{equation} 

We consider here the problem of the existence of tail-invariant 
measures on the path space of a generalized Bratteli diagram $B$. 
Our main results are mostly related to Bratteli diagrams of bounded 
size. Note that every incidence matrix $F_n$ defines a linear map
from $\R^{V_{n}}$ to $\R^{V_{n+1}}$ (recall that we identify all 
$V_n$). Using Lemma \ref{lemma_bdd_size_lower_cone}, we see that, 
for every fixed $n \in \mathbb{N}_0$ and any $m \in \mathbb{N}$, we 
can define the sequence of convex sets 
$$
C^{(n)}_m = F_n^T\cdots F_{n + m - 1}^T
(\mathbb{R}_{+}^{V_{n+m}}),
$$ 
where $F_i^T$ stands for the transpose of $F_i$. The above relation 
is well defined because $F_i^T$ maps the positive 
cone $\mathbb{R}_{+}^{V_{i+1}}$ into the positive cone of 
$\mathbb{R}_{+}^{V_{i}}$. Set 
$$
C^{(n)}_{\infty} = \bigcap_{m = 1}^{\infty} C^{(n)}_m.
$$ 
In general, the set $C^{(n)}_{\infty}$ might be empty.

Given a Bratteli diagram $B$ (generalized or classical), let 
$\mathcal M(B)$ denote the set of tail-invariant finite or 
$\sigma$-finite measures on the path space $X_B$ which takes finite 
values on cylinder sets. In the following theorem, we assume that 
the set $\mathcal M(B)$ is not empty. 

\begin{thm}\label{BKMS_measures=invlimits}
 Let $B = (V,E)$ be a Bratteli diagram (generalized or classical) 
 with the sequence of incidence matrices $(F_n)$. Then:
\begin{enumerate}

\item If  $\mu \in \mathcal M(B)$, then for every $n\in \N_0$ the 
vector defined as follows
\begin{equation}\label{eq:def_p_n}
    p^{(n)}= (\mu(X_w^{(n)}(\ov e)))_{w\in V_n}
\end{equation} satisfies $p^{(n)} \in  C_{\infty}^{(n)}$ and 
\begin{equation}\label{eq:formula_p_n}
F^{T}_n p^{(n+1)} = p^{(n)}
\end{equation}
for all $n\geq 0$.
\\

\item Conversely, suppose that $\{p^{(n)}= (p_w^{(n)}) \}_{n \in 
\N_0}$ is a sequence of non-negative vectors such that $p^{(n)} \in 
C_{\infty}^{(n)}$ and $F^{T}_n p^{(n+1)} = p^{(n)}$ for all $n \in 
N_0$. Then there exists a uniquely determined tail-invariant 
measure $\mu$ such that $\mu(X_w^{(n)}(\ov e))= p_w^{(n)}$ for 
$w\in V_n, n \in \mathbb N_0$.

\end{enumerate}

\end{thm} The \textit{proof} of Theorem \ref{BKMS_measures=invlimits} is 
straightforward and can be found in 
\cite{BezuglyiKwiatkowskiMedynetsSolomyak2010} (for classical Bratteli diagrams) and \cite{Bezuglyi_Jorgensen_2021} (for generalized Bratteli diagrams).

\begin{remark}
 We stress that part (1) of Theorem \ref{BKMS_measures=invlimits}
is true for any generalized Bratteli diagram whose path space 
admits a tail-invariant measure. This means also that, for 
this diagram, the cone $C_{\infty}^{(n)}$ is not empty for all
$n \in \N_0$. In Proposition  \ref{Prop:no_meas_fin_cyl_sets}, 
we give an example of a bounded size diagram such that both sets $\mathcal M(B)$ and $C_{\infty}^{(n)}$ are empty.
\end{remark}

Let $\mathcal E$ be a countable Borel equivalence relation on 
a standard Borel space $X$. It is a well-known fact that the 
existence of an $\mathcal E$-invariant probability Borel 
measure $\mu$  on $X$ is determined by the property of $\mathcal{E}$ called
\textit{compressibility}. For a fixed $x \in X$, the set 
$\{y \in X : 
(x, y) \in \mathcal E\}$ is called the $\mathcal E$-class. 
It is said that $\mathcal E$ is \textit{compressible} 
if there is an injective Borel map $f: X \to X$ such that for each 
$\mathcal E$-class $L$, $f(L) \subsetneq  L$. A Borel set $A 
\subset X$ is compressible if the restriction of $\mathcal E$ onto
$A$ is compressible. We refer to \cite{Nadkarni1991}, 
\cite{Nadkarni1995} where the following lemma is proved,  
(see also \cite{Dougherty_Jackson_Kechris1994}).

\begin{lemma} 
Let $\mathcal E$ be a countable equivalence relation on a standard
Borel space X. The following are equivalent:

\begin{enumerate}
    \item $\mathcal E$ is not compressible.
    \item  There is an $\mathcal E$-invariant probability measure.
    \item There is an $\mathcal E$-ergodic, $\mathcal E$-invariant probability measure.
\end{enumerate}
\end{lemma} 

Now we give an explicit example of a generalized Bratteli diagram 
such that the restriction of the tail equivalence relation onto the equivalence class of a non-empty clopen set is compressible. 

\begin{theorem}\label{Thm:no_full_inv_prob_mu}
For the one-sided generalized Bratteli diagram $B = (B,\omega)$ 
with the left-to-right ordering $\omega$ shown in Fig. 
\ref{Fig:GBD_Example_Downar_Karpel_2019}, we have
\begin{enumerate}
    \item The set $X_{max}$ is empty. 
    \item The Vershik map $\varphi_B : X_B \to X_B \setminus X_{min}$
is a homeomorphism.
\item There exists a non-empty clopen set $C$ such that its tail equivalence class $\mc R(C)$ is compressible. 
    \item There is no probability $\varphi_B$-invariant measure on $X_B$ that assigns positive values to all cylinder sets. 
\end{enumerate}
\end{theorem}

\begin{proof} The one-sided diagram in Fig. 
\ref{Fig:GBD_Example_Downar_Karpel_2019} is a modified version of 
\cite[Example 7.2]{DownarowiczKarpel_2019}. We identify each vertex 
level $V_i$ with $\N$. Similar to 
Example~\ref{Ex:LROrdering_no_V.map}, it is easy to note that  
there are no infinite maximal paths in the diagram with respect to
the left-to-right order $\omega$. This shows (1). As a 
consequence, the Vershik map $\varphi_B$ corresponding to $\omega$ 
is a homeomorphism from $X_B$ to $X_B \setminus X_{min}$. Hence we get (2).

%maps $X_B$ continuously and is an onto map from $X_B$ to $X_B 
%\setminus X_{\min}$. 

To prove (4), we show that there exists a cylinder set 
$C \subset X_B$ such that $\varphi_B (X_B) = X_B \setminus C$. Let 
$C$ be the cylinder set formed by all paths that begin at the 
leftmost vertex of $V_0$. The sub-diagram corresponding to $C$ is a 
tree and consists only of infinite minimal paths. Thus, 
$\varphi_B$ maps continuously 
$X_B$ to $X_B \setminus C$. The property 
$$
\varphi_B (X_B) = X_B \setminus X_{min} \subset X_B 
\setminus C
$$ 
shows that there does not exist any probability $\varphi_B$-invariant measure $\mu$ such that $\mu(C) >0$. Denote by $\mathcal{R}(C)$ the tail equivalence class of $C$. Since every tail equivalence class $L$ in $\mathcal{R}(C)$ contains a minimal path, we have $\varphi_B(L) \subsetneq L$ for every $L$. Thus, (3) is also proved.
\end{proof}

\begin{figure}[hbt!]
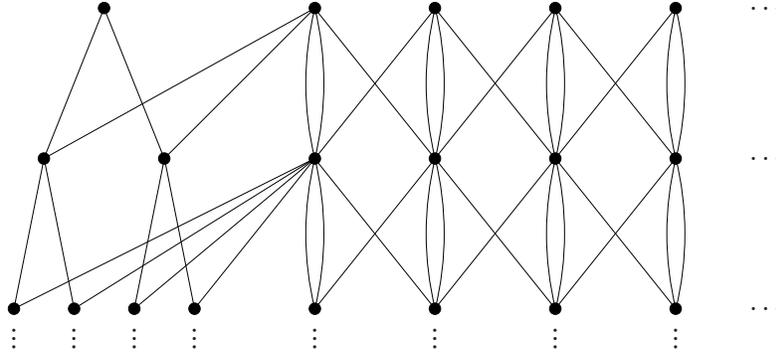

\unitlength = 0.4cm
\begin{center}
\begin{graph}(30,12)
\graphnodesize{0.4}

% Vertices of the first level
\roundnode{V11}(5,11)
%\nodetext{V11}(-0.6,0.4){$u$}
\roundnode{V12}(12,11)
%\nodetext{V12}(-0.6,0.4){$v$}
\roundnode{V13}(16,11)
%\nodetext{V13}(0.6,0.4){$w$}
\roundnode{V14}(20,11)
\roundnode{V15}(24,11)

%Vertices of the second level
\roundnode{V21}(3,6)
\roundnode{V22}(7,6)
\roundnode{V23}(12,6)%odometer
\roundnode{V24}(16,6)
\roundnode{V25}(20,6)
\roundnode{V26}(24,6)
%\freetext(19,6){$w$}

% Edges of the second level
\bow{V23}{V12}{-0.06}%odometer
\bow{V23}{V12}{0.06}%odometer
%\freetext(14.3,8.5){0}
%\freetext(15.7,8.5){1}

\edge{V21}{V11}
%\edgetext{V21}{V11}{0}
\edge{V22}{V11}
%\edgetext{V22}{V11}{0}
\edge{V21}{V12}
%\edgetext{V21}{V12}{1}
\edge{V22}{V12}
%\edgetext{V22}{V12}{1}

\edge{V24}{V12}
%\edgetext{V24}{V12}{0}
\bow{V24}{V13}{-0.06}%odometer
\bow{V24}{V13}{0.06}%odometer
\bow{V25}{V14}{-0.06}%odometer
\bow{V25}{V14}{0.06}%odometer
\bow{V26}{V15}{-0.06}%odometer
\bow{V26}{V15}{0.06}%odometer
%\edgetext{V24}{V13}{1}
%\edgetext{V25}{V12}{0}
\edge{V25}{V13}
\edge{V26}{V14}
\edge{V23}{V13}
\edge{V24}{V14}
\edge{V25}{V15}

%\edgetext{V25}{V13}{1}
%\edgetext{V23}{V12}{0}

%Vertices of the third level

\roundnode{V31}(2,1)
\roundnode{V32}(4,1)
\roundnode{V33}(6,1)
\roundnode{V34}(8,1)

\roundnode{V35}(12,1)

\roundnode{V36}(16,1)
\roundnode{V37}(20,1)
\roundnode{V38}(24,1)

\freetext(2,0){$\vdots$}
\freetext(4,0){$\vdots$}
\freetext(6,0){$\vdots$}
\freetext(8,0){$\vdots$}
\freetext(12,0){$\vdots$}
\freetext(16,0){$\vdots$}
\freetext(20,0){$\vdots$}
\freetext(24,0){$\vdots$}
\freetext(27,11){$\ldots$}
\freetext(27,6){$\ldots$}
\freetext(27,1){$\ldots$}
% Edges of the third level

\bow{V35}{V23}{-0.06}
\bow{V35}{V23}{0.06}
%\freetext(14.3,3.5){0}
%\freetext(15.7,3.5){1}

\edge{V31}{V21}
%\edgetext{V31}{V21}{0}
\edge{V31}{V23}
%\edgetext{V31}{V23}{1}

\edge{V32}{V21}
%\edgetext{V32}{V21}{0}
\edge{V32}{V23}
%\edgetext{V32}{V23}{1}

\edge{V33}{V22}
%\edgetext{V33}{V22}{0}
\edge{V33}{V23}
%\edgetext{V33}{V23}{1}

\edge{V34}{V22}
%\edgetext{V34}{V22}{0}
\edge{V34}{V23}
%\edgetext{V34}{V23}{1}

\edge{V35}{V24}
%\edgetext{V36}{V24}{1}
\edge{V36}{V23}
\edge{V36}{V25}
%\edgetext{V36}{V23}{0}

\edge{V37}{V24}
\edge{V37}{V26}
%\edgetext{V37}{V24}{1}

\edge{V38}{V25}
%\edgetext{V38}{V25}{1}

\bow{V24}{V36}{-0.06}%odometer
\bow{V24}{V36}{0.06}%odometer
\bow{V25}{V37}{-0.06}%odometer
\bow{V25}{V37}{0.06}%odometer
\bow{V26}{V38}{-0.06}%odometer
\bow{V26}{V38}{0.06}%odometer

\end{graph}
\caption{ A diagram with the left-to-right ordering, and no infinite maximal paths.}\label{Fig:GBD_Example_Downar_Karpel_2019}
\end{center}
\end{figure}

\begin{remark} Observe that for the classical (standard) Bratteli 
diagram shown in Fig.~\ref{Fig:Example_Downar_Karpel_2019}, the 
conclusion (4) of Theorem~\ref{Thm:no_full_inv_prob_mu} also 
holds. To see this, we extend the Vershik map to the entire path 
space by mapping the unique maximal path to the unique minimal path 
in the $2$-odometer (i.e. the subdiagram corresponding to the 
vertex $v$ in Fig.~\ref{Fig:Example_Downar_Karpel_2019}). 

Let $C$ be the cylinder set defined by the edge $[v_0, u]$. Then it 
is easy to see that $\varphi_B(X_B) = X_B \setminus C$. Hence, 
every probability $\varphi_B$-invariant measure $\mu$ must satisfy 
the condition $\mu(C) = 0$ which imply $(4)$. There is a unique 
probability invariant measure on $X_B$ sitting on the minimal 
component of the tail equivalence relation, the 2-odometer 
corresponding to the vertex $v$.
\end{remark}

\begin{figure}[hbt!]
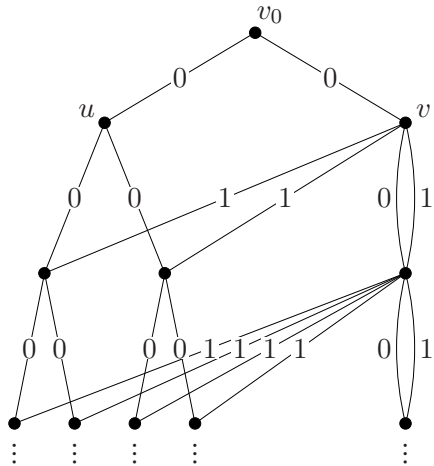

\unitlength = 0.4cm
\begin{center}
\begin{graph}(17,15)
\graphnodesize{0.4}
% The top vertex
\roundnode{V0}(10,14)
\freetext(10.5,14.6){$v_0$}
% Vertices of the first level
\roundnode{V11}(5,11)
\nodetext{V11}(-0.6,0.4){$u$}
\roundnode{V12}(15,11)
\nodetext{V12}(0.6,0.4){$v$}
%\roundnode{V13}(25,11)
%\nodetext{V13}(0.6,0.4){$w$}

% Edges of the first level
\edge{V11}{V0}
\edgetext{V11}{V0}{0}
\edge{V12}{V0}
\edgetext{V12}{V0}{0}
%\edge{V13}{V0}
%\edgetext{V13}{V0}{0}

%Vertices of the second level
\roundnode{V21}(3,6)
\roundnode{V22}(7,6)
\roundnode{V23}(15,6)%odometer
%\roundnode{V24}(23,6)
%\roundnode{V25}(27,6)

%\freetext(19,6){$w$}

% Edges of the second level
\bow{V23}{V12}{-0.06}%odometer
\bow{V23}{V12}{0.06}%odometer
\freetext(14.3,8.5){0}
\freetext(15.7,8.5){1}

\edge{V21}{V11}
\edgetext{V21}{V11}{0}
\edge{V22}{V11}
\edgetext{V22}{V11}{0}
\edge{V21}{V12}
\edgetext{V21}{V12}{1}
\edge{V22}{V12}
\edgetext{V22}{V12}{1}

%\edge{V24}{V12}
%\edgetext{V24}{V12}{0}
%\edge{V24}{V13}
%\edgetext{V24}{V13}{1}
%\edge{V25}{V12}
%\edgetext{V25}{V12}{0}
%\edge{V25}{V13}
%\edgetext{V25}{V13}{1}
%\edgetext{V23}{V12}{0}

%Vertices of the third level

\roundnode{V31}(2,1)
\roundnode{V32}(4,1)
\roundnode{V33}(6,1)
\roundnode{V34}(8,1)

\roundnode{V35}(15,1)

%\roundnode{V36}(22,1)
%\roundnode{V37}(24,1)
%\roundnode{V38}(26,1)
%\roundnode{V39}(28,1)

\freetext(2,0){$\vdots$}
\freetext(4,0){$\vdots$}
\freetext(6,0){$\vdots$}
\freetext(8,0){$\vdots$}
\freetext(15,0){$\vdots$}
%\freetext(22,0){$\vdots$}
%\freetext(24,0){$\vdots$}
%\freetext(26,0){$\vdots$}
%\freetext(28,0){$\vdots$}

% Edges of the third level

\bow{V35}{V23}{-0.06}
\bow{V35}{V23}{0.06}
\freetext(14.3,3.5){0}
\freetext(15.7,3.5){1}

\edge{V31}{V21}
\edgetext{V31}{V21}{0}
\edge{V31}{V23}
\edgetext{V31}{V23}{1}

\edge{V32}{V21}
\edgetext{V32}{V21}{0}
\edge{V32}{V23}
\edgetext{V32}{V23}{1}

\edge{V33}{V22}
\edgetext{V33}{V22}{0}
\edge{V33}{V23}
\edgetext{V33}{V23}{1}

\edge{V34}{V22}
\edgetext{V34}{V22}{0}
\edge{V34}{V23}
\edgetext{V34}{V23}{1}

%\edge{V36}{V24}
%\edgetext{V36}{V24}{1}
%\edge{V36}{V23}
%\edgetext{V36}{V23}{0}

%\edge{V37}{V24}
%\edgetext{V37}{V24}{1}
%\edge{V37}{V23}
%\edgetext{V37}{V23}{0}

%\edge{V38}{V25}
%\edgetext{V38}{V25}{1}
%\edge{V38}{V23}
%\edgetext{V38}{V23}{0}

%\edge{V39}{V25}
%\edgetext{V39}{V25}{1}
%\edge{V39}{V23}
%\edgetext{V39}{V23}{0}

\end{graph}
\caption{Illustration of the result in Theorem \ref{Thm:no_full_inv_prob_mu} via a standard Bratteli diagram.}\label{Fig:Example_Downar_Karpel_2019}
\end{center}
\end{figure} 

Proposition \ref{Prop:no_meas_fin_cyl_sets}, gives an example of 
a stationary generalized Bratteli diagram such that there is no 
tail-invariant measure with finite values on clopen sets. 
We emphasize that this example can be generalized to a class of 
stationary diagrams with similar property. These diagrams have an 
incidence matrix of the form given by \eqref{matrix_no_prob} where 
the diagonal entries form an increasing  sequence of positive 
integers. 

\begin{prop}\label{Prop:no_meas_fin_cyl_sets}
Let $B = B(F)$ be a one-sided generalized stationary Bratteli 
diagram as shown in Figure \ref{Fig:no measure} and given by 
$\mathbb{N} \times \mathbb{N}$ incidence matrix 
\begin{equation}\label{matrix_no_prob}
    F = \begin{pmatrix}
2 & 1 & 0 & 0 & \ldots\\
0 & 3 & 1 & 0 & \ldots\\
0 & 0 & 4 & 1 & \ldots\\
0 & 0 & 0 & 5 & \ldots\\
\vdots & \vdots & \vdots & \vdots & \ddots\\
\end{pmatrix}
\end{equation} There does not exist any tail-invariant measure on 
$X_B$ that assigns finite values to cylinder sets.
\end{prop}

\begin{figure}[hbt!]
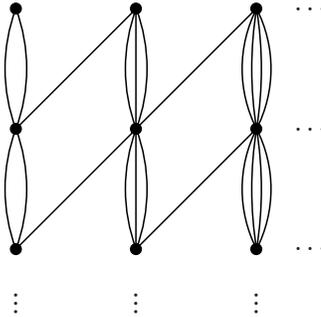

\unitlength=.8cm
\begin{graph}(7,6)
% \graphnodesize{0.2}
% \roundnode{V0}(3,6)
%  %\nodetext{V0}(-1,0){$V_0$}
 \roundnode{V11}(2,5)
 %\nodetext{V11}(-0.7,0){$w_1^{(0)}$}
  %\nodetext{V12}(0.7,0){$w_2^{(0)}$}
 \roundnode{V12}(4,5)
 \roundnode{V13}(6,5) 
  % The second level vertices
 \roundnode{V21}(2,3)
 \roundnode{V22}(4,3)
 \roundnode{V23}(6,3)
  %\nodetext{V21}(-0.7,0){$w_1^{(1)}$}
 % \nodetext{V22}(0.7,0){$w_2^{(1)}$}
 % The third level vertices
 \roundnode{V31}(2,1)
 \roundnode{V32}(4,1)
 \roundnode{V33}(6,1)
 % \nodetext{V31}(-0.7,0){$w_1^{(2)}$}
  %\nodetext{V32}(0.7,0){$w_2^{(2)}$}
  %
 %
 % EDGES
 \graphlinewidth{0.025}
% % First level
%  \edge{V0}{V11}
%  \edge{V0}{V12}

 % Second level
 \bow{V21}{V11}{0.09}
  \bow{V21}{V11}{-0.09}
    \edge{V21}{V12}
     
 \bow{V22}{V12}{0.09}
 \bow{V22}{V12}{-0.09}
 \edge{V22}{V12}
 
 \bow{V23}{V13}{0.12}
  \bow{V23}{V13}{-0.12}
   \bow{V23}{V13}{0.04}
  \bow{V23}{V13}{-0.04}
 \edge{V22}{V13}

 %third level
 \bow{V31}{V21}{0.09}
 \bow{V31}{V21}{-0.09}
   \edge{V31}{V22}
%     \edge{V32}{V22}[\graphlinecolour(1,0,0)]
 \bow{V32}{V22}{0.09}
 \bow{V32}{V22}{-0.09}
 \edge{V32}{V22}

  \bow{V33}{V23}{0.12}
  \bow{V33}{V23}{-0.12}
   \bow{V33}{V23}{0.04}
  \bow{V33}{V23}{-0.04}
 \edge{V32}{V23}
 
    \freetext(6.9,5){$\ldots$}
  \freetext(6.9,3){$\ldots$}
    \freetext(6.9,1){$\ldots$}
    \freetext(2,0.1){$\vdots$}
  \freetext(4,0.1){$\vdots$}
    \freetext(6,0.1){$\vdots$}
%\freetext(3,0.1){.\,.\,.\,.\,.\,.\,.\,.\,.\,.\,.\,.\,.\,.\,.\,.\,.\,.\,.\,.\,..\,.\,.\,.\,.\,.\,.\,.}
\end{graph}
\caption{A Bratteli diagram with no finite ergodic invariant measure.}\label{Fig:no measure}
\end{figure}

\begin{proof} As above, we identify vertices at each level with 
natural numbers. For $k \in \mathbb{N}$, denote by $C_k$ the 
cylinder set 
corresponding to the vertex $k$ on level $V_0$, i.e. 
$$
C_k = \{x \in X_B : s(x) = k \}, \ \ \  k \in V_0.
$$ 
Recall that  $s:E \rightarrow V$ is the source map. Suppose that 
there exists a non-zero tail-invariant measure $\mu$ on $X_B$ 
which assigns finite values to cylinder sets. Denote by
$$
m = \min\{k \in \mathbb{N} : \mu(C_k) > 0\}.
$$ 
By our assumption, the minimum exists. Normalize the measure $\mu$ 
such that $\mu(C_m) = 1$. Let $\ov e =(e_0,\ldots,e_n)$ be a finite 
path of length $(n+1)$ such that the range of $\ov e$ is the vertex 
labeled by $m$ on level $V_n$, i.e., $r(\ov e) = m \in V_n$. 
As before, we denote by $[\ov e]$ the corresponding cylinder set. 
It follows from the definition of $F$ that the path space $X_B$ 
contains countably many odometers: the set of vertical paths going 
through a vertex $i$ is an $(i+1)$-odometer. By tail invariance of 
$\mu$, all cylinder sets in the $(m + 1)$-odometer have the same 
measure as the set $[\ov e]$ has: 
$$
\mu([\ov e]) = \frac{1}{(m+1)^n}.
$$
Thus we have 
$$
\mu(C_{m+1}) \geq \sum_{n = 1}^{\infty} \frac{(m+2)^{n-1}}{(m+1)^n}
= \infty,
$$ 
and this is a contradiction. 
\end{proof}

\begin{remark} Let $B(F)$ be a stationary Bratteli diagram as in 
Proposition \ref{Prop:no_meas_fin_cyl_sets}, and let $\mu$ be any 
tail-invariant measure. Recall that the vector $p^{(0)}$
(see \eqref{eq:def_p_n}) consists of the values of the measure 
$\mu$ of cylinder sets corresponding to the level $V_0$.  It 
follows from the proof of  Proposition 
\ref{Prop:no_meas_fin_cyl_sets} that if, for some vertex $i\in 
V_0$, we have 
$$
0 < p_i^{(0)} < \infty,
$$ 
then $p_j^{(0)} = 0$ for 
every $j < i$ and $p_j^{(0)} = \infty$ for every $j > i$. Since the 
diagram is stationary, the same property holds for every level 
$V_n$; $n \in \mathbb{N}_0$.
\end{remark}

\section{Uniqueness of tail-invariant measures  
for stationary Bratteli diagrams }\label{Section:stat_GBD}

\subsection{Tail-invariant measures and Perron-Frobenius 
eigenvectors}

In \cite{Bezuglyi_Jorgensen_2021}, the authors used the
Perron-Frobenius (P-F) 
theory for infinite matrices to provide a description of tail 
invariant measures on the path space of a class of stationary 
generalized Bratteli diagrams. This class consists of diagrams 
with irreducible, aperiodic, and recurrent incidence matrices. 
for the reader's convenience, we provide a brief description of 
the P-F theory for infinite matrices in 
Appendix \ref{APP:Perron-Frobenius_Theory}. The results formulated 
in Appendix \ref{APP:Perron-Frobenius_Theory} are mostly taken from 
Chapter $7$ of the book \cite{Kitchens1998}. 
The foundations of the P-F theory for infinite matrices were laid 
down in the 1960s in a series of articles by D. Vere-Jones
\cite{VereJones_1967}, \cite{VereJones_1968}, 
\cite{VereJones_1962}. 
In this section, we use results from 
Appendix \ref{APP:Perron-Frobenius_Theory}, in particular, 
Theorem\ref{Thm:Generalized_Perron_Frobenius} and Theorem 
\ref{Thm:Pos_recc} to prove a criterion for the uniqueness of the 
tail-invariant measure on the path space of a stationary 
generalized Bratteli diagram $B(F)$. Recall that a matrix is 
called countably infinite if its rows and columns are indexed by 
a countable set. If $F$ is indexed by $\N$, then the diagram 
$B(F)$ is one-sided infinite; and if $F$ is indexed by $\Z$, then 
the diagram is two-sided infinite. 
The main results in this section hold for both kinds of diagrams. 
We provide examples (see Subsection \ref{Subsec:Examples}) of both 
kinds of diagrams.

We will keep the following notations: $A$ is the transpose of the 
infinite
incidence matrix $F$ of the generalized Bratteli diagram $B(F)$. 
When it exists (see Theorem\ref{Thm:Generalized_Perron_Frobenius}), 
we will denote by $\lambda$ the Perron eigenvalue and by 
$\xi = (\xi_v)$, $\eta = (\eta_v)$ the right and left eigenvectors 
for $A$, i.e., $A\xi = \lambda \xi$ and $\eta A = \lambda\eta$. 
Note that all entries of $\xi$ and $\eta$ are positive. 

In this section, we will work with an ordered stationary 
generalized Bratteli diagram $B(F) = B(V, E,>)$ where $>$ denotes 
a fixed order. We will assume that the order $>$ gives rise to a 
Vershik map on the path space of the diagram. As before we will 
denote the corresponding dynamical system by $(X_B, \varphi_B)$. We 
recall the following result (proved in 
\cite{BezuglyiJorgensen2022}) which gives an explicit formula 
for a tail-invariant measure $\mu$ on the path space of 
a stationary generalized Bratteli diagram.

\begin{theorem} [Theorem 2.20, \cite{BezuglyiJorgensen2022}] 
\label{Thm:inv1} 
Let $B(F) = B(V,E)$ be a stationary generalized 
Bratteli diagram such that the matrix $A = F^T$ is infinite, 
irreducible, aperiodic, and recurrent. Let $\xi = (\xi_v)$ be a 
Perron-Frobenius right eigenvector for $A$, i.e., $A \xi = 
\lambda \xi$, $\xi_v >0$.
\begin{enumerate}
    \item There exists a tail-invariant measure $\mu$ on 
    the path space $X_B$, satisfying the following property:
    if $\ol e(w, v)$ is a  
    finite path that begins at $w \in V_0$ and ends at $v \in 
    V_n$, $n \in \N$, then 
 \be\label{eq inv meas left}
 \mu([\ol e(w, v)]) = \frac{\xi_v}{\lambda^{n}}, 
 \ee  
 where $[\ol e(w, v)]$ is the corresponding cylinder set. 
 %In other words, the $(X_B,\varphi_B,\mu)$ is a measure-preserving dynamical system. 
 
 \item The measure $\mu$ is finite if and only if the 
Perron eigenvector $\xi = (\xi_v)$ has the property 
$\sum_{v} \xi_v < \infty$.

\end{enumerate}
\end{theorem} 

Now we show that if the incidence matrix of a stationary 
generalized Bratteli diagram is  \textit{positive recurrent} (in 
addition to the properties described in Theorem \ref{Thm:inv1}), 
then the invariant measure given by \eqref{eq inv meas left} is 
unique. We consider  the cases of finite and $\sigma$-finite 
measures separately. In the $\sigma$-finite case (Theorem 
\ref{Thm:Unique sigma finite}), we will work with the assumption 
that the dynamical system is conservative. Examples \ref{Ex:renewal_subshift}, \ref{Ex:BobokBruin} illustrate Theorems \ref{Thm:Unique finite}, \ref{Thm:Unique sigma finite}. 

\begin{thm}\label{Thm:Unique finite}
Let $B(F) = B(V,E,>)$ be an ordered stationary generalized 
Bratteli diagram 
such that the matrix $A = F^T$ is infinite, irreducible, aperiodic, 
and positive recurrent. Let $\xi = (\xi_v)$ be a Perron-Frobenius 
right eigenvector for $A$ such that 
$\sum_{u\in V_0} \xi_u = 1$. Then the measure $\mu$ given in \eqref{eq inv meas left} is the unique probability  
$\varphi_B$-invariant measure that takes positive values 
on cylinder sets. 
\end{thm}

\begin{proof}
%Without loss of generality, we can assume that 
%$\sum_{v\in V_0} \xi_v = 1$. 
We fix a vertex $w \in V_n$, and 
consider a cylinder set given by  $[\ov e] = 
(e_0, \ldots, e_{n-1})$ 
with $r(e_{n-1}) = w$. Then by (\ref{eq inv meas left}), we have 
$$
\mu([\ov e]) = \frac{\xi_w}{\lambda^n}.
$$
As proved in Theorem \ref{Thm:inv1}, the measure $\mu$ is 
probability and takes finite positive values on cylinder sets. 
Suppose that $\nu$ is a probability ergodic $\varphi_B$-invariant 
measure with positive values on cylinder sets. 
\ignore{ Denote by $p_v^{(n)} = \nu([\ol e])$ where $[\ov e] = 
(e_0, \ldots, e_{n-1})$ with $r(e_{n-1}) = v$.}
Let $N \geq n$ and $v \in V_N$ be such that $E(w,v) \neq 
\emptyset$. 
Then, by the Birkhoff ergodic theorem, 
$$
\nu([\ov e]) = \lim_{N \rightarrow \infty} \frac{|E(w,v)|}
{H_v^{(N)}}, 
$$ 
where $H_v^{(N)}$ is the total number of finite paths 
with the range at vertex $v \in V_N$ (see Definition 
\ref{Def:Height}). Since $A$ is positive recurrent, it follows 
from Theorem \ref{Thm:Pos_recc} that, for any $v,w \in V$, 
$$
\underset{N \rightarrow 
\infty}{\mathrm{lim}} \,\, \dfrac{a_{wv}^{(N)}}{\lambda^{N}} = 
\xi_w\eta_v,
$$ 
where $a_{wv}^{(N)}$ is the entry of $A^N$ and $\eta = (\eta_v)$ 
is the left eigenvector of $A$ normalized by the condition
$\eta \cdot \xi =1$. Thus, we obtain 
$$
\nu([\ov e]) = \lim_{N \rightarrow \infty} 
\frac{a^{(N-n)}_{wv}}{\sum_{u \in V_0} a^{(N)}_{uv}} = \frac{\xi_w  
\eta_v}{\lambda^n\sum_{u \in V_0} \xi_u  \eta_v } = \frac{\xi_w}
{\lambda^n} = \mu([\ov e]).
$$
\end{proof}

Now we show the uniqueness of the infinite $\sigma$-finite 
$\varphi_B$-invariant measure given by \eqref{eq inv meas left}. 
We work with an additional assumption that the dynamical system 
$(X_B,\varphi_B,\mu)$ is conservative. 

\begin{thm}\label{Thm:Unique sigma finite} Let $B(F) = B(V,E,>)$ 
be an ordered stationary generalized Bratteli diagram such that 
the matrix $A = F^T$ is infinite, irreducible, aperiodic, and 
positive 
recurrent. Let $\xi = (\xi_v)$ be a Perron-Frobenius right 
eigenvector for $A$ such that $\sum_{u\in V_0} \xi_u = \infty$. 
Let $\mu$ be the $\sigma$-finite $\varphi_B$-invariant (given by 
\eqref{eq inv meas left}) such that $(X_B, \varphi_B, \mu)$ is 
conservative. Then $\mu$ is the unique (up to a constant 
multiple) $\sigma$-finite $\varphi_B$-invariant ergodic measure 
that takes positive values on cylinder sets.
\end{thm}

\begin{proof} By Theorem~\ref{Thm:inv1}, there exist an 
invariant $\sigma$-finite measure $\mu$ on the path space of a 
generalized Bratteli diagram with irreducible, aperiodic and 
recurrent incidence matrix. Consider two cylinder sets 
$[\ov{e_1}],\, [\ov{e_2}] \subset X_B$ 
such that  $r(\ov{e_1})= v_1 \in V_{n_1}$ and $r(\ov{e_2})= v_2 
\in V_{n_2}$. Without loss of generality, assume that 
$n_2 > n_1$. Using (\ref{eq inv meas left}), we 
calculate the ratio of their measures,
\begin{equation}\label{ratio 1}
    \dfrac{\mu([\ov{e_1}])}{\mu([\ov{e_2}])} = 
    \dfrac{\xi_{v_1}/\lambda^{n_1}}{\xi_{v_2}/\lambda^{n_2}} = 
    \dfrac{\xi_{v_1}}{\xi_{v_2}} \lambda^{(n_2-n_1)}. 
\end{equation}

Let $m$ be a $\sigma$-finite ergodic measure on $X_B$. 
Now we apply Hopf's ratio ergodic theorem for $m$  and find the 
ratio of measures of the same cylinder sets $[\ov{e_1}]$ and 
$[\ov{e_2}]$ (see \cite{Aaronson97} for references). 
For this, let $N > n_2$ and $w \in V_N$ be such that the sets 
$E(v_1,w)$ and $E(v_2,w)$ are non empty. 
Since $A$ is irreducible, we can choose such $N$ using Lemma 
\ref{Lemma_Perron_value}(i). Now, we apply Hopf's ratio 
ergodic theorem to obtain
$$
\dfrac{m([\ov{e_1}])}{m([\ov{e_2}])} = \underset{N \rightarrow 
\infty}{\mathrm{lim}} \,\, \dfrac{|E(v_1,w)|}{|E(v_2,w)|}\,\,\, = 
\underset{N \rightarrow \infty}{\mathrm{lim}} \,\, 
\dfrac{a_{v_1w}^{(N-n_1)}}{a_{v_2w}^{(N-n_2)}} 
= \underset{N \rightarrow \infty}{\mathrm{lim}} \,\, 
\dfrac{a_{v_1w}^{(N-n_1)}}{\lambda^{N-n_1}} \cdot 
\dfrac{\lambda^{N-n_2}}{a_{v_2w}^{(N-n_2)}}\cdot \lambda^{n_2-n_1}.
$$ 
Since $A$ is positive recurrent, it follows from 
Theorem \ref{Thm:Pos_recc} that, for any $v,w \in V$,  
$$
\underset{N \rightarrow 
\infty}{\mathrm{lim}} \,\, \dfrac{a_{vw}^{(N)}}{\lambda^{N}} = 
\xi_v\eta_w.
$$
Therefore, 
\begin{equation}\label{ratio 2}
\dfrac{m([\ov{e_1}])}{m([\ov{e_2}])} = \dfrac{\xi_{v_1}\,\,\eta_w}
{\xi_{v_2}\,\,\eta_w} \lambda^{n_2-n_1}  = \dfrac{\xi_{v_1}}
{\xi_{v_2}} \lambda^{n_2-n_1} = \dfrac{\mu([\ov{e_1}])}
{\mu([\ov{e_2}])}.
\end{equation} 
Relation \eqref{ratio 2} shows that, for every cylinder set 
$[\ol e]$, the ratio 
$$
\frac{m([\ov{e}])}{\mu([\ov{e}])} = c
$$
for some constant $c$. This means that $m = c\mu$, and  the  
proof is complete. 
\end{proof}

\ignore{Fix a finite path $p$ in the path space and consider the 
corresponding cylinder set $[\ov{p}] \subset X_B$. Since we are 
working with $\sigma$-finite measures, there exists a constant 
$c >0$
such that $\frac{m([\ov{p}])}{\mu([\ov{p}])} = c$. Thus, for any 
cylinder set $[\ov{e}]$, 
$$
\frac{m([\ov{e}])}{\mu([\ov{e}])} =  \frac{m([\ov{e}])}
{m([\ov{p}])}\frac{m([\ov{p}])}{\mu([\ov{e}])} =  \frac{m([\ov{e}])}
{m([\ov{p}])}\frac{c\cdot \mu([\ov{p}])}{\mu([\ov{e}])}.
$$ Now setting $e= e_1$ and $e_2 = p$ in  (\ref{ratio 2}) we get 
$\frac{m([\ov{e}])}{\mu([\ov{e}])} = c$. This shows that $m = c\cdot
\mu$. Hence the measure $\mu$ defined in (\ref{eq inv meas left} 
a unique (up to a constant) $\varphi_B$-invariant $\sigma$-finite 
measure on $X_B$.}

\ignore{
As a corollary of Theorems 
\ref{Thm:Unique finite} and  
Theorem \ref{Thm:Unique sigma finite}, we obtain following result: 

\begin{corol}
Let $\sigma$ be a bounded size, left determined substitution on a 
countably infinite alphabet with irreducible, aperiodic, and 
positive recurrent substitution matrix. Then the substitution 
dynamical system 
$(X_{\sigma}, T)$ is uniquely ergodic. Furthermore, the measure 
is finite if the sum of entries of the Perron eigenvector 
corresponding to the incidence matrix converges and $\sigma$-finite, 
otherwise. 
\end{corol}}

\subsection{Generalized Bratteli diagrams with finite tail 
invariant measures}\label{subsec:finite_tail-inv_measure}

Suppose that a matrix $A = (a_{ij})$ with $a_{ij} \in \mathbb{N}_0$ 
for all $i,j$ is such that the Perron-Frobenius theorem holds: 
there exists a finite Perron eigenvalue $\lambda$ and a 
non-negative right eigenvector $\xi = (\xi_i)$ such that $A\xi = 
\lambda \xi.$ We give sufficient conditions on the matrix $A$ that 
lead to the existence of a summable eigenvector $\xi$, i.e., 
$\sum_i \xi_i < \infty$. By Theorem \ref{Thm:inv1}, these 
conditions will guarantee the existence of a finite tail-invariant
measure. 

\begin{prop}
Let $A = (a_{i,j} : i, j \in \N)$ be a non-negative matrix such 
that there exists a positive eigenvector $\xi$ corresponding an 
eigenvalue $\lambda$. 
If there exists a row of $A$ with finitely many zero entries, 
then the eigenvector $\xi = (\xi_i)$ is summable:
$$
\sum_{i \in \mathbb{N}} \xi_i < \infty.
$$
\end{prop}

\begin{proof}
Let the $i$th row of $A$ have finitely many zero entries:
$$
a_{ij} = 0 \ \Longleftrightarrow \ (j \in I, \; |I|<\infty).
$$
From the equality
$$
\sum_{j \in \mathbb{N}} a_{ij}\xi_j = \lambda \xi_i
$$
we have
$$
\sum_{j \in \mathbb{N}} \xi_j = \sum_{j \in I} \xi_j + 
\sum_{j \notin I} \xi_j \leq \sum_{j \in I} \xi_j  + \sum_{j \in 
\mathbb{N}} a_{ij} \xi_j = \sum_{j \in I} \xi_j + 
\lambda \xi_i < \infty.
$$
\end{proof} 

Note that the matrix $A$ will also have a summable right 
eigenvector if there are two rows in $A$ such that one row 
contains non-zero elements in even-numbered columns and 
the other one in odd-numbered columns. More generally, 
the following result holds.

\begin{prop}\label{prop:finite_eigenvector}
Let $A$ be a $\mathbb{N} \times \mathbb{N}$ matrix with 
non-negative integer entries such that $A\xi = \lambda \xi$ for 
$0 < 
\lambda < \infty$ and $\xi > 0$. For every $k\in \mathbb{N}$, let
$$
M_k = \{j \in \mathbb{N} : a_{kj} > 0\}.
$$
Assume that there exists a finite collection of rows $\{k_1, 
\ldots, k_p\}$ such that
$$
|\mathbb{N} \setminus \bigcup_{t = 1}^p M_{k_p}| < \infty.
$$
Then the eigenvector $\xi = (\xi_i)$ is finite in the sense that 
$$
\sum_{i \in \mathbb{N}} \xi_i < \infty.
$$
\end{prop}

\begin{proof}
For each $t = 1, \ldots, p$, we have
$$
\sum_{j \in \mathbb{N}} a_{k_t j} \xi_j = \sum_{j \in 
M_{k_t}}a_{k_tj} \xi_j = \lambda \xi_{k_t}.
$$
Consider the sum of these relations:
$$
\sum_{t = 1}^p \sum_{j \in M_{k_t}}a_{k_tj} \xi_j = 
\lambda \sum_{t = 1}^p \xi_{k_t}.
$$
Note that since $a_{k_t j} \geq 1$ for $j \in M_{k_t}$, we have
$$
\sum_{t = 1}^p \sum_{j \in M_{k_t}}a_{k_tj} \xi_j \geq \sum_{j 
\in \bigcup_{t = 1}^p M_{k_p}} \xi_j.
$$
Then we get 
$$
\ba 
\sum_{j \in \mathbb{N}} \xi_j = & \sum_{j \in \mathbb{N} 
\setminus\bigcup_{t = 1}^p M_{k_p}} \xi_j + \sum_{j \in 
\bigcup_{t = 1}^p M_{k_p}} \xi_j \\
\leq  & \sum_{j \in \mathbb{N} \setminus\bigcup_{t = 1}^p 
M_{k_p}} \xi_j + \lambda \sum_{t = 1}^p \xi_{k_t} < \infty,
\ea
$$
since the set $\mathbb{N} \setminus \bigcup_{t = 1}^p M_{k_p}$ 
is finite. 
\end{proof}

\begin{remark}
The converse of Proposition~\ref{prop:finite_eigenvector} is not 
true since there are banded matrices with probability right 
eigenvectors, see examples in Subsection \ref{Subsec:Examples}.
\end{remark}

\subsection{Examples}\label{Subsec:Examples} In this subsection we 
consider several classes of stationary 
generalized Bratteli diagrams that admit finite and 
$\sigma$-finite  tail-invariant measures on their path spaces.
The reader can find the proof of these results in Appendix 
\ref{APP:Example}. 

\begin{example}\label{Ex:Matrix A_1} For $a,b \in \N$, consider 
the generalized 
stationary Bratteli diagram $B(F_1)$ where $A_1 = F_1^{T}$ is 
given by 
\begin{equation}\label{Matrix A_1}
A_1 =  \left(
  \begin{array}{cccccccccccc}
   \ddots & \vdots & \vdots & \vdots & \vdots & \vdots & \vdots 
   & \vdots & \vdots & \vdots & \udots\\
    \cdots & 2b & 0 & a & 0 & 0 & \textbf  0 & 0 & 0 & 0 & 0 & \cdots\\
    \cdots & 0 & 2b & 0 & a & 0 & \textbf 0 & 0 & 0 & 0 & 0 & \cdots\\
    \cdots & 0 & 0 & 2b & 0 & b & \textbf  0 & 0 & 0 & 0 & 0 & \cdots\\
    \cdots & 0 & 0 & 0 & 2b & a & \textbf b & 0 & 0 & 0 & 0 & \cdots\\
    \cdots & \textbf 0 & \textbf 0 &\textbf  0 &\textbf  0 &\textbf{b} & \textbf a & 
    \textbf{2b} &\textbf 0 &\textbf 0 &\textbf 0 & \cdots\\
    \cdots & 0 & 0 & 0 & 0 & 0 & \textbf  b & 0 & 2b & 0 & 0 & \cdots\\
    \cdots & 0 & 0 & 0 & 0 & 0 &\textbf  0 & a & 0 & 2b & 0 & \cdots\\
    \cdots & 0 & 0 & 0 & 0 & 0 &\textbf  0 & 0 & a & 0 & 2b & \cdots\\
    \udots & \vdots & \vdots & \vdots & \vdots & \vdots & \vdots & 
    \vdots & \vdots & \vdots & \vdots & \ddots\\
    \end{array}
\right)
\end{equation} 
We use the bold font to indicate the 0-th row and 0-th column.
Remark that this matrix is also considered in Example 
\eqref{ex:stochastic matrix} from a different point of view. 
In~\cite{Bezuglyi_Jorgensen_Sanadhya_2022}, it was shown that for $a = b = 1$, one can model by $B(F_1)$ endowed with the left-to-right order a substitution dynamical system given by the so-called ``one step forward, two steps back'' substitution on $\mathbb{Z}$:
\begin{gather*}
-1 \mapsto -2 \;-1 \; 0; \quad 0 \mapsto -1\; 0\; 1;\\
n \mapsto (n-1)(n + 1)(n + 1) \mbox{ for } n \leq -2;\\
n \mapsto (n - 1)(n - 1)(n + 1) \mbox{ for } n \geq 1.
\end{gather*}
%We 
%observe that $A_1$ is an aperiodic, irreducible, and
%recurrent two-sided infinite matrix.
\end{example}

\begin{prop}\label{prop Ex1 sect 5}
The stationary generalized Bratteli diagram $B(F_1)$, where $A_1 = 
F_1^T$ as in \eqref{Matrix A_1}, supports a tail-invariant measure 
$\mu$ given by \eqref{eq inv meas left}. The measure $\mu$ is 
defined using the eigenvalue $\lambda = a + 2b$ and the 
corresponding right eigenvector $\xi$ of $A_1$ given by
$$
\xi = \left(\ldots, \dfrac{1}{2^3}\left(\dfrac{a}{b}\right)^2, 
\dfrac{1}{2^2}\left(\dfrac{a}{b}\right), \dfrac{1}{2},1, \mathbf{1},
\dfrac{1}{2}, \dfrac{1}{2^2}\left(\dfrac{a}{b}\right), 
\dfrac{1}{2^3}\left(\dfrac{a}{b}\right)^2, \ldots \right)^T.
$$ 
(The 0-th entry of $\xi$ is shown in bold font). 
Moreover, $\mu$ is finite if and only if $a < 2b$.
\ignore{In this case, 
the matrix $A_1$ is positive recurrent and $\mu$ is uniquely 
ergodic. If $a \geq 2b$, then the measure $\mu$ is $\sigma$-finite, 
and the matrix $A_1$ is null-recurrent.} 
\end{prop}

We prove this result in Appendix \ref{APP:Example}. 

\begin{example}\label{Ex:BobokBruin} (see also \cite{BobokBruin2016}) This example deals
with a Bratteli diagram  of different nature. This diagram is
defined by a null-recurrent matrix of period $2$ and has at least two infinite $\sigma$-finite measures which take finite positive values on cylinder sets. 

Let $a$ and $b$ be natural numbers. Consider the 
generalized stationary Bratteli diagram $B(F_2)$ where $A_2 = F_2^T$ 
is given by
\begin{equation}\label{Matrix A_2}
    A_2 = A_2(a,b)= \left(
  \begin{array}{cccccccccccc}
   \ddots &  \vdots &  \vdots & \vdots & \vdots & \vdots & \udots\\
    \cdots & 0 & {b} & 0 & 0 & 0 & \cdots\\
    \cdots &{a} & 0 & {b} & 0 & 0 & \cdots\\
    \cdots & 0 & {a} & 0 & {b} & 0 & \cdots\\
    \cdots & 0 & 0 & {a} & {0} & {b} &  \cdots\\
    \cdots & 0 & 0 & 0 & {a} & {0} &  \cdots\\
    \udots & \vdots & \vdots & \vdots & \vdots & \vdots &  \ddots\\
    \end{array}
\right)
\end{equation} 
\end{example}

\begin{prop}\label{prop Ex 2 sect 5} 
(1) If $a\neq b$,  then there exist at least two  
tail-invariant infinite $\sigma$-finite measures on the path space $X_B$,
$\mu$ and $\wt \mu$. According to \eqref{eq inv meas left}, 
the measure 
$\mu$ is determined by $\lambda = a + b$ and 
$\xi = \left(\ldots, 1,1,1 \ldots \right)^T$, and 
the measure $\wt \mu$ is determined by $\wt \lambda = 2 \sqrt{ab}$
and $\wt \xi = \left( \wt \xi_n \right)^T$ where
$\wt {\xi}_n = \left(\dfrac{a}{b}\right)^{\tfrac{n}{2}}$, $n\in \Z$.\\
(2) If $a=b$, then $\mu = \wt \mu$. \\
The matrix $A_2(a,b)$ is null recurrent for any $a,b \in \mathbb{N}$.
\end{prop}

This matrix is also considered in Example \ref{ex transient Sect 6} 
where we use the corresponding stochastic matrix to show that 
$A_2(a,b)$ is null recurrent for any $a,b \in \mathbb{N}$.
We prove Proposition \ref{prop Ex 2 sect 5} in Appendix 
\ref{APP:Example}. 
\\

\textit{Examples of one-sided infinite Bratteli diagrams}. 
In this subsection, we 
consider examples of one-sided infinite stationary generalized 
Bratteli diagrams and find conditions under which infinite 
tri-diagonal matrices (indexed by $\N$) have eigenvectors with 
finite entry sum (hence admit finite invariant measure). 

Let $A = (a_{ij})_{i,j \in \N}$ be a one-sided infinite 
tri-diagonal matrix:
$$ 
A = \begin{pmatrix}
a_{11} & a_{12} & 0 & 0 & 0 & ... \\
a_{21} & a_{22} & a_{23} & 0 & 0 & ... \\
0 & a_{32} & a_{33} & a_{34} & 0 & ... \\
0 & 0 & a_{43} & a_{44} & a_{45} & ... \\
0 & 0 & 0 & a_{54} & a_{55} & ... \\
 \vdots &  \vdots & \vdots & \vdots & \vdots &  \ddots
\end{pmatrix}  
$$ 
We assume that $A$ has an eigenvalue $\lambda$ with the right 
eigenvector $\xi$, i.e. $A \xi = \lambda \xi$. Our goal is to 
find conditions when the vector $\xi$ is summable, 
$\sum_{i =1}^{\infty} \xi_i < \infty$.  
We denote by 
$$
\sigma_i = \sum_{j = 1}^{\infty} a_{ij}
$$ 
the sum of the $i$-th row of $A$. It turns out that an important 
class of examples comes from the matrices having equal column sum 
property together with an additional requirement on the row sum 
as described in the definition below.

\begin{definition}\label{Def:Balanced} 
We say that a one-sided infinite  matrix $A = 
(a_{ij})_{i,j \in \N}$ is \textit{balanced} if it satisfies 
following conditions :

\begin{enumerate}
    \item $A$ has the property of equal column sum, i.e. $\sum_{i 
    =1}^{\infty} a_{ij} = c$ for every $j \in \N$;   this 
    automatically implies that $c = \lambda$ for every $j \in \N$. 
    
    \item $\sigma_2 = \sigma_3 = \sigma_4= ....$ and $\sigma_1 > 
    \sigma_i$ for all $i>1$. 
\end{enumerate} 
    
\end{definition} 

Below we provide some examples of one-sided infinite stationary 
Bratteli diagrams with \textit{balanced} incidence matrices. 

\begin{example}\label{ex:0ne_sided_1} 
Fix $b,c,\alpha \in \N$ such that $\alpha > 1$. Consider the 
generalized diagram $B(F_3)$ where  $A_3 = F_3^T$ is given by
\begin{equation}\label{Matrix A_3}
    A_3 = \begin{pmatrix}
b+\alpha c & \alpha c & 0 & 0 & 0 & ... \\
c & b & \alpha c & 0 & 0 & ... \\
0 & c & b & \alpha c & 0 & ... \\
0 & 0 & c & b & \alpha c & ... \\
 \vdots &  \vdots & \vdots & \vdots & \vdots &  \ddots
\end{pmatrix}
\end{equation} 
\end{example}

\begin{prop} \label{Prop:Ex_one_side_1} The stationary generalized 
Bratteli diagram $B(F_3)$, where $A_3 = F_3^T$ is defined  in 
\eqref{Matrix A_3}, supports a tail-invariant 
measure $\mu$ defined by 
\eqref{eq inv meas left} using the eigenvalue 
$\lambda = b+ c + \alpha c$ and 
the corresponding right eigenvector $\xi = (\xi_n)_{n\in \N}$ 
for $A_3$ such that
$$
\xi = \Big( \xi_1, \dfrac{\xi_1}{\alpha}, 
\dfrac{\xi_1}{\alpha^2}, 
\cdots, \dfrac{\xi_1}{\alpha^{n-1}}, \cdots \Big)^T
$$ 
where $\xi_1$ can be chosen to be any positive integer. 
    
\end{prop} 
Proposition \ref{Prop:Ex_one_side_1} is proved in Appendix 
\ref{APP:Example}.

\begin{example}\label{ex:0ne_sided_2} 
Fix $b\in \N_0$, $r \in \N$ and take integers $\alpha$ and  
$\beta$ such that  $|\alpha|, |\beta| < r$. Denote by
\be\label{eq q1q2}
q_1 = \dfrac{r-\alpha}{r + \beta}, \qquad q_2 = \dfrac{r-\beta}
{r + \alpha}.
\ee
Consider the one-sided infinite stationary Bratteli diagram 
$B(F_4)$ where $A_4 = F_4^{T}$ is given by
\begin{equation}\label{Matrix A_4}
    A_4 = F_4^T = \begin{pmatrix}
(b+r+\alpha) & r+\beta & 0 & 0 & 0 & ... \\
r-\alpha & b & r+\alpha & 0 & 0 & ... \\
0 & r-\beta & b & r + \beta & 0 & ... \\
0 & 0 & r-\beta & b & r+\beta & ... \\
 \vdots &  \vdots & \vdots & \vdots & \vdots &  \ddots
\end{pmatrix}
\end{equation}     
\end{example} 

\begin{prop} \label{Prop:Ex_one_side_2} For the
stationary Bratteli diagram $B(F_4)$ defined by the matrix given 
in \eqref{Matrix A_1}, the following statements hold.

\begin{enumerate}

\item For the eigenvalue $\lambda = b+ 2r$, the right eigenvector 
$\xi = (\xi_n)$ ($A\xi = \lambda \xi$) for $A_4$ has the entries 
\begin{equation}\label{induction}
 x_1 =1, \qquad    \xi_{2n} = \dfrac{(r-\alpha)^n\,(r-\beta)^{n-1}}
 {(r+\alpha)^{n-1}(r+\beta)^n}, \qquad 
\xi_{2n+1} = \dfrac{(r-\alpha)^n(r-\beta)^n}
{(r+\alpha)^n (r+\beta)^n};
\end{equation}

\item  If $r, \alpha, \beta$  are chosen such that $q_1q_2 < 1$
(see \eqref{eq q1q2}),
then the diagram $B(F_4)$ supports a finite tail 
invariant measure $\mu$ determined as in \eqref{eq inv meas left}. 

\item If $r, \alpha, \beta$  are chosen such that $q_1q_2 \geq 1$,
the tail-invariant measure is  $\sigma$-finite.

\end{enumerate}

\end{prop}

The proof of Proposition \ref{Prop:Ex_one_side_2} is given in 
Appendix \ref{APP:Example}.
\vskip 0.3cm

We finish this section with two examples of generalized 
Bratteli diagrams of different types: 
they are not of bounded size.

\begin{example}[Renewal subshift]\label{Ex:renewal_subshift} 
Consider the one-sided infinite generalized diagram $B(F_5)$ 
where $A_5 = F_5^{T} = (a_{ij})_{i,j \in \N}$ is defined by
\begin{equation}
    A_5 = \begin{pmatrix}
    1 & 1 & 1 & 1 & 1 & 1 & \cdots\\
    1 & 0 & 0 & 0 & 0 & 0 & \cdots\\
    0 & 1 & 0 & 0 & 0 & 0 & \cdots\\
    0 & 0 & 1 & 0 & 0 & 0 & \cdots\\
    0 & 0 & 0 & 1 & 0 & 0 & \cdots\\
    0 & 0 & 0 & 0 & 1 & 0 & \cdots\\
   \vdots & \vdots & \vdots & \vdots & \vdots & \vdots & \ddots
    \end{pmatrix}.
\end{equation}
In other words, 
\begin{equation} \label{Matrix A_5}
     a_{ij} = \left\{
\begin{aligned}
& 1 \mbox{ if } i = 1 \mbox{ or } i = j + 1,\\
& 0 \mbox{ otherwise }.
\end{aligned}
\right.
\end{equation}
We remark that the matrix $A_5$ is balanced, see Definition 
\ref{Def:Balanced}. 

\end{example}

\begin{prop} \label{prop:renew} The stationary generalized 
Bratteli diagram $B(F_5)$ corresponding to the renewal shift in 
Example \ref{Ex:renewal_subshift} supports a unique probability ergodic tail 
invariant measure $\mu$. The 
measure $\mu$ is defined by \eqref{eq inv meas left} using the Perron eigenvalue $\lambda = 2$ 
and the corresponding right eigenvector $\xi = 
\left(\frac{1}{2}, \frac{1}{2^2}, \frac{1}{2^3}, 
\ldots\right)^T$ of $A_5$. The matrix $A_5$ is positive recurrent.
\end{prop} 

The proof follows from an application of 
Theorem \ref{Thm:Unique finite} and is given 
in Appendix \ref{APP:Example}. 

Rest of the examples in this section consist of generalized Bratteli diagrams with incidence matrices that are not balanced. 

\begin{example}[Pair renewal shift, 
see \cite{Raszeja2021}]\label{Ex:Pair_Renewal_Subshift} 
Consider the one-sided infinite generalized diagram $B(F_6)$ where 
$A_6 = F_6^{T} = (a_{ij})_{i,j \in \N}$ is given by
\begin{equation} \label{Matrix A_6}
     a_{ij} = \left\{
\begin{aligned}
& 1, \mbox{ if } i = 1 \mbox{ and } j \in \N,\\
& 1, \mbox{ if } i = 2 \mbox{ and } j = 2n \mbox{ for } n \in \N,\\
& 1, \mbox{ if } i = n+1 \mbox{ and } j = n \mbox{ for } n \in \N,\\
& 0 \mbox{ otherwise}.
\end{aligned}
\right.
\end{equation}Note that $A_6$ does not have the equal column sum property 
and is not balanced.
\ignore{By Proposition 
\ref{Prop_Perron_eigenvalue_estimates}, we see that the Perron
eigenvalue $\lambda \in [2, 3]$. }
Explicitly, the matrix $A_6$ is 
\begin{equation*}
    A_6 = \begin{pmatrix}
    1 & 1 & 1 & 1 & 1 & 1 & \cdots\\
    1 & 1 & 0 & 1 & 0 & 1 & \cdots\\
    0 & 1 & 0 & 0 & 0 & 0 & \cdots\\
    0 & 0 & 1 & 0 & 0 & 0 & \cdots\\
    0 & 0 & 0 & 1 & 0 & 0 & \cdots\\
    0 & 0 & 0 & 0 & 1 & 0 & \cdots\\
   \vdots & \vdots & \vdots & \vdots & \vdots & \vdots & \ddots
    \end{pmatrix}.
\end{equation*}

 This infinite matrix is a modified version of $A_5$, and the corresponding shift space is called the pair renewal shift (see~\cite{Raszeja2021}).

\end{example}

\begin{prop} \label{prop:renew_pair} 
The stationary generalized Bratteli diagram $B(F_6)$ corresponding 
to the pair renewal shift in Example \ref{Ex:Pair_Renewal_Subshift} 
supports a unique up constant multiple finite tail-invariant measure $\mu$ 
given by \eqref{eq inv meas left}. The measure $\mu$ is defined 
by the Perron eigenvalue $\lambda =  1 + \sqrt{2}$ and the 
corresponding  right eigenvector $\xi = (\xi_n)_{n \in \N}$ of 
$A_6$ which is given by
\begin{equation*}
\xi_1 = \frac{1}{1+\sqrt{2}}, \qquad 
    \xi_n = \frac{2}{(1+\sqrt{2})^n}, \quad n \geq 2.
\end{equation*} The matrix $A_6$ is positive recurrent.
\end{prop} 

The proof of Proposition \ref{prop:renew_pair} is given in Appendix \ref{APP:Example}. 

\begin{example} Consider the one-sided infinite generalized diagram $B(F_7)$ 
where $A_7 = F_7^{T} = (a_{ij})_{i,j \in \N}$ is defined by
    \begin{equation}\label{Matrix A_7}
    A_7 = \begin{pmatrix}
    c_0 & 1 & 0 & 0 & 0 & 0 & \cdots\\
    c_1 & 0 & 1 & 0 & 0 & 0 & \cdots\\
    c_2 & 0 & 0 & 1 & 0 & 0 & \cdots\\
    c_3 & 0 & 0 & 0 & 1 & 0 & \cdots\\
    c_4 & 0 & 0 & 0 & 0 & 1 & \cdots\\
   \vdots & \vdots & \vdots & \vdots & \vdots & \vdots & \ddots
    \end{pmatrix}
\end{equation} Here $c_k \in \N$ for all $k \geq 0$.
\end{example}

\begin{prop} \label{prop:A_7} 
The stationary generalized Bratteli diagram $B(F_7)$ 
corresponding to the matrix defined in \eqref{Matrix A_7}
supports a $\sigma$-finite infinite tail-invariant measure 
$\mu$ 
given by \eqref{eq inv meas left} if there exists constant 
$C \in \N$ such that for every $k \in \N_0$, $c_k < C$. The 
measure $\mu$ is defined 
by the eigenvalue $\lambda \leq C+1$ and the 
corresponding  right eigenvector $\xi = (\xi_k)_{k \in \N}$ 
of $A_7$ which is given by
\begin{equation}\label{eq:right_A7}
    \xi_{k+1} = \lambda^{k+1} - \sum_{i=0}^k c_i 
    \lambda^{k-i}, \quad k \in \N.
\end{equation} 
Moreover, $\eta \cdot \xi < \infty$ if and only if 
$$
\sum_{k=1}^{\infty} \frac{k c_k}{\lambda^{k+1}} < 
\infty,
$$
where $\eta$ is the left eigenvector corresponding to $\lambda$.
\ignore{In this case, the measure $\mu$ is uniquely ergodic.} 
\end{prop} 

The proof of Proposition \ref{prop:A_7} is given in Appendix 
\ref{APP:Example}. 

%end \ignore

%%%%%%%

\section{Stochastic matrices in Bratteli diagrams}
\label{sect:stochastic}

In this section, we consider several stochastic matrices and 
discuss the relations between them. Also, we consider the 
properties of the corresponding generalized Bratteli diagrams. 

\subsection{Stochastic matrices and measures} 
Let $B = (V, E)$ be a stationary generalized Bratteli diagram with 
infinite incidence matrix $F$. Let $A = F^T$.
Assume that $A$ is irreducible, has a finite Perron eigenvalue $\lambda$, and that $A$ admits a positive right eigenvector $\xi$ for $\lambda$:
$$
A \xi = \lambda \xi.
$$
%As in Section 
%\ref{sect Basic}, we assume that $A$ is %irreducible, aperiodic, and 
%recurrent with finite Perron eigenvalue $\lambda$. By 
%Theorem~\ref{Thm:Generalized_Perron_Frobenius} there exist unique 
%up to constant multiple positive left and right eigenvectors 
%$\eta$, $\xi$ for $\lambda$:
%$$
%A \xi = \lambda \xi, \quad \eta A = \lambda %\eta.
%$$ 
Define the matrix $P = (p_{w,v} : w,v \in V_0)$ as follows:
\begin{equation}\label{Formula:stochastic_matrix_elements}
p_{w,v} = \frac{a_{w,v}\;\xi_v}{\lambda\xi_w}. 
\end{equation}  
Clearly, the matrix $P$ is row stochastic, that is 
$$
\sum_{v \in V} p_{w,v} = 1.
$$
Hence, $P$ can be considered as a Markov matrix that gives the probability to get from $w \in V_0$ to $v \in V_1$ along any edge from $E(v,w)$. The matrix $P$ is called also a probability transition kernel.
Denote $P^n = (p^{(n)}_{w,v})$ and $A^n = (a^{(n)}_{w,v})$. 
By induction, we have 
$$
p^{(n)}_{w,v} = \frac{a^{(n)}_{w,v}\;\xi_v}{\lambda^n\xi_w}, 
\qquad v, w \in V_0, \ n\in \N.
$$
In particular, 
\begin{equation}\label{Equation:connection_p_and_a}
p^{(n)}_{vv} = \frac{1}{\lambda^n}a^{(n)}_{v,v}.   
\end{equation}
From~(\ref{Equation:connection_p_and_a}) it easily follows that 
the spectral radius of $P$ is $1$, and the corresponding right 
eigenvector consists of all ones.

The following result was proved by Thiago Raszeja.\footnote{We 
are thankful to Thiago for the permission to include this 
statement in the paper.}
\begin{prop}
\label{prop:ThiagoThm}
    Let $P$ be the stochastic matrix defined in 
    \eqref{Formula:stochastic_matrix_elements} by a matrix $A$.
    Then $P$ is recurrent (null recurrent, positive recurrent) or
    transient if and only $A$ is recurrent (null recurrent, positive 
    recurrent) or transient. In particular, $P$ is 
    recurrent if and only if $\sum_n p^{(n)}_{w,w} = \infty$ for 
    all $w\in V_0$.
\end{prop}

\begin{example}\label{ex transient Sect 6}
Consider the matrix $A =A(a,b)$ with two positive integer parameters
$a, b$:
\begin{equation}\label{eq:matrix A(a,b)}
A =  \left(
  \begin{array}{cccccccccccc}
   \ddots &  \vdots &  \vdots & \vdots & \vdots & \vdots & 
   \udots\\
    \cdots & 0 & {b} & 0 & 0 & 0 & \cdots\\
    \cdots &{a} & 0 & {b} & 0 & 0 & \cdots\\
    \cdots & 0 & {a} & 0 & {b} & 0 & \cdots\\
    \cdots & 0 & 0 & {a} & {0} & {b} &  \cdots\\
    \cdots & 0 & 0 & 0 & {a} & {0} &  \cdots\\
    \udots & \vdots & \vdots & \vdots & \vdots & \vdots &  \ddots\\
    \end{array}
\right)
\end{equation}

\ignore{Note that $A$ is the matrix of the random walk on $\Z$ such that
$Prob(v \to (v+1)) = \frac{b}{a+b} $ and $Prob(v \to (v - 1)) = 
\frac{a}{a+b} $, $v \in \Z$.
Moreover, this matrix has the properties of equal row sums and 
equal column sums. Then $\lambda = a+b $ is an eigenvalue for 
$A$, and  $\xi = (\ldots, 1, 1, 1, \ldots)^T$ and 
$\eta = (\ldots, 1, 1, 1, \ldots)$ are the right and left 
eigenvectors for $A$ corresponding to $\lambda$, respectively.  
Denote by $P$ the stochastic matrix defined by $A$ and let 
$Q = P^2$
with entries $(q^{(n)}_{v,w})$, $w,v \in \Z$.}
As shown in \cite{BobokBruin2016}, the spectral radius (Perron eigenvalue)
of $A$ is $\lambda_A = 2\sqrt{ab}$. It can be also checked, using Proposition \ref{prop:ThiagoThm}, 
that $A$ is null recurrent for any $a, b \in \mathbb{N}$. 
\ignore{$$
\sum_{n\geq 1} q^{(n)}_{v,v} = \sum_{n\geq 1} \binom{2n}{n}
\frac{a^nb^n}{(a+b)^{2n}}.
$$
This series converges if $a\neq b$. }
\end{example}

\begin{prop}
Let $A$ be as in \eqref{eq:matrix A(a,b)} with $a\neq b$. Then 
the corresponding stationary generalized Bratteli diagram $B$ 
admits at least two  tail-invariant measures.
\end{prop}

\begin{proof} For $\lambda = a+b$, we define a $\sigma$-finite 
invariant measure $\mu$ using the eigenvector $\xi = 
(\ldots, 1,1, \ldots)$:
for $\ol e = (e_0, \ldots, e_{n-1})$, $\mu([\ol e]) = (a+b)^{-n}$.

To get another tail-invariant measure $m$, we solve the equation 
$A \tau = \lambda_A \tau$. Omitting computations, we find that
$$
\tau = \left(\left(\frac{a}{b}\right)^{i/2} : i \in \Z\right).
$$
Therefore, the measure $m$ of the cylinder set $[\ol e]$  is
$$
m([\ol e]) = \frac{1}{(2b)^n}
\left(\frac{a}{b}\right)^{\frac{i-n}{2}}
$$
where $r(\ol e) = r(e_{n-1}) =i$.

\end{proof}

%%%%%%

\begin{example}\label{ex:stochastic matrix}
[see Example~\ref{Ex:Matrix A_1}] 
One can define stochastic matrix $P$ using $\lambda = a + 2b$ as follows: 
$$
p_{0,0} = p_{-1,-1} = \frac{a}{a + 2b};\quad p_{0,-1} = p_{0,1} = 
p_{-1,-2} = p_{-1,0} = \frac{b}{a + 2b}; 
$$
for $k \geq 1$ we have
$$
p_{k, k-1} = \frac{2b}{a+2b}; \quad p_{k,k+1} = \frac{a}{a+2b};
$$
and for $k \leq -2$:
$$
p_{k, k-1} = \frac{a}{a+2b}; \quad p_{k,k+1} = \frac{2b}{a+2b}.
$$
All other entries of $P$ are zero. 

Notice that if 
$a < 2b$, the random walk on $\mathbb{Z}$ 
corresponding to $P$ is positive recurrent. Indeed, for 
$k \geq 1$, the probability to walk from $k$ to $k+1$ is less than 
the probability to walk from $k$ to $k-1$. For $k \leq -2$, the 
probability to walk from $k$ to $k-1$ is less than the probability 
to walk from $k$ to $k+1$. This means that the random walk 
approaches $\{-1,0\}$ with a higher probability than escapes 
to infinity. Since the inverse of the Perron eigenvalue $\lambda_P$ for $A_1$ is the convergence radius of the series from Defnition~\ref{Def_recurrent_transient}, we obtain that $\lambda_P$ is greater or equal to $a + 2b$. 
\end{example}

%%%%%%

\subsection{Stationary diagrams with positive recurrent incidence 
matrix} 
Let $A= F^T$ be an infinite matrix that determines a stationary
generalized 
Bratteli diagram $B$. Suppose that there exists a Perro-Frobenius 
eigenpair $(\lambda, \xi)$, $A\xi = \lambda \xi$.
Let $P = (p_{w,v})_{w,v \in V}$ be the 
stochastic matrix corresponding to the matrix $A$ as defined in 
(\ref{Formula:stochastic_matrix_elements}). Using the definition of 
a tail-invariant measure as in (\ref{eq inv meas left}) note that 
$P$ can be also defined as
\begin{equation}\label{matrix P}
    p_{w,v} = a_{w,v} \frac{\mu^{(n+1)}_v}{\mu^{(n)}_w}
  \qquad \mbox{where} \ \mu^{(n)}_v = \frac{\xi_v}{\lambda^n} .
\end{equation} 

This formula remains true for non-stationary Bratteli diagrams $B$
defined by a sequence of incidence matrices $(F_n)$ and 
the corresponding sequence of transpose matrices $A_n = F_n^T$. 
We suppose that there exists a sequence of positive vectors 
$(\mu^{(n)})$  such that $A_n \mu^{(n +1)} = \mu^{(n)}$. 
Recall that such a sequence generates a tail-invariant measure. 
In this case, we define the sequence of row stochastic matrices 
$\widetilde P_n = (\widetilde p^{(n)}_{w,v})_{w,v \in V}$:
\begin{equation}
    \widetilde p^{(n)}_{w,v} = a^{(n)}_{w,v}\, \frac{\mu^{(n+1)}_v}{\mu^{(n)}_w}. 
\end{equation} 
Observe that if $B = (V,E)$ is stationary, then $\widetilde P_n = P$ 
for each $n \in \N$. 

Another way to realize row stochastic matrices is by using 
the height vectors $H^{(n)} = (H^{(n)}_{v} : 
v \in V_n)$ as in Definition \ref{Def:Height} and
(\ref{lem vector H}). Recall that $F_n H^{(n)} = H^{(n+1)}$ for 
$n \in \N$. Thus, we can define $\widetilde F_n = 
(\widetilde f^{(n)}_{v,w})$, $w \in V_n$, $v \in V_{n+1}$ 
as follows: 
\begin{equation}
    \widetilde f^{(n)}_{v,w} = f^{(n)}_{v,w}\, \frac{H^{(n)}_w}
    {H^{(n+1)}_v} \,\, w \in V_n,\, v \in V_{n+1}. 
\end{equation} 
The sequence of matrices $\widetilde F_n$ consists of row 
stochastic 
matrices. Let the clopen set $X_v^{(n)}$ be as in Definition 
\ref{Def:Kakutani-Rokhlin}, then for any tail-invariant measure 
$\mu$ on $X_B$, we have
\begin{equation}\label{X_v}
\mu (X_v^{(n)}) = \mu_v^{(n)} H_v^{(n)} = : \widetilde 
q_{v}^{(n)}.
\end{equation} 
Setting $\widetilde q^{(n)} : = 
(\widetilde q_{v}^{(n)})_{v \in V_n}$, we observe that 
\begin{equation}\label{q_n+1}
    \widetilde q^{(n+1)} \widetilde F_n = \widetilde q^{(n)}.
\end{equation}

In what follows we will focus on the case of \textit{positive 
recurrent} incidence matrices. 

Let $B = (V,E)$ be a stationary generalized Bratteli diagram such 
that the matrix  $A = F^T$ is irreducible, aperiodic, and 
positive recurrent. Let $\xi$ and $\eta$  denote the  
right and left positive 
eigenvectors, respectively, corresponding to the Perron value 
$\lambda$. Moreover, we assume that the right eigenvectors $\xi = 
(\xi_v)_{v \in V_n}$  is summable, i.e. 
$\underset{v\in V_n}{\sum} 
\xi_v < \infty$. Without loss of generality, we assume that 
$\underset{v\in V_n}{\sum} \xi_v = 1$. 
According to Theorem \ref{Thm:inv1},
the right eigenvector $\xi$ defines a tail-invariant measure.
We discuss here the role of the left eigenvector $\eta$,
see Theorem \ref{thm:H grow}. 
 
We define a sequence of vectors $(\nu^{(n)})_{n \in \N_0}$ as 
follows: $\nu^{(0)} = \xi$, and $\nu^{(n+1)}= \nu^{(n)} P$ 
for all $n \in \N_0$, where $P$ is the row stochastic matrix as 
defined in (\ref{Formula:stochastic_matrix_elements}). Note that 
$(\nu^{(n)})$ is a sequence of probability vectors. To see this,  
we note that $\nu^{(0)} = \xi$ is a probability vector, and by 
induction,  
$$
\underset{v\in V_n}{\sum}\, \nu_v^{(n+1)} =  \underset{v\in V_n}
{\sum} \, \underset{w\in V_n}{\sum} \nu_w^{(n)} p_{w,v} = 
\underset{v\in V_n}{\sum} \, \underset{w\in V_n}{\sum} 
\nu_w^{(n)} \frac{a_{w,v}\;\xi_v}{\lambda\,\xi_w} 
$$

$$
 = \underset{w\in V_n}{\sum} \nu_w^{(n)} \frac{\lambda \, \xi_w}
 {\lambda\, \xi_w} = \underset{w\in V_n}{\sum} \nu_w^{(n)} = 1.
$$ 

\begin{lemma}\label{lem:vectors nu}
For every $v \in V_n$, we have 
$$
\nu_v^{(n)} = \dfrac{\xi_v}{\lambda^n} \underset{w\in V_0}{\sum} 
a_{w,v}^{(n)},
$$ 
where $a_{w,v}^{(n)}$ ($a_{w,v}^{(1)} = a_{w,v}$) 
denotes the $(w,v)$-th entry of the 
matrix $A^n$. 
\end{lemma}

\begin{proof} The proof is by induction. We recall that 
$\underset{w\in V_0}{\sum} \, a_{w,v} < \infty$. Next, 
$$
\nu_v^{(1)} = \underset{w\in V_0}{\sum} \nu_{w}^{(0)} P_{w,v} = 
\underset{w\in V_0}{\sum} \, \xi_w \, a_{w,v} 
\frac{\xi_v}{\lambda \, \xi_w} = \frac{\xi_v}{\lambda} \, 
\underset{w\in V_0}{\sum} \, a_{w,v}. 
$$ 
Check the induction step:
$$
\nu_v^{(n+1)} = \underset{w\in V_0}{\sum} \nu_{w}^{(n)} P_{w,v} 
= \underset{w\in V_0}{\sum} \bigg( \dfrac{\xi_w}{\lambda^n} 
\underset{s\in V_0}{\sum} a_{s,w}^{(n)} \bigg)  p_{w,v}
$$

$$
= \underset{w\in V_0}{\sum} \bigg( \dfrac{\xi_w}{\lambda^n} 
\underset{s\in V_0}{\sum} a_{s,w}^{(n)} \, \frac{a_{w,v}\;\xi_v}
{\lambda\,\xi_w} \bigg)  = \underset{s\in V_0}{\sum} \frac{\xi_v}
{\lambda^{n+1}} \underset{w\in V_0}{\sum} a_{s,w}^{(n)} 
\, a_{w,v} 
= \frac{\xi_v}{\lambda^{n+1}} \underset{s\in V_0}{\sum} 
a_{s,v}^{(n+1)}.
$$
\end{proof}

\begin{prop}\label{xi_eta} 
(1) Let $A$ be a positive recurrent matrix. Then, for every 
vertex $v \in V_0$, 
$$
\mu(X_v^{(n)}) = \nu_v^{(n)}. 
$$ 
where $\mu$ is the tail-invariant  measure 
defined in (\ref{eq inv meas left}).

(2)
If $A$ is a positive recurrent 
matrix and the right eigenvector $\xi$ is probability,  
then for every $v \in V_0$, 
$$
\nu_v^{(n)} \rightarrow \xi_v \, \eta_v \qquad n \rightarrow 
\infty.
$$ 
\end{prop}

\begin{proof} (1) Note that by (\ref{X_v}) we have 
$$
\mu (X_v^{(n)}) 
= \mu_v^{(n)} H_v^{(n)} =  \sum_{w \in V_0} a^{(n)}_{w, v} 
\,\,\mu_v^{(n)} 
=  \sum_{w \in V_0} a^{(n)}_{w, v}\,\, \frac{\xi_v}{\lambda^{n}} 
= \nu_v^{(n)}.
$$ 

(2) 
Theorem \ref{Thm:Pos_recc} states that if  
$A$ is positive recurrent and the left and right eigenvectors 
are normalized by the condition  $\eta \cdot \xi = 1$, 
then for every $v,w \in V_0$,
$$
\lim_{n \rightarrow \infty} \frac{a_{w,v}^{(n)}}{\lambda^n} = 
\xi_w \eta_v.
$$ 
It follows from Lemma \ref{lem:vectors nu} that 
$$ 
\lim_{n \rightarrow \infty} \, \nu_v^{(n)} = \lim_{n \rightarrow 
\infty}\, \dfrac{\xi_v}{\lambda^n} \underset{w\in V_0}{\sum} 
a_{w,v}^{(n)} = \lim_{n \rightarrow \infty}\, \xi_v 
\underset{w\in
V_0}{\sum} \dfrac{a_{w,v}^{(n)}}{\lambda^n} = \xi_v 
\underset{w\in V_0}{\sum} \xi_w \eta_v =   \xi_v \, \eta_v.
$$ 
We used here the fact that the vector $\xi$ is probability.
\end{proof}

\begin{theorem}\label{thm:H grow}
Let $B(F)$ be a stationary generalized Bratteli diagram and the
matrix $A = F^T$. Suppose that $A$ is irreducible, aperiodic, 
and positive recurrent. Let $\lambda$ be the Perron eigenvalue 
of $A$ and let $\xi = (\xi_i)$, $\eta = (\eta_i)$ be the 
corresponding right and left eigenvectors normalized such that 
$\sum_{v\in V_0} \xi_v = 1$ and $\eta \cdot \xi = 1$. Then, 
for every $v,w \in V_{0}$, 
$$
\dfrac{H^{(n)}_w}{H^{(n+1)}_v} \rightarrow \frac{\eta_w}{\lambda 
\cdot \eta_v}\mbox{ as } n \rightarrow \infty.
$$
\end{theorem} 

\begin{proof} Since $B$ is a stationary Bratteli diagram, we 
identify all levels $V_n$ with $\Z$.  
It follows from \eqref{X_v} and \eqref{q_n+1} 
that, for every $n\in \N$,
\begin{equation}\label{eq:F and X_v}
\mu(X_v^{(n+1)}) \widetilde F = \mu(X_v^{(n)}),
\end{equation} 
By definition of $\widetilde F$, we get that, for every 
$ w \in \Z$,  
\begin{equation} \label{eq with Z}
    \underset{v\in \Z}{\sum} \mu (X_v^{(n+1)})\, f_{v,w}\, 
    \frac{H^{(n)}_w}{H^{(n+1)}_v} = \mu (X_w^{(n)}).
\end{equation} 
Taking the limit in \eqref{eq with Z} as $n \rightarrow \infty$ 
and using Proposition \ref{xi_eta}, we obtain that
\begin{equation}
 \lim_{n \rightarrow \infty} \, \bigg(\underset{v\in Z}{\sum} 
 \mu (X_v^{(n+1)})\, f_{v,w}\, \frac{H^{(n)}_w}{H^{(n+1)}_v} 
 \bigg) = \lim_{n \rightarrow \infty} \big(\mu (X_w^{(n)})\big) = 
 \xi_w \, \eta_w,
\end{equation} 
This implies that the series is convergent:
\begin{equation}\label{double conv}
    \underset{v\in V_0}{\sum} \,  \lim_{n \rightarrow \infty} 
    \bigg( \mu (X_v^{(n+1)})\, f_{v,w}\,
    \frac{H^{(n)}_w}{H^{(n+1)}_v}\bigg) = \xi_w \, \eta_w < 
    \infty, \qquad w\in \Z.
\end{equation} 
Hence, the limit 
$$ 
\lim_{n \rightarrow \infty} \bigg( \mu (X_v^{(n+1)})\,
f_{v,w}\, \frac{H^{(n)}_w}{H^{(n+1)}_v}\bigg)
$$ 
exists. Proposition \ref{xi_eta} states that
 the limit 
$\lim_{n \rightarrow \infty} \mu (X_v^{(n+1)})$ exists for 
every $v$. We conclude therefore that 
$$ 
L_{v,w} : =  \lim_{n \rightarrow \infty} \, \dfrac{H^{(n)}_w}
{H^{(n+1)}_v} < \infty.
$$ 
The proved facts allow us to rewrite (\ref{double conv}) 
as follows: 
$$ 
\underset{v\in V_0}{\sum} \,  \xi_v \, \eta_v \, f_{v,w} 
\, L_{v,w} = \xi_w \, \eta_w.
$$ 
Multiplying in the above relation both sides by 
$\dfrac{\lambda}{\eta_w}$, we get 
$$ 
\underset{v\in V_0}{\sum} \,  \xi_v \, f_{v,w}\,
\bigg( L_{v,w} \, \dfrac{\lambda  \eta_v}{\eta_w}\bigg)  
= \lambda  \xi_w.
$$ 
Since   
$$ 
\underset{v\in V_0}{\sum} \,  \xi_v \, f_{v,w} = \lambda 
\xi_w,
$$
we obtain that $L_{v,w} \, \dfrac{\lambda 
\eta_v}{\eta_w} = 1$. 
In other words, 
$$ 
L_{v,w} =  \lim_{n \rightarrow \infty} \, 
\dfrac{H^{(n)}_w}{H^{(n+1)}_v} = \frac{\eta_w}{\lambda  
\eta_v},
$$ 
as needed. 
\end{proof}

\subsection{Stochastic matrices and measures for non-stationary generalized Bratteli diagrams}
Suppose that a generalized Bratteli diagram is defined by the sequence of
incidence matrices $(F_n)$ and $A_n = F_n^T$. Assume that there exists
a probability tail-invariant measure $\mu$. According to Theorem
\ref{BKMS_measures=invlimits}, this measure is completely determined by the 
sequence of non-negative vectors $(\mu^{(n)})$ such that 
$A_n \mu^{(n+1)} = \mu^{(n)}$. Simultaneously, we have the sequence
$(H^{(n)})$ which satisfies the condition $F_n H^{(n)} = H^{(n+1)}$.
Since $\mu$ is a probability measure, the sequences $(H^{(n)})$ and $(\mu^{(n)})$ satisfy the equality 
\be\label{eq inn prod is 1}
\langle \mu^{(n)}, H^{(n)} \rangle :=  \sum_{v\in V_n} \mu_v^{(n)} H_v^{(n)} = 1, \quad n \in \N_0.
\ee

Let 
$$
|| \mu^{(n)} ||_\infty = \sup_{v \in V_n} \mu^{(n)}_v, \quad n \in \N_0.
$$
Denote 
$$
\wh\mu^{(n)} = \frac{\mu^{(n)}}{|| \mu^{(n)} ||_\infty}, \quad 
\wh H^{(n)} = \frac{H^{(n)}}{\langle \wh\mu^{(n)}, H^{(n)} \rangle }.
$$

It can be checked directly that $\langle \wh\mu^{(n)}, \wh H^{(n)} \rangle 
=1$. 

\begin{lemma}\label{Lemma:lambda_n}
Let 
$$
\lambda_n = \frac{|| \mu^{(n)} ||_\infty}{|| \mu^{(n+1)} ||_\infty}.
$$
Then, for  $n \in \N_0$,
\begin{enumerate}
    \item $\lambda_n > 1$,
    \smallskip
    
    \item $A_n \wh \mu^{(n+1)} = \lambda_n \wh \mu^{(n)}$,

      \smallskip
    \item $F_n \wh H^{(n)} = \lambda_n \wh H^{(n+1)}$.
\end{enumerate}
\end{lemma}

\begin{proof}
We have for all $n \in \N_0$
$$
A_n  \wh \mu^{(n+1)} = \frac{1}{|| \mu^{(n+1)} ||_\infty} A_n 
\mu^{(n+1)} = \frac{\mu^{(n)}}{|| \mu^{(n+1)} ||_\infty}  = 
\frac{|| \mu^{(n)}||_\infty }{||\mu^{(n+1)} ||_\infty} \wh\mu^{(n)} 
= \lambda_n \wh \mu^{(n)}.
$$

To see that $\lambda_n > 1$, we take $\varepsilon >0$ and find $v_0
\in V_{n+1}$ such that 
$$
||\mu^{(n+1)} ||_\infty - \varepsilon < \mu_{v_0}^{(n+1)}. 
$$
Let $w$ be a vertex in $V_n$ such that $E(w, v_0) \neq \emptyset$. 
Since $|s^{-1}(w)| > 1$, we see that $\mu_{v_0}^{(n+1)} < 
\mu_{w}^{(n)}  \leq ||\mu^{(n)} ||_\infty $. Using the fact that the set
$r^{-1}(v_0)$ is finite, we conclude that 
$||\mu^{(n+1)} ||_\infty < ||\mu^{(n)} ||_\infty$ as desired. 

For the third relation, we compute using \eqref{eq inn prod is 1}
$$
\ba 
F_n \wh H^{(n)} = &\ \frac{H^{(n+1)}}{\langle \wh\mu^{(n)}, H^{(n)}
 \rangle }\\
 = & \ \frac{|| \mu^{(n)}||_\infty}{\langle \mu^{(n)}, H^{(n)} \rangle }
 H^{(n+1)}\\
 = & \ \frac{|| \mu^{(n)}||_\infty}{\langle \mu^{(n+1)}, H^{(n+1)} \rangle }
 H^{(n+1)}\\
 = & \ \frac{|| \mu^{(n)}||_\infty}{|| \mu^{(n+1)}||_\infty}
 \frac{1}{\langle \wh \mu^{(n+1)}, H^{(n+1)} \rangle}  H^{(n+1)}\\
 = & \ \lambda_n \wh H^{(n+1)}.
\ea
$$
\end{proof}

We summarize the above discussion in the following theorem.

\begin{theorem}\label{thm seq wh mu}
Let $B$ be a generalized Bratteli diagram and $\mu$ a tail-invariant
 measure on the path space $X_B$. Then there exist two sequences of
  positive vectors 
$(\wh \mu^{(n)})$ 
and  $(\wh H^{(n)})$  such that for all $n\in \N_0$
$$
\langle \wh \mu^{(n)}, \wh H^{(n)} \rangle = 1,
$$
and 
$$
A_n \wh \mu^{(n+1)} = \lambda_n \wh \mu^{(n)},  \ \ \  
A^T_n \wh H^{(n)} = \lambda_n \wh H^{(n+1)}.
$$

\end{theorem}

\begin{remark}
Theorem \ref{thm seq wh mu} remains true if instead of $H^{(0)} = (1, 1, 
.... )$ one takes an arbitrary 
sequence $t^{(0)}$ of positive integers $t^{(0)}_v, v \in V_0$,
then the sequences $t^{(n)}$ are determined automatically by 
the relation $F_{n-1}\cdots F_1F_0 t^{(0)}, n \in \N$. 
\end{remark}

Theorem \ref{thm seq wh mu} can be used to construct a sequence
$(\wh P_n)$ of row stochastic matrices. 

\begin{lemma}\label{lem matrix wh P} 
Let $D_n$ be the diagonal matrix
whose non-zero entries are $\wh \mu^{(n)}_v, v \in V_n$. 
Then the matrix 
\be\label{eq def wh P_n}
\wh P_n = \frac{1}{\lambda_n}D_n^{-1} A_n D_{n+1}, \quad n \in N_0, 
\ee
is stochastic. 
\end{lemma}

\begin{proof}
Indeed,  for $w \in V_n, v \in V_{n+1}$, one has
$$
\ba 
\sum_{v \in V_{n+1}} \wh p^{(n)}_{w, v} = & \ \sum_{v \in V_{n+1}} 
\frac{1}{\lambda_n \wh \mu^{(n)}_w} a^{(n)}_{w,v} \wh \mu^{(n+1)}_v\\
= &\ \frac{1}{\lambda_n \wh \mu^{(n)}_w} \sum_{v \in V_{n+1}} 
a^{(n)}_{w,v} \wh \mu^{(n+1)}_v\\
=& \ \frac{1}{\lambda_n \wh \mu^{(n)}_w} \lambda_n \wh \mu^{(n)}_w\\
= & \ 1.
\ea
$$
\end{proof}

We note that the entries of $\wh P_n$ can be written in two ways:
$$
\wh p^{(n)}_{w, v} = \frac{a^{(n)}_{w,v} \wh \mu^{(n+1)}_v}
{\lambda_n \wh \mu^{(n)}_w} 
 = \frac{a^{(n)}_{w,v} \mu_v^{(n+1)}}{\mu^{(n)}_w}.
$$
It can be easily seen that, for a stationary generalized Bratteli diagram, the stochastic matrices $\wh P_n$ coincide with 
the matrix $P$ defined in \eqref{Formula:stochastic_matrix_elements}.

Similarly to the case of finite Bratteli diagrams, we can produce a sequence
of stochastic incidence matrices for any generalized Bratteli diagram $B$ 
with incidence matrices $(F_n)$. For given $B$ and $(F_n)$, compute the
sequence of vectors $(H^{(n)})$ as in Lemma \ref{lem vector H}. Then 
 define entries of a new matrix $G_n =(g_{v,w}^{(n)})$ as follows:
\be\label{eq entries of G_n}
g_{v,w}^{(n)} := \frac{f^{(n)}_{v,w} H^{(n)}_w}{H^{(n+1)}_v},\quad 
v\in V_{n+1}, w\in V_n, n \in \N.
\ee

\begin{lemma} \label{lem stoch m-x C_n}
Let $\mu$ be a tail-invariant measure on a generalized 
Bratteli diagram $B = (V,E)$. 
Then the matrix $G_n$, with entries defined by \eqref{eq entries of G_n},
is stochastic and satisfies the relation
$$
C_n s^{(n+1)} = s^{(n)}, \quad n \in \N,
$$
where $C_n = G_n^T$ and $s^{(n)}$ is the probability vector with entries 
$(\mu(X^{(n)}_v) : v \in V_n)$. 
\end{lemma}

\begin{proof}
The fact that $G_n$ is stochastic follows from the relation $F_n H^{(n)} 
= H^{(n+1)}$. 

To check the other statement of the lemma, we first recall that 
$s^{(n)}_v = \mu^{(n)}_vH^{(n)}_v$ and then compute
$$
\ba
(C_n s^{(n+1)})_v = \ & \sum_{w\in V_{n+1}} g^{(n)}_{w,v}s^{(n+1)}_w 
\\
=  \ & \sum_{w\in V_{n+1}} f^{(n)}_{w, v} \frac{H^{(n)}_v}{H^{(n+1)}_w}
 \mu^{(n+1)}_w
H^{(n+1)}_w\\
=\ & H^{(n)}_v \sum_{w\in V_{n+1}} f^{(n)}_{w, v}  \mu^{(n+1)}_w\\
=\ & H^{(n)}_v \mu^{(n)}_v\\
=\ & s^{(n)}_v.
\ea
$$
We used here relation \eqref{eq:formula_p_n} of Theorem 
\ref{BKMS_measures=invlimits}.
\end{proof}

The main result of this subsection is as follows:

\begin{theorem}\label{Thm:seq_wh_mu_main}
    Let $B$ be a generalized Bratteli diagram with incidence matrices $F_n$ and $\mu$ be a probability tail-invariant measure on $B$. Let $A_n = F_n^T$. Then there exist a sequence of numbers $\lambda_n > 1$, a sequence of vectors $\wh\mu^{(n)} = (\wh\mu_v^{(n)})$, $v \in V_n$, and a sequence of vectors $\wh H^{(n)} = (\wh H_v^{(n)})$, $v \in V_n$, 
such that for every $n \in \mathbb{N}$

\begin{enumerate}
    \item $\langle\wh \mu^{(n)}, \wh H^{(n)}\rangle = 1$,
    \item $A_n \wh\mu^{(n+1)} = \lambda_n \wh\mu^{(n)}$,
    \item $F_n \wh H^{(n)} = \lambda_n \wh H^{(n+1)}$.
\end{enumerate}

Conversely, if there exist a sequence of non-negative infinite integer matrices $A_n$ with finite column sums, a sequence of numbers $\lambda_n > 1$, a sequence of non-negative infinite vectors $\wh \mu^{(n)} = (\wh \mu_v^{(n)})$, and a sequence of positive infinite vectors $\wh H^{(n)} = (\wh H_v^{(n)})$ with $\wh H^{(0)} = H^{(0)}$ which satisfy conditions (1)-(3) above, then a generalized Bratteli diagram $B$ defined by incidence matrices $F_n = A_n^T$ possesses a probability tail-invariant measure, for which the measures of cylinder sets are defined by the sequence of vectors
\begin{equation}\label{eq:formula_for_p^n}
p^{(0)}= \wh \mu^{(0)} \mbox{ and } p^{(n)}=  \frac{1}{\lambda_0 \cdots \lambda_{n-1}}\wh \mu^{(n)}, \mbox{ for } n \geq 1.
\end{equation}
where $p^{(n)} = (\mu(X_w^{(n)}(\ov e))) : w\in V_n) $.

\end{theorem}
\begin{proof}
The ``if'' part of the theorem is proved in Lemma~\ref{Lemma:lambda_n} above. To prove the ``only if'' part, we first notice that for all $n \geq 0$:
$$
A_n p^{(n+1)} = \frac{1}{\lambda_0 \cdots \lambda_{n}} A_n \wh \mu^{(n+1)} = \frac{1}{\lambda_0 \cdots \lambda_{n}} \lambda_n \wh \mu^{(n)} = p^{(n)}. 
$$
Hence by Theorem~\ref{BKMS_measures=invlimits} the measure $\mu$ defined by \eqref{eq:formula_for_p^n} is a tail-invariant measure on $B$. Since $\wh H^{(0)} = H^{0}$ and $\langle\wh \mu^{(0)}, \wh H^{(0)}\rangle = 1$, we obtain that $\mu$ is a probability measure.
\end{proof}

\section{Open problems}\label{Sect:OP}
This section contains several open problems. We have not 
tried to create a comprehensive list of problems that would cover all possible directions. We focus here on the existence of Vershik maps 
and probability tail-invariant measures. These areas are 
well-studied in the case of standard Bratteli diagrams.
 We refer to the literature mentioned in Section \ref{intro}. It would be interesting to understand which of these results (or their 
 analogs) can be proved 
 in the context of generalized Bratteli diagrams. 

\begin{enumerate}
    
    \item Find  necessary and sufficient conditions for an aperiodic irreducible infinite non-negative integer matrix with finite Perron eigenvalue to have a right eigenvector 
    $\xi = (\xi_v)$ with $\sum_{v \in V_0} \xi_v < \infty$. 
    It follows from our results given in Section 
    \ref{Section:stat_GBD} that 
    such conditions will imply the existence of a finite 
    tail-invariant measure for the corresponding stationary 
    generalized Bratteli which takes positive values on cylinder sets. We present some sufficient conditions of this kind in Subsections~\ref{subsec:finite_tail-inv_measure}, \ref{Subsec:Examples}.

    \item Find conditions on incidence matrices of generalized Bratteli diagrams which allow determining the number of ergodic tail-invariant measures. In particular, it 
    is important to know when a generalized Bratteli diagram is uniquely ergodic. This problem was discussed in many papers on Cantor dynamics, see e.g. \cite{Durand2010}, 
    \cite{Putnam2018}, \cite{BezuglyiKarpel2016}, 
    \cite{BezuglyiKarpelKwiatkowski2019} for the case of standard Bratteli diagrams. As a part of this problem, it would be 
    interesting to consider the cases of null-recurrent and/or transient incidence matrices. In Example~\ref{ex transient Sect 6}, we present two different tail-invariant measures. 
    If a generalized Bratteli diagram is not uniquely ergodic, how can one determine the support of ergodic measures? 

    \item Most results of this paper are related to the case of irreducible Bratteli diagrams. How can we describe 
    tail-invariant measures on reducible (stationary) generalized Bratteli diagrams? We mention Proposition~\ref{Prop:no_meas_fin_cyl_sets} to illustrate 
    what may happen in this case. For the 
standard Bratteli diagrams, we refer to  \cite{BezuglyiKwiatkowskiMedynetsSolomyak2010} where the method
of measure extension played an important role. This remark 
motivates the following problem: Find conditions under which a measure on a generalized Bratteli diagram  is an extension of a measure from a subdiagram (see~\cite{AdamskaBezuglyiKarpelKwiatkowski2017}).

     \item Let $B$ be a stationary generalized Bratteli diagram such that the corresponding incidence matrix is aperiodic irreducible and does not have  a finite Perron eigenvalue 
     (we give examples of such matrices in Appendix~\ref{APP:Perron-Frobenius_Theory}). Can such a diagram possess a finite tail-invariant measure? In particular, can a stationary generalized Bratteli diagram with the incidence matrix 
$$
F = \begin{pmatrix}
2 & 1 & 0 & 0 & \ldots\\
1 & 3 & 1 & 0 & \ldots\\
0 & 1 & 4 & 1 & \ldots\\
0 & 0 & 1 & 5 & \ldots\\
\vdots & \vdots & \vdots & \vdots & \ddots\\
\end{pmatrix}
$$
possess a probability tail-invariant measure which takes positive values on cylinder sets?

    \item If $B$ is a generalized Bratteli diagram, then one 
    can consider various partial orders on the set of edges. 
For standard Bratteli diagrams, we know that some orders 
generate continuous Vershik maps. On the other hand, there are 
diagrams on which it is impossible to define a Vershik map,
see details in \cite{Medynets_2006}, 
\cite{BezuglyiKwiatkowskiYassawi2014}.
Is it true that any generalized Bratteli diagram $B$ can be endowed with an order which generates a Borel Vershik map?  In particular, is there an order without infinite maximal and minimal paths? Are there algebraic conditions on incidence matrices that guarantee the existence of a Vershik map? 
The reader can find more information in  \cite{BezuglyiKwiatkowskiYassawi2014},
\cite{BezuglyiYassawi2017} for standard Bratteli diagrams.

    \item For an ordered generalized Bratteli diagram find conditions, under which the corresponding Vershik map is a homeomorphism (we have partly answered this question in Subsection~\ref{sub:Left_right}).
\end{enumerate}

%%% Edit and add the text below

\vspace{5mm}

\textbf{Acknowledgments}. 
The authors are pleased to thank our colleagues and collaborators, 
especially, J. Bobok, H. Bruin, R. Curto, J. Kwiatkowski, P. Muhly, 
and W. Polyzou for valuable and stimulating discussions. We are 
grateful to  T. Raszeja for computing the eigenvalue and 
eigenvector problem in Example~\ref{Ex:Pair_Renewal_Subshift} and checking the recurrence properties in Example~\ref{Ex:renewal_subshift}, ~\ref{Ex:Pair_Renewal_Subshift} and 
proving 
Proposition~\ref{prop:ThiagoThm}. S.S. is thankful to colleagues at 
the Ben Gurion University of the Negev for their encouragement. 
S.B. and O.K. are also grateful to the Nicolas Copernicus 
University in Torun for its hospitality and support. S.B. 
acknowledges the hospitality of AGH University during his visit to 
Krakow. O.K. is supported by the NCN (National Science Center, Poland) Grant 2019/35/D/ST1/01375 and by the program ``Excellence initiative - research university'' for the AGH University of Science and Technology. 

%%%%%% APPendix

\appendix
\section{Perron-Frobenius theory for infinite 
matrices}\label{APP:Perron-Frobenius_Theory} 

For the benefit of the readers, in this appendix, we provide some 
definitions and results from the Perron-Frobenius theory of 
infinite matrices 
which are of direct relevance to the proofs in the body of our 
paper. As mentioned before, these results are due to 
Vere-Jones (see \cite{VereJones_1962} 
\cite{VereJones_1967} \cite{VereJones_1968}). The formulations of 
statements and definitions in this section are taken from the 
book by Kitchens \cite[Chapter 7]{Kitchens1998}. 
%%%%

 Recall that a matrix $A = (a_{ij})$ is called \textit{infinite} 
(or countably infinite) if its rows and columns
are indexed by the same countably infinite set. Let $A = 
(a_{ij})_{i,j \in \Z}$ be a real, non-negative, infinite 
matrix. We enumerate the rows and columns of $A$ by $\Z$ to
make this material closer to two-sided generalized Bratteli 
diagrams. Note that the results in this section are independent of 
the fact that we enumerate the rows and columns by $\Z$ or by $\N$. 
As before, we denote by $a_{ij}^{(n)}$ the $(i,j)$-th 
entry of $A^n$, $n \in \N$, whenever it exists, i.e., whenever 
it is finite. The matrix $A$ is called 
\textit{irreducible} if for every pair $i,j \in \mathbb{Z}$ there 
is $n > 0$ such that $a^{(n)}_{ij} > 0$. Denote $$
p(i) = \gcd\{n:a^{(n)}_{ii} > 0\}, \ \ \ i \in \mathbb{Z}.
$$
Then $p(i)$ is called the \textit{period of index} $i$. For an 
irreducible matrix $A$, the periods of all indices are the same 
and called the \textit{period of $A$}. An irreducible matrix with 
period one is called \textit{aperiodic}.

\begin{lemma}\cite{Kitchens1998} \label{Lemma_Perron_value}
Let $A$ be a real, non-negative, irreducible, aperiodic, infinite 
matrix. Fix $i \in \mathbb{Z}$. Then:
\begin{enumerate}[label=(\alph*)]
    \item there exists $k \in \mathbb{N}$ such that $a_{ii}^{(n)} 
    > 0$ for all $n \geq k$,
    \item 
\be\label{eq:spectral radius}
\lambda = \lim\limits_{n \rightarrow \infty} \sqrt[n]{a^{(n)}_{ii}}
= \sup\limits_{n\in \mathbb{N}} \sqrt[n]{a^{(n)}_{ii}} \leq \infty,
\ee
and the value of $\lambda$ does not depend on $i$.
\end{enumerate} 
\end{lemma} 

If the value $\lambda$ in Lemma~\ref{Lemma_Perron_value} is 
finite, then it is called the \textit{Perron eigenvalue} of $A$. 
If $A$ is a finite non-negative irreducible and aperiodic matrix, 
then $\lambda$ coincides with the usual Perron eigenvalue of $A$. If the matrix $A$ is periodic, then one defines the Perron 
eingenvalue of $A$ as 
$$
\lambda = \limsup\limits_{n \rightarrow \infty} \sqrt[n]{a^{(n)}_{ii}}.
$$

The next example gives a banded matrix $A$ with infinite spectral 
radius given by formula \eqref{eq:spectral radius}. 

\begin{example}\label{example_infinte_Perron_value} 
Let $A = (a_{ij})$ be a non-negative integer $\mathbb{Z} \times 
\mathbb{Z}$ matrix with the entries
$$
\left\{
\begin{aligned}
&a_{0,0} = 1,\\
&a_{m,m} = |m|^{|m|}, & \mbox{ for } m \in \mathbb{Z}\setminus
\{0\},\\
&a_{m, m - 1} = a_{m, m + 1} = 1, & \mbox{ for } m \in 
\mathbb{Z},\\
&a_{i,j} = 0 \mbox{ for } |i - j| > 1, \; & i,j \in \mathbb{Z}.
\end{aligned}
\right.
$$
We prove straightforwardly that the Perron value $\lambda$ of $A$ 
is infinite. Fix $m \in \mathbb{N}$. First we show by induction 
that $a^{(n)}_{m,m+n} \geq 1$ for all $n \in \mathbb{N}$. To see 
this, observe that $a_{k,k+1} = 1 \geq 1$ for all $k \in \N$. By 
induction step, 
$$
a^{(n+1)}_{m, m+n+1} \geq a^{(n)}_{m,m+n}\,\,a_{m+n, m+n+1} \geq 1.
$$
Similarly, $a^{(n)}_{m+n,m} \geq 1$ for all $n \in \mathbb{N}$. 
Then 
$$
a^{(n+1)}_{m,m+n} \geq a^{(n)}_{m,m+n}a_{m+n, m+n} \geq 
(m+n)^{(m+n)}.
$$ 
%(we use hypothesis I here.)
It follows that
$$
a^{(2n+1)}_{mm} \geq a^{(n+1)}_{m,m+n}a^{(n)}_{m+n,m} \geq 
(m+n)^{(m+n)}.
$$
Therefore
$$
\sqrt[2n+1]{a^{(2n+1)}_{mm}} \geq ((m+n)^{(m+n)})^{\tfrac{1}{2n+1}} 
\rightarrow \infty
$$
as $n \rightarrow \infty$. Thus, the Perron value $\lambda$ is 
infinite.

\end{example}

\begin{definition}\label{Def_recurrent_transient}
Let $A$ be a real, non-negative, irreducible, aperiodic, infinite 
matrix with a finite Perron value $\lambda$ and $i \in \mathbb{Z}$. 
\begin{enumerate}[label=(\roman*)]
\item $A$ is called \textit{recurrent} if
$$
\sum\limits_{n = 0}^{\infty} \frac{a^{(n)}_{ii}}{\lambda^n} = 
\infty,
$$
where $A^0 = I$ is an infinite identity matrix.

\medskip
\item $A$ is called \textit{transient} if
$$
\sum\limits_{n = 0}^{\infty} \frac{a^{(n)}_{ii}}{\lambda^n} < 
\infty.
$$ 
Moreover, the convergence of the series does not depend on the 
choice of $i$. 
\end{enumerate}
\end{definition}
For a real, non-negative, irreducible, aperiodic infinite matrix 
$A$, we define the following generating function:
$$
T^A_{w,v}(z) = \sum_{i = 0}^{\infty} a_{w,v}^{(i)} z^i,
$$ 
where $a_{w,v}^{(0)} = \delta_{w,v}$. The radius of convergence 
of $T^A_{w,v}(z)$ is $\lambda^{-1}$. Hence, $A$ is recurrent if 
and only if
$$
T^A_{w,w}(\lambda^{-1}) = \infty,
$$
and $A$ is transient if and only if
$$
T^A_{w,w}(\lambda^{-1}) < \infty.
$$

\begin{thm}[Generalized Perron-Frobenius theorem, 
\cite{Kitchens1998}] \label{Thm:Generalized_Perron_Frobenius}
Let $A$ be a real, non-negative, irreducible, aperiodic, recurrent, 
infinite matrix. Let
$\lambda < \infty$ be a Perron eigenvalue of $A$. Then 
\begin{enumerate}[label=(\roman*)]
\item  there exist strictly positive eigenvectors $\eta,\xi$ such 
that $\eta A = \lambda \eta$, $A \xi = \lambda \xi$;

\item $\eta$ and $\xi$ are unique up to constant multiples.
\end{enumerate}
We will call $\xi$ and $\eta$ the left and the right \textit{Perron 
eigenvectors} of $A$.
\end{thm}

\begin{example}
There are banded matrices that do not have finite Perron values.
Here is an example of such a matrix:
$$
A = \begin{pmatrix}
n_1 & 1 & 0 & 0 & \ldots\\
1 & n_2 & 1 & 0 & \ldots\\
0 & 1 & n_3 & 1 & \ldots\\
0 & 0 & 1 & n_4 & \ldots\\
\vdots & \vdots & \vdots & \vdots & \ddots\\
\end{pmatrix}.
$$
It is easy to see that if the sequence $(n_i)$ is unbounded, 
then the solution of 
$A x = \lambda x$ does not exist for positive vectors $x$.
\end{example}

We set $l_{w,v}(0) = 0$ and for a real, non-negative, irreducible, 
aperiodic infinite matrix $A$ define $l_{w,v}(1) = a_{w,v}$ and 
$$
l_{w,v}(n+1) = \sum_{i \neq w} l_{w,i}(n) a_{i,v}.
$$
Then $l_{w,v}(n)$ is the number of paths of length $n$ from vertex 
$w$ to vertex $v$ which do not return to $w$ at any time prior to 
$n$. Define the corresponding generating function
$$
L_{w,v}(z) = \sum_{n = 1}^{\infty}l_{w,v}(n)z^n.
$$ The matrix $A$ is called \textit{positive recurrent} if 
$$
\sum_{n = 1}^{\infty} n l_{w,w}^{(n)}(\lambda^{-n})  < \infty
$$
and $A$ is called \textit{null recurrent} if 
$$
\sum_{n = 1}^{\infty} n l_{w,w}^{(n)}(\lambda^{-n}) = \infty. 
$$

\begin{prop}\cite{Kitchens1998}\label{Prop:Pos_rec}
Let $A$ be a real, non-negative, irreducible, aperiodic, 
recurrent, infinite matrix. Let $\eta$ and $\xi$ be the left and 
right Perron eigenvectors of $A$. If
$$
\eta \cdot \xi = \sum_i \eta_i\xi_i < \infty,
$$  
then $A$ is positive recurrent. Otherwise, $A$ is null recurrent.
\end{prop}

\begin{thm}\cite{Kitchens1998}\label{Thm:Pos_recc} Let $A$ be a real, non-negative, 
irreducible, aperiodic, recurrent, infinite matrix, and let $\eta$ 
and $\xi$ be the left and right Perron eigenvectors of $A$.
If $A$ is positive recurrent and $\eta \cdot \xi = 1$ then
$$
\lim_{n \rightarrow \infty} \frac{A^n}{\lambda^n} = \xi \eta.
$$
\end{thm}

The next fact explains why we were interested in measures 
taking positive values on cylinder sets.

\begin{prop} 
Let $A$ be a real, non-negative, irreducible, 
and aperiodic infinite matrix such that there exists $\lambda > 0$ 
and a non-zero non-negative vector $x = (x_i)_{i \in \Z}$ for 
which
\begin{equation}\label{Formula_Ax<lambdax}
Ax \leq \lambda x.
\end{equation}
Then if $x_j > 0$ for some $j \in \Z$, it follows that $x_i > 0$ 
for all $i \in \Z$.
\end{prop}

\begin{proof}
By Formula~\ref{Formula_Ax<lambdax}, we have
$$
\sum_j a_{ij}x_j \leq \lambda x_i\,\, \mathrm{for}\,\,
\mathrm{ all}\,\, i \in \Z.
$$ 
Thus, we have
$$
a_{ij}x_j \leq \lambda x_i
$$
for all $i,j$. We show by induction that
$$
a_{i_0 i_1} \cdots a_{i_{n-1}i_n} x_{i_n} \leq \lambda^n x_{i_0}.
$$
Indeed, suppose the above inequality is true for some $n$, we 
prove it for $n+1$:
$$
a_{i_0 i_1} \cdots a_{i_{n-1}i_n}(a_{i_{n}i_{n+1}} x_{i_{n+1}}) 
\leq a_{i_0 i_1} \cdots a_{i_{n-1}i_n}(\lambda x_{i_n}) \leq 
\lambda (\lambda^n x_{i_0}) = \lambda^{n+1} x_{i_0}.
$$
Let $x_j > 0$, fix any $x_i$. Since $A$ is irreducible, there 
exists $n$ such that
$$
a^{(n)}_{ij} > 0,
$$
where $A^n = (a^{(n)}_{ij})$. Thus, there exist $i_1, \ldots, 
i_{n-1}$ such that
$$
a_{i i_1}\cdots a_{i_{n-1}j} > 0.
$$
Hence,
$$
\lambda^n x_i \geq a_{i i_1}\cdots a_{i_{n-1}j}x_j > 0.
$$
\end{proof}

The next proposition gives bounds for a positive eigenvalue of an infinite matrix which has a non-negative corresponding eigenvector.
\begin{prop}\label{Prop_Perron_eigenvalue_estimates}
Let $A$ be a real, non-negative, irreducible, and aperiodic 
infinite matrix with the uniformly bounded sum of elements in 
each column. 
Suppose there exists $0 < \lambda < \infty$ and a non-zero 
non-negative vector $x = (x_i)_{i \in \Z}$ for which $Ax = 
\lambda x$. If 
$$
\sum_{i} x_i < \infty
$$
then
$$
\inf_{j}\sum_{i} a_{ij} \leq \lambda \leq \sup_{j} \sum_i a_{ij}. 
$$
In particular, if the sum of entries in every column of
$A$ is $c$,  then $\lambda = c$.
\end{prop}

\begin{proof} We have
$$
\sum_{j} a_{ij} x_j= \lambda x_i
$$
and
$$
\sum_i \sum_{j} a_{ij} x_j = \lambda \sum_i x_i.
$$
Let 
$$
m = \inf_{j}\sum_{i} a_{ij}, \quad M = \sup_{j}\sum_{i} a_{ij}.
$$
Then 
$$
m \sum_{j} x_j \leq \lambda \sum_i x_i \leq M \sum_{j} x_j
$$
and
$$
m \leq \lambda \leq M.
$$
In particular, if $m = M = c$ then $\lambda = c$.
\end{proof}

%%%%%

\section{Examples of Bratteli diagrams supporting invariant 
measures}\label{APP:Example}

In this appendix, we give the proofs of the propositions formulated
in Subsection \ref{Subsec:Examples} for some classes of stationary 
generalized Bratteli diagrams. Most of them satisfy the conditions 
of Theorem \ref{Thm:inv1}. This means that, for such diagrams,
there exist tail-invariant measures 
(finite or $\sigma$-finite) given by (\ref{eq inv meas left}).

We will follow the notations of Appendix 
\ref{APP:Perron-Frobenius_Theory}. Recall that  
$A = (a_{ij})_{i,j \in \Z}$ is an integer-valued, non-negative, and
countably infinite matrix.
\\

\noindent 
\begin{proof} (\textit{Proposition \ref{prop Ex1 sect 5}})
We first note that $A_1$ has the equal column sum property. 
Hence, there is an eigenvalue equal to the column sum $\lambda = a + 2b$. Denote by 
$\xi = (\xi_n)_{n \in \Z}$ the right eigenvector for $A_1$, 
i.e., $A_1 \xi = (a +2b) \xi$.
Then $b\xi_{-1} + a \xi_0 + 2b \xi_1 = (a + 2b) \xi_0$. 
We choose $\xi_{-1} = \xi_0 = 1$, and find that $\xi_1 = 
\frac{1}{2}$. Analogously, the equality
$$
b\xi_{0} + 2b \xi_2 = b + 2b \xi_2 =(a + 2b)\xi_1 =\frac{a}{2} + b
$$ 
gives $\xi_2 = \frac{a}{4b}$. By induction,  
we prove that 
$$\xi_n = \frac{a^{n-1}}{2^n b^{n-1}}
$$ 
for each $n \in \N$. Indeed, we have  
$$
a\xi_{n-1} + 2b\xi_{n+1} = \frac{a^{n-1}}{2^{n-1}b^{n-2}} + 
2b\xi_{n+1} = (a + 2b)\xi_{n} = \frac{a^{n} + 
2a^{n-1}b}{2^nb^{n-1}}.
$$ 
Thus, 
$$
\xi_{n+1} = \frac{a^{n} + 2a^{n-1}b - 2a^{n-1}b}{2^{n+1}b^{n}} = 
\frac{a^{n}}{2^{n+1}b^{n}}. 
$$
To find  $\xi_{-n}$, we use the same equations. Therefore, 
$$
\xi = \left(\ldots, \dfrac{1}{2^3}\left(\dfrac{a}{b}\right)^2, 
\dfrac{1}{2^2}\left(\dfrac{a}{b}\right), \dfrac{1}{2},1,1,
\dfrac{1}{2}, \dfrac{1}{2^2}\left(\dfrac{a}{b}\right), 
\dfrac{1}{2^3}\left(\dfrac{a}{b}\right)^2, \ldots \right)^T
$$
is the right eigenvector for $A_1$. Clearly, the tail 
invariant measure $\mu$, which is determined by 
\eqref{eq inv meas left}, is finite if and only if 
$$
\sum_{n \in \mathbb{Z}} \xi_n < \infty \Longleftrightarrow a < 2b.
$$

Since $\eta = (\cdots, 1,1, 1, \cdots)$ is the left eigenvector
for $A_1$, the condition $a < 2b$ is equivalent to the 
property $\eta \cdot \xi < \infty$.
\ignore{of positive recurrence of $A_1$. In this case, we can 
apply Theorem \ref{Thm:Unique finite} to conclude the uniqueness 
of $\mu$. If $a \geq 2b$ then it follows from the discussion in 
Section 
\ref{sect:stochastic}, Example \ref{ex:stochastic matrix} that 
$A_1$ is null recurrent. }
\end{proof}

\noindent 
\begin{proof}  (\textit{Proposition \ref{prop Ex 2 sect 5}})
(1) For $a \neq b$, calculations show that $A_2$ has eigenvalues 
$\lambda = a+b$ and $\wt {\lambda} = 2\sqrt{ab}$. 
We note that $\wt {\lambda}$ is the spectral radius of the matrix 
$A_2$ (see also 
Example \ref{ex transient Sect 6}). The eigenvectors 
corresponding to $\lambda$ and $\wt \lambda$ are given below:
$$
A_2 \,\xi = \lambda \, \xi, \quad \lambda = a + b, \quad
\xi = (\xi_n) = \left(\ldots, 1,1,1 \ldots \right)^T, \quad 
\sum_{n  \in \mathbb{Z}} \xi_n = \infty
$$ 
and 
$$
A_2 \,\wt{\xi} = \wt{\lambda} \, \wt{\xi}, \quad \wt{\lambda} = 
2\sqrt{ab}, \quad \wt{\xi} = (\wt{\xi}_n), \quad \wt{\xi}_n = 
\left(\frac{a}{b}\right)^{\tfrac{n}{2}}, \quad \sum_{n \in 
\mathbb{Z}} \wt{\xi}_n = \infty.
$$ 
We apply \eqref{eq inv meas left} to obtain two $\sigma$-finite 
invariant measure $\mu$ and $\wt \mu$ on the path space $X_B$, 
corresponding to $\lambda$ and $\wt \lambda$ respectively. 

(2) When $a = b$, we obtain the  eigenvalue $\lambda = \wt \lambda = 2a$ and and 
$\xi = \wt \xi = \left(\ldots, 1,1,1 \ldots \right)^T$ hence 
$\mu = \wt \mu$. 

Direct computations or application of Proposition~\ref{prop:ThiagoThm} show that matrix $A$ is null recurrent.
\end{proof}

\noindent 
\begin{proof} (\textit{Proposition \ref{Prop:Ex_one_side_1}}) 
Note that the tri-diagonal matrix $A_3$ is a balanced matrix (see 
Definition \ref{Def:Balanced}) with $\lambda = b+ c + \alpha c$,  
$\sigma_1 = b + 2 \alpha c$ and $\sigma_i = b + c + \alpha c$ for 
$i > 1$. Let $\xi = (\xi_n)_{n \in \N}$ be the corresponding right  
eigenvector. Then the first entry of the  
equation 
$$
A_3 \xi = (b+c+\alpha c) \xi
$$ 
implies that $(b+\alpha c)\xi_1 + \alpha c \, \xi_2 = 
(b+c+\alpha c) \xi_1$. Hence, we get 
$\xi_2 = \dfrac{\xi_1}{\alpha}$. 

Then, we apply the relation 
$$ 
c\, \xi_{n-1} + b \, \xi_n + \alpha c \, \xi_{n+1} = 
(b+c+\alpha c) \xi_n
$$
for $n > 2$, and by induction, we can easily obtain that 
$\xi_n = \dfrac{\xi_1}{\alpha^{n-1}}$. Since $\alpha > 1$, 
$\sum_{n \in \mathbb{Z}} \xi_n < \infty$. 
Finally, we observe that $\eta A = \lambda \eta$ where 
$\eta = (\cdots,1,1,1,\cdots)$. This means that $\eta \cdot \xi <
\infty$.
\ignore{and $A_3$ is positively recurrent. Therefore the tail
invariant measure $\mu$ defined in Theorem \ref{Thm:inv1} is
finite and uniquely ergodic by Theorem \ref{Thm:Unique finite}. }
\end{proof}

\noindent 
\begin{proof} (\textit{Proposition \ref{Prop:Ex_one_side_2}}) 
(1) Note that the tri-diagonal matrix $A_4$ given by 
\eqref{Matrix A_4} 
is a balanced matrix (see Definition \ref{Def:Balanced}) with the 
eigenvalue $\lambda = b+ 2r$, $\sigma_1 = b + 2 r + \alpha + 
\beta$ and $\sigma_i = b + 2r $ for $i > 1$. We show that the 
right eigenvector $\xi$ corresponding to $\lambda$ has entries as 
in \eqref{induction}. To see this, we set $\xi_1 =1$ and 
find from the equation  $(b + r +\alpha) x_1 + (r+\beta)x_2 = 
(b+2t) x_1$ that $x_2 = \dfrac{r - \alpha}{r + \beta}$.
Fix $n >0$ and  assume that $\xi_i$ satisfies equations in 
(\ref{induction}) for all $i\leq 2n-1$. We note that, using the
equality 
$$
(r-\beta) \xi_{2n-2} = (r+ \alpha) \xi_{2n-1},
$$
one can find $\xi_{2n}$ from the relation 
 $$
(r-\beta) \xi_{2n-2} + b\, \xi_{2n-1} + (r+\beta) \xi_{2n} = 
(b+2 r) \xi_{2n-1}.
$$ 
Then, 
$$
(r+ \alpha) \xi_{2n-1} + (r+\beta) \xi_{2n} = 2 r\, \xi_{2n-1}
$$
and 
$$
\xi_{2n} = \dfrac{(r-\alpha)\,\xi_{2n-1}}{(r+\beta)}= 
\dfrac{(r-\alpha)^n\,(r-\beta)^{n-1}}{(r+\alpha)^{n-1}(r+\beta)^n}
$$
A similar calculation gives the formula for $\xi_{2n+1}$:
$$
\xi_{2n+1} = \dfrac{(r-\alpha)^n(r-\beta)^n}
{(r+\alpha)^n (r+\beta)^n}.
$$

(2) Now we find the conditions under which the eigenvector 
$\xi$ is summable, i.e., $\sum_{i \in \N} \xi_i < \infty$. 
Recall that the parameters $q_1$ and $q_2$ of the matrix $A_4$ 
have been defined in \eqref{eq q1q2}. 

Suppose that $q_1q_2 < 1$. Then we can write 
$\xi_{2n+1} = (q_1\,q_2)^{n}$, and 
$\xi_{2(n+1)} = q_1\,(q_1\,q_2)^{n}$. This means that  
the sum of all entries of $\xi$ can be found as follows:
$$
\ba
\sum_{i=0}^{\infty} \xi_i =&  \sum_{n=0}^{\infty} (q_1\,q_2)^{n} +  
\sum_{n=0}^{\infty} q_1\,(q_1\,q_2)^{n} \\
=  &\dfrac{1}{(1-q_1q_2)} + \dfrac{q_1}{(1-q_1q_2)}\\
= & \dfrac{2\,r+\beta-\alpha}{(r+\beta)} < \infty. 
\ea
$$

It is obvious that $\eta = (\cdots,1,1,1,\cdots)$ is 
the left eigenvector corresponding to $\lambda$ and we have
$\eta\cdot\xi < \infty$. 
\ignore{the matrix $A_4$ is positive recurrent. 
We now apply Theorem \ref{Thm:Unique finite} to conclude that 
the corresponding tail-invariant measure is uniquely ergodic. }

(3) If $q_1 q_2 \geq 1$, then the series 
$\sum_{i=0}^{\infty} \xi_i$ diverges and the tail-invariant measure 
$\mu$ given by \eqref{eq inv meas left} is $\sigma$-finite. 
\end{proof}

\begin{proof} (\textit{Proposition \ref{prop:renew}})
It is obvious that the matrix  $A_5$ has the equal column
sum property and the eigenvalue $\lambda = 2$. One can check by definition that $\lambda$ is the Perron eigenvalue and the matrix $A_5$ is positive recurrent.
Then the left 
and right eigenvectors are 
$$
\eta = (1, 1, 1, \ldots)
$$
and
$$
\xi = \left(1, \frac{1}{2}, \frac{1}{2^2}, \frac{1}{2^3}, 
\ldots\right)^T.
$$
Since the matrix $A_5$ is positive recurrent, it follows from
Theorem~\ref{Thm:Unique finite} that
the Bratteli diagram
supports the unique (up to constant multiple)
finite ergodic invariant measure which takes positive values on cylinder sets, and this measure is generated by $\xi$ and $\lambda$. It is easy to see that any  probability tail-invariant measure on $B(F_5)$ should have positive values on cylinder sets. It follows from the fact that for every non-negative $n$ there is an edge between the first vertex on level $n$ and every vertex on level $n+1$. Thus, the probability ergodic tail-invariant measure for $B(F_5)$ is unique.
\end{proof}
 
\begin{proof} (\textit{Proposition \ref{prop:renew_pair}}) 
By definition of the matrix $A_6$, the entries of $A_6$ are
\begin{equation*}
    a_{1,n}=a_{2,2n}=a_{n+1,n} = 1,\quad \forall n\in \mathbb{N}
\end{equation*}
and zero otherwise. 
Hence, the equation $A_6 \xi = \lambda\xi$ is equivalent to
the system
\begin{equation}\label{eq:system_pair_renewal}
\begin{cases}
    \sum_{k=1}^\infty \xi_k = \lambda \xi_1;\\
    \xi_1 + \sum_{k=1}^\infty \xi_{2k} = \lambda \xi_2;\\
    \xi_n = \lambda \xi_{n+1}, \qquad n\geq 2.
\end{cases}
\end{equation}
Direct calculations show that the Perron eigenvalue 
$\lambda$ is $1 + \sqrt{2}$, and the corresponding probability 
right eigenvector $\xi = (\xi_n)_{n \in \N}$ is given by
\begin{equation*}
\xi_1 = \frac{1}{1+\sqrt{2}}; \quad 
    \xi_n = \frac{2}{(1+\sqrt{2})^n}, \quad n \geq 2,
\end{equation*} 
and the matrix $A_6$ is recurrent (see~\cite{Raszeja2021}).
The left eigenvector is $\eta = (1, \lambda - 1, 1, \lambda - 1, \ldots)$. Thus, the matrix $A_6$ is positive recurrent, and we 
can again apply Theorem \ref{Thm:Unique finite} to conclude that the invariant measure given by \eqref{eq inv meas left} is uniquely ergodic.

\ignore{There is a unique solution of \eqref{eq:system_pair_renewal} 
$$
\lambda = 1 + \sqrt{2}.
$$
The corresponding probability right eigenvector for $A$ has the 
entries
\begin{equation*}
\xi_1 = \frac{1}{1+\sqrt{2}}, \qquad 
    \xi_n = \frac{2}{(1+\sqrt{2})^n}, \quad n \geq 2.
\end{equation*}
Similarly, we can find the left eigenvector $\eta = (\eta_i)$ for 
$A_6$:
$$
\eta = (1, \lambda - 1, 1, \lambda - 1, \ldots),
$$
Thus, we see that $\eta \cdot \xi < \infty$ and the matrix $A_6$ 
is positive recurrent.}
\end{proof}

\begin{proof} (\textit{Proposition \ref{prop:A_7}})  Let 
$\eta$ and $\xi$ be left and right eigenvectors 
corresponding to eigenvalue $\lambda$. Then 
$\eta A_7 = \lambda \eta$ implies 
$$
\sum_{k=0}^{\infty} c_k \eta_k = \lambda\,\eta_0, \quad 
\eta_0 = \lambda\, \eta_1 ,\cdots \cdots,\eta_k = \lambda\, 
\eta_{k+1},\cdots \cdots
$$ 
Setting $\eta_0 = 1$, we get $\eta = (1, \frac{1}
{\lambda},\cdots,\frac{1}{\lambda^k},\cdots) $ and 
\be\label{eq A7}
\sum_{k=0}^{\infty} \dfrac{c_k}{\lambda^{k+1}} = 1.
\ee
If $c_k \leq  C$ for all $k \in \N_0$, then we deduce from
\eqref{eq A7} that $1 \leq \frac{C}{\lambda} \cdot 
\frac{1}{1 - 1/\lambda}$ or  $\lambda \leq C+1$. In 
particular, $\lambda = C+1 $ if all $c_k = C$.

Now we calculate the right eigenvector $\xi = 
(\xi_k)_{k=0}^{\infty}$. Observe that $A_7 \xi = 
\lambda \xi$ implies the following relations
$$
c_o\,\xi_0 + \xi_1 = \lambda\, \xi_0,\quad \cdots \cdots
\quad,c_k \,\xi_0 + \xi_{k+1} =\lambda \, \xi_k,\,\cdots \,.
$$ 
Hence, setting $\xi_0 =1$, we have 
$$ 
\xi_{k+1} =  \lambda^{k+1} - c_0 \, \lambda^k - c_1 
\lambda^{k-1} - \cdots - c_{k-1} \lambda - c_k, \quad
k \in \N_0,
$$ 
which proves \eqref{eq:right_A7}. We check
that $\xi_k > 0$ for all $k \in \N_0$. Indeed, 
$$
\xi_{k+1} = \lambda^{k+1} - \lambda^{k+1} \bigg(\dfrac{c_0}
{\lambda} + \dfrac{c_1}{\lambda^2} + \cdots + \dfrac{c_k}
{\lambda^k}\bigg) >  \lambda^{k+1} \bigg(1 - 
\sum_{i=0}^{\infty} \dfrac{c_i}{\lambda^{i+)}}\bigg) = 0
$$ 
as follows from \eqref{eq A7}. 
The eigenvector $\xi = (\xi_k)$ and $\lambda$ define the 
tail-invariant measure $\mu$ according to Theorem 
\ref{Thm:inv1}. Clearly, this measure is infinite. 

\ignore{By Proposition \ref{Prop:Pos_rec}, $A_7$ is positive 
recurrent if and only if} 
We can also find conditions under which
$\xi \cdot \eta < \infty$. 
We calculate 
$$
\ba 
\sum_{k=1}^{\infty} \xi_k \, \eta_k = & \sum_{k=1}^{\infty} 
\dfrac{1}{\lambda^k} \bigg(\lambda^k - \sum_{j=0}^{k-1} c_j 
\lambda^{k-j-1} \bigg)\\
= & \sum_{k=1}^{\infty} \bigg( 1 - \sum_{j=0}^{k-1} 
\dfrac{c_j}{\lambda^{j+1}}\bigg)\\
= & \sum_{k=1}^{\infty} \,\, \sum_{j\geq 1} \dfrac{c_j}{\lambda^{j+1}}\\
= &\sum_{k=1}^{\infty} \dfrac{k \cdot c_k}{\lambda^{k+1}}.
\ea 
$$ 
This proves the second statement of Proposition 
\ref{prop:A_7}. 
\ignore{Observe that if there exists constant $C \in \N$ such that 
for every $k \in \N_0$, $c_k < C$ then $A_7$ is positive 
recurrent and the corresponding $\mu$ is uniquely ergodic.} 
\end{proof}

\ignore{
\begin{thm}
Let $B$ be a generalized Bratteli diagram with the matrix $A = F^T$ 
of the form 
$$
A =\left(
  \begin{array}{ccccccccc}
   \ddots & \vdots & \vdots & \cdots & \vdots & \vdots & \cdots & \vdots & \udots \\
    \cdots & A_1 & 0 & \cdots & 0 & Y_{1,s+1} & \cdots & Y_{1,m} & \cdots \\
    \cdots & 0 & A_2 & \cdots & 0 & Y_{2,s+1} & \cdots & Y_{2,m} & \cdots\\
    \cdots & \vdots & \vdots & \ddots & \vdots & \vdots & \cdots& \vdots & \cdots\\
    \cdots & 0 & 0 & \cdots & A_s & Y_{s,s+1} & \cdots & Y_{s,m} & \cdots\\
    \cdots & 0 & 0 & \cdots & 0 & A_{s+1} & \cdots & Y_{s+1,m} & \cdots\\
    \udots & \vdots & \vdots & \cdots & \vdots & \vdots & \ddots & \vdots & \ddots\\
    \end{array}
\right)
$$
where all $A_i$ are finite primitive matrices. Then every distinguished eigenvalue corresponds to a finite ergodic invariant measure.
\end{thm}
}

\bibliographystyle{alpha}
\bibliography{referencesBKK5}

\end{document}